\documentclass[mnsc,nonblindrev]{informs4}
\OneAndAHalfSpacedXI

\usepackage{amsmath,amssymb,amsfonts}

\usepackage{natbib}
 \bibpunct[, ]{(}{)}{,}{a}{}{,}%
 \def\bibfont{\small}%

\usepackage{rotating}
\usepackage{fancyvrb}
\usepackage{enumerate}
\usepackage[shortlabels]{enumitem}
\usepackage{tabularx,booktabs,mathrsfs,setspace}%
\usepackage{mathtools}
\usepackage{multirow}
\usepackage{float}
\usepackage{pgfplots}
\pgfplotsset{compat=1.9}
\usepackage{csvsimple}
\usepackage{makecell}
\usepackage{multirow}
\usepackage{adjustbox}
\usepackage[caption=false, position=top]{subfig}
\usepackage{alphalph}

\usepackage{mathtools}
\usepackage{comment}
\usepackage{lscape}
\usepackage{longtable}
\usepackage{afterpage}

\usepackage{bbm}
\def\N{\mathbb{N}}

\def\R{\mathbb{R}}
\def\Z{\mathbb{Z}}
\def \E {\mathbb{E}}
\def \P {\mathbb{P}}

\def \I {\mathcal{I}}

\def \-> {\rightarrow}
\def \ind {\mathbf{1}}

\def \Var {\mbox{\bf Var}}

\def \LT {\Tilde{L}}
\def \cT {\Tilde{c}}

\def \ellT {\Tilde{\ell}}

\def \YT {\Tilde{Y}}

\def \xiT {\Tilde{\xi}}

\def \nuH {\hat{\nu}}

\def \thetaH {\hat{\theta}}

\def \yb {\mathbf{y}}
\def \hT {\Tilde{h}}
\def \ellT {\Tilde{\ell}}

\def \O {\mathcal{O}}
\def \Tcal {\mathcal{T}}

\def \Ib {\mathbf{I}}

\def \J {\mathcal{J}}
\def \xf {x_{\text{flat}}}

\def \F {\mathcal{F}}
\def \Ob {\mathbf{O}}
\def \xio {\xi^{\tiny{\mbox{O}}}}
\def \xia {\xi^{\tiny{\mbox{A}}}}
\def \xip {\xi^{\tiny{\mbox{P}}}}
\def \xipklast {\xi^{\tiny{\mbox{P}}}_{k,\tiny{\mbox{last}}}}
\def \xiptrunc {\xi^{\tiny \mbox{P,trunc}}}

\def \xiI {\xi^{\tiny{\mbox{IB}}}}
\def \xil {\xi^{\tiny{\mbox{L}}}}
\def \reg {\mbox{REG}}

\TheoremsNumberedThrough     % Preferred (Theorem 1, Lemma 1, Theorem 2)
\ECRepeatTheorems
\JOURNAL{Management Science}

\EquationsNumberedThrough    % Default: (1), (2), ...
\begin{document}
%\setcounter{page}{1}
%\DOI{10.1287/xxxx.0000.0000}%

\RUNAUTHOR{Moradi et al.}

\RUNTITLE{Risk or Replace: Efficient Asymptotics for Data-Driven Maintenance}

\TITLE{Risk or Replace: Efficient Asymptotics for Data-Driven Maintenance}

\ARTICLEAUTHORS{%
\AUTHOR{Poulad Moradi}
\AFF{Department of Industrial Engineering and Innovation Sciences, Eindhoven University of Technology, Eindhoven, the Netherlands, PO BOX 513, 5600 MB, \EMAIL{p.moradi.shahmansouri@tue.nl}}
\AUTHOR{Melvin Drent}
\AFF{Department of Information Systems and Operations Management, Tilburg University, the Netherlands, PO Box 90153, 5000 LE, \EMAIL{m.drent@tilburguniversity.edu}}
\AUTHOR{Joachim Arts}
\AFF{Luxembourg Centre for Logistics and Supply Chain Management, University of Luxembourg, Luxembourg City, Luxembourg, 6, rue Richard Coudenhove-Kalergi L-1359, \EMAIL{joachim.arts@uni.lu}}
\AUTHOR{Collin Drent}
\AFF{Department of Industrial Engineering and Innovation Sciences, Eindhoven University of Technology, Eindhoven, the Netherlands, PO BOX 513, 5600 MB, \EMAIL{c.drent@tue.nl}}
%,\textsuperscript{a} Second Author,\textsuperscript{b} Third Author,\textsuperscript{c} Fourth Author,\textsuperscript{c}

% \AFF{INFORMS, 5521 Research Park Drive, Suite 200, Catonsville, Maryland 21228 \EMAIL{mirko.janc@informs.org}}
% %\textsuperscript{b}School of Industrial Engineering, Good College, Collegeville, Maine 01234 \EMAIL{secauth@goodcoll.edu}; 
% %\textsuperscript{c}Their Common Affiliation \EMAIL{thauth@anywhere.edu, fourauth@anywhere.edu}

% %mirko.janc@informs.org
% \AUTHOR{Second Author}

% \AFF{School of Industrial Engineering, Good College, Collegeville, Maine 01234, \EMAIL{secauth@goodcoll.edu}}

% \AUTHOR{Third Author, Fourth Author}

% \AFF{Their Common Affiliation \{thauth@anywhere.edu, fourauth@anywhere.edu\}}
}

\ABSTRACT{%
Condition-based maintenance (CBM) is an approach that plans interventions for deteriorating systems according to their observed operational state. CBM reduces unplanned downtime and extends usable lifetime. We study a heterogeneous population of components that degrade over time according to a stochastic processes with non-negative and i.i.d. increments that are characterized by component-specific parameters that remain unobservable to the decision maker. We rely on degradation data to estimate these parameters and determine replacement actions at equidistant epochs. The goal is to minimize the long-run average cost, which incorporates fixed replacement costs, failure costs, and operating costs. This problem can be formulated as a high-dimensional partially observable Markov decision process (POMDP), which is generally intractable. We develop a tractable, data-driven CBM policy that estimates the optimal policy of a hypothetical Oracle that has full information of the underlying degradation parameters and call this policy the Estimated Oracle's Optimal Policy (EOP). We introduce a scaling regime where both the failure thresholds and cost parameters increase proportionally, reflecting practical settings in which component lifetimes and maintenance costs are large relative to  the time between two consecutive CBM decision moments. We show that the regret of the EOP, defined as the difference between its long-run average cost and that of the Oracle, converges to zero in the scaling regime when the parameter estimator is consistent. Across extensive experiments using both real and simulated data, the EOP achieves very low regret and, whenever the optimal POMDP policy can be computed exactly, a negligible optimality gap.
}


\maketitle
% \begin{comment}
\section{Introduction}
\label{section:introduction}
Maintenance strategies are essential to sustain the reliable and efficient operation of critical systems that degrade over time.
Failure of critical systems can cause safety hazards and substantial economic losses.
The Guardian \citep{Guardian2025} reports that equipment malfunctions in the UK’s National Health Service (NHS) have led to almost 100 deaths and harm to nearly 4,000 individuals, between 2022 and 2025. Some of these incidents are linked to critical device failure in high-risk areas such as neonatal wards and emergency care.
They also estimate that the costs of addressing NHS maintenance problems have nearly tripled over the past decade, rising from ``\pounds 4.5 billion in 2012–2013 to \pounds 13.8 billion in 2023–2024".
In the industrial sector, unplanned downtime also imposes enormous costs.
\cite{Siemens2024} find that the 500 globally largest companies lose 11\% of their revenue on unplanned downtime, which amounts to \$1.4 trillion per year, more than the 2024 GDP of the Netherlands \citep{WBGData}.
The financial impact varies sharply across industries, from roughly \$36,000 per hour in Fast-Moving Consumer Goods to \$2.3 million per hour in the automotive sector.
Maintenance decisions affect not only the reliability and economic performance of industrial systems, but also their environmental footprint.
Components that can still function are often replaced unnecessarily.
This replacement creates extra demand for raw materials, energy, and manufacturing, and leads to avoidable waste and emissions.
The transportation of new equipment and technicians further contributes to the environmental impact.
With lifetime extension and resource efficiency recognized among the core sustainability strategies by international and national organizations \citep[see e.g.,][]{UNEP2018,IEA2019,NIST2023}, maintenance decisions informed by system condition monitoring offer a practical approach to reduce the frequency of failure while reducing the footprint of early replacements.

Condition-based maintenance (CBM) is a prominent maintenance strategy that enables organizations to track the condition of their systems and perform maintenance interventions precisely when necessary.
The advancement of sensor technologies has resulted in an increasing adoption of CBM policies to maintain critical equipment.
These technologies also provide more condition monitoring data, often in real-time, which enables and motivates the use of more complex yet better-performing data-driven maintenance decision models.
\cite{Siemens2024} estimate that full adoption of CBM by Fortune 500 companies could increase productivity by 5\% and reduce maintenance costs by 40\%, saving more than \$600 billion per year.

The growing availability of condition monitoring data also transforms CBM strategies and allows organizations to make better use of the information.
A traditional, yet simplistic, approach relies on the data to model the degradation process of an operating system and then makes CBM decisions under the assumption that all systems degrade identically \citep{ARTS2024,DEJONGE2020}.
This assumption has often been necessary because degradation data were traditionally available only during downtime.
However, recent developments in automatic real-time monitoring \citep[see e.g.][]{GEHealthcare2010,Radiologybusiness2010} allow decision makers to account for differences in degradation behaviors across individual systems.
With such data, decision makers can integrate learning directly into the CBM problem and make data-driven maintenance decisions that are optimally tailored to each system's condition and degradation trajectory.
Although this approach yields better decisions, it also introduces substantial computational complexity, making the CBM problem tractable only under restrictive modeling assumptions \citep{Drent2023,Chen2015,Elwany2011}.
Even in such simplified settings, implementation and computation of optimal policies remains a challenge for practitioners.
We therefore develop efficient, data-driven CBM policies that are computationally tractable, achieve high performance in realistic operating environments, and offer provable performance guarantees.

\textbf{Brief Model Explanation.}
In this paper, we study a critical component subject to stochastic degradation over time that must be replaced at the appropriate time to minimize long-run maintenance and operating costs.
We allow both continuous-time and discrete-time degradation models, while interventions (decisions and replacements) are restricted to evenly spaced time epochs, typically coinciding with planned site visits or shutdowns.
We consider a general class of degradation models where the total degradation sustained between any two consecutive epochs are non-negative, independent, and identically distributed, and satisfy mild technical conditions.
These classes contain many stochastic shock models with cumulative damage, such as compound Poisson processes, as well as stochastic wear processes, such as the gamma processes.
All operating components follow the same type of stochastic degradation process, but the parameters of these processes differ across components, modeling heterogeneity in the population.
For instance, when degradation is modeled as a compound Poisson shock process, both the shock arrival rate and the parameters of the compounding (damage) distribution vary from one component to another.
Furthermore, the degradation model parameters for an operating component are not directly observable and can only be inferred from real-time sensor measurements. For example, in turbomachinery, measurements of vibration, pressure, and temperature can be used to infer the underlying degradation process of a blade.
Similarly, in medical imaging, changes in electrical resistance provide information to estimate the degradation process of X-ray tube filaments.
We refer to the collection of such data as the degradation data, which contains all available relevant information about the degradation behavior of a component.
This degradation data is critical to learn about the (degradation) parameters of any individual component. 
The variation in parameters from one component to the next is modeled by a prior distribution with hyperparameters.
This assumption is common in practice as the prior distribution models variation in component characteristics that the manufacturer may provide.
%It is worth noting that our methodology can be extended to settings where components follow different types of stochastic processes.
% However, assuming a single process family better reflects practical applications, where parametric approaches are typically used to model degradation.
% This assumption also allows us to present our arguments more clearly.
A component is considered to have failed once its degradation level exceeds a failure threshold.
In our model, planned preventive interventions incur a fixed cost, whereas a component failure incurs both this cost and an additional fixed failure cost.
The model also accounts for a cumulative operating cost that reflects the system’s performance deterioration as the critical component degrades.
Specifically, the operating cost of a new component is zero, and as the component’s condition worsens, higher costs accumulate in each period.
An example of operating cost is the revenue loss resulting from reduced production due to component aging.
Another example is the increased energy expenditure resulting from the deterioration of system performance.
Such costs vary with the system state and must be accounted for when determining optimal maintenance decisions.
At the beginning of each period, the decision maker uses the available degradation data to decide whether to continue operating the component or replace it with a new one.
Our objective is to determine maintenance policies that minimize the long-run average cost-rate, conditional on the degradation data observed by the decision maker.

A similar problem, under the discounted total cost criterion and for certain degradation models, has been formulated as a multi-dimensional POMDP and solved to optimality \citep[see e.g.][]{Drent2023}.
However, POMDP optimization models become immediately intractable for many degradation models when dealing with (i) large finite state spaces, (ii) multiple unknown parameters, (iii) denumerable and continuous state spaces, due to the curse of dimensionality.
In this paper, we propose a tractable approach that can handle both discrete and continuous state spaces, accommodate a broad range of degradation models, and support high-dimensional data.
This approach uses the available data to estimate the actions of a hypothetical “\textit{Oracle}” who has full knowledge of the true parameters of the operating component’s degradation process.
We first characterize the Oracle’s optimal maintenance policy as a parameter-dependent replacement threshold.
Since computing these Oracle’s optimal thresholds poses significant challenges, we employ a renewal-theoretic asymptotic approximation instead, which yields efficient data-driven replacement policies that use parameter estimates to approximate this threshold.
For several degradation models of practical interest, these replacement policies are not only computationally tractable but also easy to implement: the data-driven replacement thresholds reduce to low-dimensional calculations and, in some cases, admit closed-form expressions that can be implemented directly in spreadsheets.
Then we introduce a scaling regime where both the failure threshold and the cost parameters increase proportionally, so that the Oracle’s optimal cost-rate converges to a strictly positive and finite limit.
In many real systems, components operate over long periods relative to the time between two consecutive CBM decision moments, and the cost of maintenance represents a notable portion of the system's running costs, conditions that motivate our proposed scaling regime.
This regime also characterizes settings where POMDP optimization faces the greatest computational challenges, as the state space becomes large.

The contributions of this paper are summarized as follows:
\begin{enumerate}
    \item We propose simple and efficient data-driven CBM policies capable of handling a wide range of degradation processes and parameter learning mechanisms.
    From a computational perspective, our approach separates parameter learning from the determination of the replacement policy. 
    % Moreover, Oracle replacement policies can be computed very efficiently for both discrete and continuous state spaces.
    %\item Our policy is highly implementable in practice. For several degradation models of practical interest, the data-driven replacement thresholds reduce to low-dimensional calculations and, in some cases, admit closed-form expressions that can be implemented directly in spreadsheets.
    \item We show that the regret of these efficient data-driven replacement policies, that is the difference between the cost-rates of our data-driven policies and the Oracle's optimal cost-rate, converges to zero in our scaling regime, when parameters are estimated using a consistent estimator.
To the best of our knowledge, this is the first strong convergence result for a data-driven CBM policy under general degradation and learning processes.
    \item We evaluate our policy against the state-of-the-art POMDP-based Integrated Bayes policy \citep{Drent2023} using real degradation data and find that our policy consistently achieves superior cost-rate performance.
    \item In an extensive simulation study, we test our data-driven policies in both discrete and continuous state spaces and observe that, in both cases, the regret of using our approach relative to the Oracle’s optimal cost-rate is small, particularly when failure thresholds are high.
    We could not statistically distinguish the cost rates of our data-driven CBM policy from those of the Integrated Bayes policies when the latter are optimal.
    \item Our policy is easily interpretable by practitioners as it relies on three intuitive steps: (i) The replacement threshold is expressed by specifying the right distance from the failure threshold, i.e., we specify a safety margin around the failure threshold such that replacement occurs when the degradation crosses this safety margin;
    (ii) we determine what the optimal safety margins are for an Oracle that can observe the degradation parameters of each component; this can be done on a spreadsheet and often in closed form; (iii) the decision maker who cannot observe the degradation parameters uses a consistent estimator of the parameter based on the available information and decides to replace or not based on the safety margin that corresponds to the current estimate of the parameter.
\end{enumerate}

The remainder of the paper is organized as follows.
Section \ref{section:literature} presents a brief review of the related literature.
We introduce our model, outline our arguments, and state the main theoretical result in Section \ref{section:Model}.
Sections \ref{section:OracleOptimal} and \ref{section:AOARP} develop our proposed adaptive policies and rigorously present the main results, which are proved through an asymptotic analysis in Section \ref{section:AsymptoticAnalysis}.
In Section \ref{section:casestudyresults}, we compare the performance of our data-driven policy with the Integrated Bayes policy using real degradation datasets, and Section \ref{section:Simulation} benchmarks our approach against the Oracle's optimal policy through extensive simulations with discrete and continuous state spaces.
Finally, we provide our concluding remarks in Section \ref{section:Conclusion}.

We finish this section with a few technical definitions.
A random variable $X$ is said to be {\em lattice} if and only if there exists a $d>0$ and such that 
$\sum_{n\in\Z}\P(X=nd)=1$ and {\em non-lattice} otherwise. For a lattice random variable, the largest $d>0$ for which $\sum_{n\in\Z}\P(X=nd)=1$ is called the period.
A lattice random variable with period $d$ and its distribution function are called {\em $d-$lattice}.\footnote{Most standard discrete distributions are $d$-lattice with span $d=1$ (e.g., Bernoulli, Binomial, Poisson, geometric, and negative binomial).}

Let $\{X_n\}$ be a sequence of random vectors in $\R^d$ and let $X$ be another random vector in $\R^d$.
We say $X_n$ converges to $X$ in probability or weakly, written $X_n \xrightarrow{p}X$, if for every $\epsilon>0$
\[
    \lim_{n \to \infty} \P\{\Vert X_n-X \Vert < \epsilon\} =1,
\]
where $\Vert \cdot \Vert$ is any norm on $\R^d$.
We say $X_n$ converges to $X$ almost surely or strongly, written $X_n \xrightarrow{a.s.}X$, if
\[
     \P \left\{\lim_{n \to \infty}\Vert X_n-X \Vert = 0 \right\} =1.
\]

% Let $\{F_n\}$ and $F$ be probability measures on the metric space $\chi$.
% We say $\{F_n\}$ converges weakly to $F$, written $\{F_n\} \xrightarrow{w}F$, if for any bounded continuous function, $g$, on $\chi$
% \[
%     \lim_{n\to \infty} \int g dF_n = \int g dF.
% \]

\section{Literature review}
\label{section:literature}

Maintenance optimization models for stochastically deteriorating systems have been extensively explored in the literature.
Comprehensive reviews of the field can be found in \cite{DEJONGE2020} and \cite{ARTS2024}.
When the stochastic process for the deteriorating system is known and the state transitions are observable, a majority of works on discrete-time CBM use either Markov decision processes (MDP) \citep[see e.g.][]{Derman1963a,Derman1963b,Kolesar1966,Ross1969,Andersen2022} or renewal theory \citep[see e.g.][]{Poppe2018,ZHANG2020} to compute optimal maintenance decisions.
These studies typically assume that all components share the same degradation process (i.e. a homogeneous population of components), this degradation process is known by decision makers, and they can perfectly observe the state of the system at any time. 
They show that, under these assumptions the optimal policy has a simple structure: a fixed threshold exists such that it is optimal to replace a component once its degradation level exceeds this threshold, and to continue operation otherwise.
Even when such simple threshold policies are not optimal, some studies focus on finding the best threshold policy, as non-threshold policies are difficult to implement in practice \citep[see e.g.][]{Feldman1986}.
However, the assumption of homogeneous component behavior seldom holds in practice, as components typically exhibit distinct degradation dynamics due to differences in their physical properties and operating conditions.
Moreover, the underlying degradation process of a component is rarely known in advance and must be inferred from condition monitoring data.
In addition, condition monitoring data can contain noise, and provide imperfect information about the component’s state.
Handling these uncertainties requires more complex decision making frameworks.

One major research stream to address these assumptions focuses on systems where low-cost, real-time condition monitoring data is available.
This stream has gained attention over the past two decades thanks to technological advancements that make such data accessible.
Studies in this area generally consider a specific stochastic process whose parameters, drawn from a prior distribution, are learned from observed condition monitoring data.
\cite{Elwany2011} consider a Brownian motion wear model with a drift parameter which is estimated from the latest degradation observation and formulates a POMDP to compute the optimal condition-based maintenance decisions.
In a similar vein, \cite{Chen2015} focus on the inverse Gaussian wear process, \cite{Zhang2016} on the gamma process, and \cite{VanOosterom2017} on a Markovian stochastic process characterized by a finite set of transition matrices.
Recently, \cite{Drent2023} study compound Poisson shock models, where both the shock arrival rate and the damage process parameters are unknown and drawn from some priors.
They establish that if the damage process belongs to a one-parameter exponential family with conjugate priors, the optimal policy can be computed using a tractable POMDP framework.
\cite{Drent2023} demonstrate that the optimal replacement policy relies not only on the observed degradation level but also on other available information, such as the number of shocks occurring between decision epochs and the component ages.

Our work advances this research stream by broadening both its scope and methodology to address a wider and more realistic class of CBM problems.
Specifically, our model accommodates a general non-decreasing independent and identically distributed degradation process with unknown parameters, which are estimated using a general consistent estimator.
To the best of our knowledge, no prior work addresses such a broad range of degradation and learning processes.
Although this problem can in principle be formulated as a POMDP, such an approach becomes computationally intractable for most realistic degradation and learning processes, particularly when component lifetimes are long.
To overcome this limitation, we propose an alternative renewal-theoretic framework and establish its asymptotically optimality when both the average component lifetime and the associated maintenance costs are large.
A further notable study in this stream, \cite{Kim2016}, formulates an optimal CBM policy designed to ensure robustness against posterior mis-specification.
It is worth noting that another research stream addresses uncertainty in observations when real-time condition monitoring data are unavailable or costly to obtain \citep{Girshick1952, Ross1971b,Maillart2006, Maillart2007, Kim2013, vanStaden2021, Khaleghei2021, Zhang2023, sun2023robust, sun2025optimal, wang2025learning}.

From an optimization perspective, this paper adopts a renewal-theoretic framework, in which the replacement of a component constitutes a renewal event, the component’s lifetime defines the renewal cycle length, and the associated maintenance cost represents the renewal cycle cost.
The use of this paradigm dates back to the classical work of \cite{Barlow1960} on age based maintenance models.
Since then, this approach has been used in some discrete-time CBM papers \citep{Kim2013, Poppe2018,ZHANG2020,vanStaden2021}.
Nevertheless, this paradigm has been less common than MDP models in discrete-time CBM research, as MDP formulations facilitate both the derivation of structural results and the computation of optimal policies.
It is noteworthy that a few other studies, such as \cite{Brenguer2003}, also use asymptotic approximations based on the renewal theory for maintenance problems.

\section{Model}
\label{section:Model}

\begin{comment}
We will consider the problem of whether or not to replace a component based on its condition. In Section \ref{subsect:degrmodelHomogeneous} we introduce the degradation model of components, then in Section \ref{subsect:optimalReplacement} we formulate the decision problem as a Markov Decision Process (MDP) under the average cost criterion and show that optimal policies can be studied through regenerative processes. 
Then in Section \ref{subsect:AsymptoticRegimeWithoutLearning} we introduce an asymptotic scaling of the replacement problem and give characterizations of the asymptotically optimal replacement policies.
In Section \ref{subsect:ProofThm1} we prove the main convergence result, i.e., Theorem \ref{thm:costconvergancenolearning}.
\end{comment}

\subsection{Degradation Process}
\label{subsect:degrmodelHeterogeneous}
We consider components that degrade with stationary non-negative independent increments. The components are indexed in the natural numbers $\N=\{1,2,\ldots\}$. A decision maker can decide to replace a component at the beginning of evenly spaced time periods.
Without loss of generality, we rescale time such that a period is one time unit.
The periods are numbered on non-negative integers $\N_0=\N\cup\{0\}$ and forward in time. 
We assume that replacement occurs instantaneously, and we define the beginning of a period $t \in \N_0$, when a decision is made, as epoch $t$.
The age of an operating component is the number of epochs elapsed since its installation.
The degradation increment between age $\tau-1\in \N_0$ and $\tau\in \N_0$ of component $i \in \N$ is denoted by $X_{i,\tau}$ so that the total degradation at age $\tau$ of component $i$ is given by
\[
S_{i,\tau} = \sum_{j=1}^\tau X_{i,j}.
\]
The empty sum is zero such that $S_{i,0}=0$, implying that a new component has no cumulative degradation; that is, the component is in an as-good-as-new state. The component $i$ fails when the degradation level, $S_{i,\tau}$, exceeds the failure level $L < \infty$.
% We denote a realization of $S_{i,\tau}$ by $s_{i,\tau}$.
For each item $i \in \N$, the sequence $\{X_{i,\tau}\}_{\tau \in\N}$ are non-negative independently and identically distributed random variables with a common probability distribution function $F_X (\cdot \mid \theta_i)$, where $\theta_i \in \Theta$ is the parameter of $F_X(\cdot \mid \theta_i)$, and $\Theta \subseteq\R^\rho, \rho \in \N$ is the parameter space.
%$F_X (\cdot,\theta_i)$ is well-defined for all $\theta_i \in \Theta$.
The parameters, $\theta_i$, vary across components with a probability distribution function $F_\theta: \Theta \to [0,1]$.
Furthermore, the degradation process of components are mutually independent, with each $\theta_i$ representing an independent realization from $F_\theta$. $F_\theta$ thus captures the heterogeneity in the population of components; if $F_\theta$ were degenerate, all components would be identical with respect to their degradation behavior.

We impose the following conditions on the distribution functions $F_\theta$ and $F_X(\cdot \mid \theta)$, $\theta \in \Theta$, and the random variable $X_{i,1}$, $i\in\N$.
\begin{enumerate}[label=A.\arabic*. , ref=A.\arabic*]
    \item $F_\theta$ is continuous with a density $f_\theta$.
    \item Fix $x\geq0$; $F_X(x \mid \theta)$ is a continuous function of $\theta$.
    \item Fix $\theta \in \Theta$; $F_X (\cdot \mid \theta)$ is continuous with a density function $f_X(\cdot \mid \theta):=dF_X (x \mid \theta)/dx$
    % such that for all $x>0$
    % \[
    %     \lim_{x^\prime \to x} \int_{0}^\infty \vert f_X(y-x \mid \theta_i) - f_X(y-x^\prime \mid \theta)   \vert dy= 0;
    % \]
    or is $d-$lattice, for some $d>0$, with a mass function $f_X(\cdot \mid \theta)$.
    \item \label{assum:FlatRigion} There exists a finite length $\xf \geq 0$ such that for all $x_1, x_2 \in \R_+, \theta \in \Theta$, if $x_2 \geq x_1 + \xf$ then $F_X(x_1\mid \theta) < F_X(x_2\mid \theta)$.
    \item \label{assum:MeanPositivity} The increments have strictly positive mean $\mu(\theta):=\E[X_{i,1} \mid \theta_i=\theta]  < \infty$ and finite variance $\sigma^2 (\theta):=\Var[X_{i,1} \mid \theta_i=\theta] < \infty$.
    \item $\mu(\theta_i)$ and $\sigma(\theta_i)$ are continuous functions of $\theta_i$.
    \item \label{assum:converge} The following integrals are finite,
    \begin{enumerate}
        \item $\displaystyle\int_\Theta \dfrac{1}{\mu(\yb)}dF_\theta(\yb)$,
        \item $\displaystyle\int_\Theta \dfrac{\sigma^2(\yb)}{\mu^2(\yb)}dF_\theta(\yb)$.
    \end{enumerate}
\end{enumerate}

The above assumptions impose mild regularity conditions that ensure analytical tractability while covering a broad range of degradation settings encountered in practice. Assumptions A.1--A.3 allow components to differ through a continuously distributed latent parameter while accommodating both continuous and lattice-valued degradation increments, which covers many commonly used degradation models. Assumption~\ref{assum:FlatRigion} rules out increment distributions with arbitrarily large gaps in their support, reflecting that physical degradation processes usually generate wear increments across a continuous range of magnitudes. Assumptions~\ref{assum:MeanPositivity}--A.6 require that degradation increments have positive finite mean and variance and that these moment characteristics vary smoothly across components, which reflects gradual wear processes observed in many applications. Finally, Assumption~\ref{assum:converge} places mild integrability conditions on the population heterogeneity to ensure that averages involving degradation rates remain well defined. Together, these assumptions capture heterogeneous populations of components whose degradation evolves gradually over time, while excluding pathological cases that are unlikely to arise in practical systems.

\noindent
% We let $M_X:=\sup\{x\geq 0 \mid  F_X(x)<1\}$ denote the upper bound of the support of an increment, which may be infinite.

\subsection{Observation Process}
Let $T_i$ denote the lifetime of the component $i\in \N$, which is the age of the component when it is replaced.
$T_i$ is a random variable that depends on the replacement decisions of component $i$.
Let $N(t)$ denote the index of the component that operates at epoch $t$ (before potential replacement):
\[
    N(t) := \min\left\{i\in\N_0 : \sum_{j=1}^i T_j \geq t \right\}.
\]
Moreover, let $A(t)$ denote the age of the component that operates at epoch $t$ (before potential replacement): 
\[
    A(t) := t - \sum_{i=1}^{N(t)-1} T_i.
\]
Note that $N(t)$ and $A(t)$ are random variables. 
%Realizations of $N(t)$ and $A(t)$ are denoted by $a(t)$ and $n(t)$, respectively.
At each epoch, every decision maker has access to the complete past of components numbers, ages and degradation increments:
$\left\{N(j), A(j), S_{N(j),A(j)} \right\}_{j=0}^{t}$. In practice, decision makers may have access to more information and we let $Q_t$ denote additional information that becomes available at epoch $t$. For example, $Q_t$ may contain (i) the degradation level(s) measured in real-time between epochs $t-1$ and $t$ (not only the accumulated level at epoch $t$) or it may contain (ii) (real-time) sensor readings of co-variates that relate to  the condition of the component or its degradation parameter $\theta_{N(t)}$. Thus, in general, the decision maker has access to the degradation data 
\[
\Ib_t :=\left\{ N(j), A(j), S_{N(j),A(j)}, Q_j \right\}_{j=0}^t
\]
at epoch $t$, taking values in $\I_t$.
Note that $\Ib_0$ contains all historical data before the start of the planning horizon.
% An arbitrary realization of $\Ib_t$ is denoted $I_t$.
% At the beginning of each period $t \in \N$ the decision maker knows the degradation data $\I_t$ and
We let $\F$ denote the sequence of sigma-algebras generated by $\{\Ib_t\}_{t \in \N_0}$, i.e., $\F := \{\F_t\}_{t \in \N_0}$, where $\F_t = \sigma(\Ib_t)$. (Note that $\sigma$ is used here to denote an induced sigma-field, not a standard deviation.) That is, $\F$ is the natural filtration of the underlying probability space.

%For instance, when degradation follows a shock model, $q_t$ records the realized number of shocks between $t-1$ and $t$ and their associated damages, along with any additional relevant details.
Note that the true parameters of the degradation processes for the operating components $N(t)$ are not included in $\Ib_t$. However, we construct a hypothetical decision maker that does have access to the true degradation parameters $\theta_{N(t)}$. We call this decision maker the ``Oracle" and note that at epoch $t$ it has access to $\Ob_t:=\Ib_t \cup \left\{\theta_{N(j)} \right\}_{j=0}^t$. As before, we let $\O$ denote the filtration of the Oracle. 
%A (hypothetical) decision maker is called an ``\textit{Oracle}" if, at any period $t \in \N_0$, her degradation data, $I^\O_t$, contains also $\left(\theta_{n(j)} \right)_{j=1}^t$, that is $\left(\theta_{N(j)} \right)_{j=1}^t$ is included in $ \I^\O_t$. The Oracle's filtration is denoted by $\I^\O$.
%It is worth mentioning that, in period $t$, the degradation data of any decision maker contains only realizations observed up to $t$.
%All future realizations remain unknown.

\subsection{Decision Problem}
At each epoch $t\in\N_0$, the decision maker chooses either to continue operating the component or to replace it instantaneously, based on the available degradation data up to $t$, i.e., $\Ib_t$. Furthermore, the costs incurred between epochs $t-1$ and $t$ are realized at epoch $t$.
The chosen action determines the evolution of the system's state.
% At epoch $t$, replacement is carried out, if this action is selected, and the costs incurred between epochs $t-1$ and $t$ are realized.
% Notice that the operating component undergoes one further degradation increment between the decision and the execution of the replacement, and the duration of the replacement itself is assumed to be negligible.
% When an intervention occurs in period $t$, the newly installed component starts its operation in an as-good-as-new state, denoted by state 0, at the beginning of period $t+1$.
We next introduce the cost structure of the model and then formulate the cost optimization problem faced by the decision maker.

Let $c_p\in\R_+$ denote the cost to replace a component preventively. Correctively replaced components incur a cost of $c_p+c_f$ such that $c_f\in\R_+$ is the additional cost of corrective replacement. 
Finally, the system incurs an operating cost between epochs $t-1$ and $t$ whose expectation depends on the degradation level observed at epoch $t$ and is given by $\ell: \R_+ \to \R_+$.
Note that this slightly different from the usual convention where operating costs are accounted for the coming in stead of the past period. Mathematically these two conventions are equivalent and our convention will lead to cleaner notation later.
We assume that $\ell$ is non-decreasing, bounded, and continuous and satisfies the following conditions:
%Finally, the system incurs an (expected) operating cost  $\ell:\R_+\to\R_+$ in period $t$ that depends on the degradation level  at the beginning of period $t$, $S_N(t),A(t)$. 
\begin{enumerate}[resume, label=A.\arabic*. , ref=A.\arabic*]
    \item \label{assum:ell0} $\ell(0)=0$ and $\displaystyle \lim_{x\to 0} \ell(x)/x =0 $,
    \item \label{assum:ellbounded} $\ell$ attains its maximum at $L$; that is, $\ell(x)=\ell(L)$ for all $x\geq L$.
\end{enumerate}
Condition~\ref{assum:ell0} implies that operating costs, at low levels of degradation, increase gradually and that minor degradation has, in relative terms, a limited impact on operating costs. This is consistent with practice, where small amounts of wear do not yet translate into significant operating inefficiencies.

The last condition serves mainly to simplify notation, and the main results hold (with minor modifications) without it.
% At a period $t \in \N_0$ the decision maker has the following choices concerning an operating component $i = N(t)$ aged $\tau = A(t)$:
% \begin{enumerate}
%     \item If the component has not failed, it can either continue to operate until the next period at the cost $\ell(S_{i,\tau})$, or be replaced preventively at cost $\ell(S_{i,\tau})+c_p$;
%     \item If the component has failed, it should be replaced correctively at cost $\ell(S_{i,\tau}) + c_p+c_f$.
% \end{enumerate}
The maintenance cost incurred by component $i$, denoted by $C_i$, is then given by,
\begin{equation}
    \label{eq:Cnindicator}
    C_i  := c_p  + c_f \cdot \ind_{S_{i,T_i} > L} + \sum_{\tau=1}^{T_i} \ell(S_{i,\tau}),
\end{equation}
where, $\ind$ is the indicator function.

An admissible policy (or simply a policy) $\pi:=\{\pi_t\}_{t\in \N_0}$ is a sequence of functions $\pi_t: \I_t \to \{ \mbox{continue},\mbox{replace}\}$ that takes in the data available at time $t$ and maps it to a decision to continue to operate the current component, or replace it.
Specifically, for the decision maker, $\pi_t$ is $\F_t-$measurable, which is equivalent to saying that $\pi$ is adapted to $\F$.
For the Oracle, $\pi$ is adapted to $\O$.
Under a policy $\pi$, the age of component $i \in \N$ at its replacement epoch is denoted by $T_i(\pi)$, and the total maintenance cost by $C_i(\pi)$.
The (infinite horizon) cost-rate of policy $\pi$ under filtration $\F$ (or $\O$ for the Oracle), denoted $g(\pi)$, is defined as
\begin{equation}
    \label{eq:gdef}
    g(\pi) := \lim\sup_{n \to \infty} \frac{\sum_{i=1}^{n} C_i(\pi)}{\sum_{i=1}^{n} T_i(\pi)}.
\end{equation}
% Notice that the policies that a decision maker selects depend on the available information, i.e. $\F$.
Let $\Pi_\F$ and $\Pi_\O$ denote the set of policies adapted to $\F$ and $\O$, respectively. 
The optimal cost-rate under $\F$ and $\O$, denoted $g^* (\F )$ and $g^* (\O )$, respectively, are given by
\begin{equation}
    \label{eq:DecisionProblem}
    g^*(\F) := \inf_{\pi \in \Pi_\F} g(\pi), \qquad g^*(\O) := \inf_{\pi \in \Pi_\O} g(\pi).
\end{equation}
When $g^*(\F)$ $\left(g^*(\O)\right)$ can be attained by a policy, such a policy is said to be optimal and is denoted $\pi^*(\F)$ ($\pi^*(\O)$).
The objective of the decision maker is to identify an optimal policy, $\pi^*(\F)$, if one exists.

\subsection{Outline of Arguments and Main Results}
\label{subsection:Outline}
In general, proving the existence of $\pi^*(\F)$ and, if it exists, computing it is not trivial.
Moreover, its implementation in practice may be highly challenging, as the policy may have a complex structure.
In certain cases, where the state space is discrete and conjugate families are used to estimate $\theta$, this problem under the discounted total cost criterion can be solved, tractably, using a multi-dimensional POMDP \citep[cf.][]{Elwany2011,Drent2023}. However, even in these specific situations, the POMDP becomes intractable quickly as (i) the failure level and/or (ii) the number of unknown parameters increases.
To overcome these challenges, our goal is to construct policies that approximate the Oracle decisions.
Notice that by construction, the Oracle has access to more information than a real decision maker, and therefore its achievable cost-rate is less than or equal to that of any other decision maker, since $\Pi_\F \subseteq \Pi_\O$, i.e., the Oracle faces an information relaxation of the problem faced by a real decision maker and so
\begin{equation}
    \label{eq:informativeness}
    g^*(\O) \leq g^*(\F),
\end{equation}
\citep[cf. Blackwell's informativeness in][]{Blackwell1951,Blackwell1953}.
We define the regret associated with any policy $\pi \in \Pi_\F$, relative to the Oracle’s optimal policy $\pi^*(\O)$, as $\reg(\pi) := g(\pi) - g^*(\O)$.
Under this definition, inequality \eqref{eq:informativeness} is equivalent to stating that the regret of every policy $\pi \in \Pi_\F$ is nonnegative, i.e., $\reg(\pi) \geq 0$.

% The decision problem in \eqref{eq:gdef} has not yet been addressed through a general tractable framework.
% A similar problem, under the discounted total cost criterion and for certain selections of the degradation data, as well as of $F_X$, and $F_\theta$, can be formulated as a multi-dimensional POMDP and solved to optimality \citep[cf.][]{Elwany2011,Drent2023}.
In the subsequent sections, we develop an efficient policy for non-Oracle decision makers.
This policy is constructed by first estimating the unknown parameters of the operating component, $\theta_{N(t)}$, from degradation data, $\Ib_t$, using a consistent estimator, and then estimating the Oracle’s optimal actions based on these parameter estimates.
We impose no assumptions on the parameter estimator other than consistency; in particular, standard Bayesian estimators, such as the posterior mean, are consistent in our setting
We term this policy the Estimated Oracle’s Optimal Policy (EOP).

We further introduce a scaling regime in which the failure thresholds, $L$, grow large together with the cost parameters $c_p$ and $c_f$, as well as the total operating cost $\sum_{\tau=1}^T \ell(S_\tau)$ to ensure that the cost rate of any policy remains finite and positive.
Specifically, we consider a base instance with failure threshold $\tilde{L}$, cost of preventive replacement $\cT_p$, additional
corrective replacement cost $\cT_f$, and operating cost function $\ellT$.
Then we create a continuum of instances for each $k>0$ such that it has failure threshold $L(k)=k\LT$, cost of preventive replacement $c_p(k)=k\cT_p$, additional cost of corrective replacement $c_f(k)=k\cT_f$, and operating cost function $\ell(x,k) = \ellT(x/k)$.
Thus, $k$ is a scaling parameter that increases both the maintenance costs and the failure threshold.
This regime is motivated by practical settings where component lifetimes are typically much longer than the time between decision epochs, and maintenance costs constitute a significant share of the average system's operating expenses per unit time.

Under our scaling regime, the cost-rate and regret functions are parameterized by the scaling parameter $k$, that is $\reg_k(\pi):=g_k(\pi)-g_k^*$ (cf. Equation \eqref{eq:gdef}).
We now state our main theoretical result:\\
\textbf{Main Result.} \textit{The regret associated with using the EOP instead of the Oracle’s optimal policy vanishes in the limit as $k$ approaches infinity, that is}
\[
    \lim_{k \to \infty} \reg_k(\mbox{EOP}) = 0.
\]
This result is formally stated in Theorem \ref{thm:costconvergancewithlearning}.
Our main result, together with Equation \eqref{eq:informativeness}, establishes the asymptotic optimality of the EOP within the proposed scaling regime.

We organize our theoretical exposition as follows. Section \ref{section:OracleOptimal} characterizes the Oracle’s optimal policy $\pi^*(\O)$, which serves as the basis for the construction of the EOP in Section \ref{section:AOARP}.
Our main result is then rigorously presented at the end of Section \ref{section:AOARP}.
Subsequently, in Section \ref{section:AsymptoticAnalysis} we explain how the results in Section \ref{section:AOARP} can be obtained, which also provides insights that improve the interpretability of our adaptive policy for practitioners. We conclude this section by introducing the compound Poisson degradation model, which serves as the running example throughout the paper and underlies the practice-based numerical study.

% \begin{comment}
\begin{example}[Compound Poisson Degradation]
\label{ex:exampleDef}
Poisson cumulative damage models have been extensively employed in both the academic literature and practical applications to characterize the progressive deterioration of components or systems under stochastic shocks.
In these models, the operating component experiences shocks that occur according to a Poisson process with rate $\nu$. Each shock induces an i.i.d. random damage, and the damages accumulate additively over time.
Thus, $X_{i,\tau}$, $i \in \N, \tau \in \N_0$, follows a compound Poisson distribution.
Here, we omit the index $i$, for notational simplicity.

The compound Poisson process is defined as follows.
Let $M_{(\tau_1,\tau_2]}$, with $\tau_1 < \tau_2$, denote the number of shocks occurring during the time interval between ages $\tau_1, \tau_2 \in \N_0$, denoted by $(\tau_1,\tau_2]$, and let $Z_{m,\tau}$ denote the damage from the $m$-th shock in the interval $(\tau-1,\tau]$.
The sequence $\{Z_{m,\tau}\}_{m \in \N, \tau \in \N_0}$ consists of non-negative i.i.d. random variables with a common distribution function $F_Z(\cdot\mid \theta_Z)$ with a mass or density function $f_Z(\cdot\mid \theta_Z)$, where $\theta_Z$ is the parameter of $F_Z$.
The parameter space, $\Theta$, is defined as the set of all possible values of $(\nu,\theta_Z)$ that may be realized across the operating components, i.e., $\Theta = \R_{++} \times \Theta_Z$ where $\Theta_Z \subseteq \R^\rho, \rho \in \N$, denote the spaces from which $\nu$ and $\theta_Z$ are drawn according to the distributions $F_{\nu}$ and $F_{\theta_Z}$, respectively. Notice that $F_{\nu}$ and $F_{\theta_Z}$ are continuous with densities $f_{\nu}$ and $f_{\theta_Z}$, respectively.
In this case $f_{\theta}(x,\yb) = f_{\nu} (x)f_{\theta_Z}(\yb)$, where $x \in \R_++$ and $\yb \in \Theta_Z$.
By the memory-less property we have $M_{[\tau_1,\tau_2)} = \sum_{\tau=\tau_1+1}^{\tau_2} M_\tau,$ where $M_\tau:=M_{[\tau-1,\tau)}$.
Consequently, the cumulative damage between $\tau-1$ and $\tau$, i.e., $X_\tau$, is given by $X_\tau = \sum_{m=1}^{M\tau} Z_{m,\tau}$.
For any $\tau \in \N_0$, the probability mass function of $M_\tau$, is $f_M(x \mid\nu):= \P \left\{M_{[\tau, \tau+1)} =x \mid \nu\right\} = \nu^x e^{-\nu}/x!$, where, $x \in \N_0$.
Furthermore, by the law of total expectation and law of total variance we have
\begin{equation}
    \label{eq:CP_moments}
    \E \left[X_\tau \mid \nu\right] = \nu \E \left[Z_{1,1} \right], \qquad  \Var \left[X_\tau \mid \nu \right]=\nu \Var\left[Z_{1,1} \right] + \nu \left(\E \left[Z_{1,1} \right] \right)^2.
\end{equation}

To facilitate the analysis of our examples, we assume a specific distribution for $Z_{m,\tau}$. In particular, we adopt the (lattice) geometric distribution with parameter $\theta_Z = p \in (0,1)$, motivated by our practice-based case study (see Section \ref{section:casestudyresults}). In Appendix \ref{expCompoundingApp}, we additionally consider the (continuous) exponential distribution with parameter $\theta_Z = \omega \in \mathbb{R}_{++}$. 

Let $Z_{m,\tau}$ be supported on $\N_0$ and have a geometric distribution with the success probability of $p \in (0,1)$, that is, $f_Z(x\mid p):=\P \left\{Z_{m,\tau}=x \mid p \right\} = \left(1-p \right)^x p$, where $x \in \N_0$.
Next, by the application of the law of total probability the distribution mass function of $X_\tau$ is expressed by
\begin{equation}
    \label{eq:pmf_CP_geom}
    f_X \left(x \mid\nu,p \right) = \begin{cases}
        f_M \left(0 \mid \nu \right) + \displaystyle\sum_{j=1}^\infty f_M \left(j \mid \nu \right)f_{NB}(0\mid j,p), & x=0\\
        \displaystyle\sum_{j=1}^\infty f_M \left(j\mid \nu \right)f_{NB}(x\mid j,p), & x \in \N,
    \end{cases}
\end{equation}
where $f_{NB}(\cdot \mid j,p)$ is the negative binomial probability mass function with the $j$ number of successes and success probability of $p$, i.e., $f_{NB}(x \mid j, p):= \binom{x+j-1}{x}(1-p)^x p^j$, with $x \in \N_0$.
Additionally by Equations \eqref{eq:CP_moments} we have
\begin{equation}
    \label{eq:CP_moments_geom}
    \mu(\nu,p) = \frac{\nu(1-p)}{p}, \qquad \sigma^2(\nu,p) = \frac{\nu(1-p) (2- p )}{p^2}.
\end{equation}

The degradation data observed by the decision maker at epoch $t \in \N_0$ includes the history of replacement times, the number of shocks, and the size of the damages induced by each shock up to $t$.
That is,
\[
    \Ib_t = \left \{N(j),A(j),M_{N(j),A(j)},Z_{N(j),1, A(j)}, \dots, Z_{N(j), M_{N(j),A(j)}, A(j)} \right\}_{j=0}^t.
\]
This degradation data also enables the computation of $X_{N(t),A(t)}$ and $S_{N(t),A(t)}$.
The functions $F_{\nu}$ and $F_{\theta_Z}$, as well as $F_X$ for every realization of $(\nu, \theta_Z)$ are also known.
The decision maker may sequentially estimate $(\nu, \theta_Z)$ in a Bayesian manner, based on the degradation data $\Ib_t$ using the estimators
$\left(\nuH \left( \Ib_t \right),\thetaH_{Z} \left( \Ib_t \right) \right)$, with density functions $f_\nu(\cdot \mid \Ib_t)$ and $f_{\theta_Z} (\cdot \mid \Ib_t)$, respectively.

Returning to our example, we notice that $f_M(\cdot \mid\nu)$ and $f_Z( \cdot \mid p)$ belong to the one-parameter exponential family.
This property allows us to select conjugate priors for each of these (likelihood) functions, so that the posterior distributions remain within the same family, and the corresponding parameters can be updated efficiently, conditional on the degradation data $\Ib_t$ \citep[cf. Section 5.1.5.][]{Ghosh2007}.
The selection of the conjugate priors is as follows.

We first select that $f_{\nu} (\cdot) = f_\nu(\cdot \mid \Ib_0)$ corresponds to a gamma distribution with shape $\alpha_0$ and scale $\beta_0$. The density function $\theta_Z$ follows a beta distribution parametrized by $a_0$ and $b_0$. The selected $f_{\nu}$ and $f_{\theta_Z}$ are conjugate priors corresponding to likelihood functions $f_M$ and $f_Z$. In both cases, $f_{\theta_Z}$.
One can simply verify that the chosen $f_{\nu}$ and $f_{\theta_Z}$ satisfy all conditions specified in Section \ref{subsect:degrmodelHeterogeneous}.

Then, $\nuH \left( \Ib_t \right)$ follows a gamma distribution with parameters $\alpha_t$ and $\beta_t$.
Furthermore, $\thetaH_{Z} \left( \Ib_t \right)$ follows a beta distribution with parameters $a_t$ and $b_t$.
We refer to $(\alpha_t, \beta_t, a_t, b_t)$ as hyperparameters.
Let $M_{A(t)}$ and $X_{A(t)}$ be the number of shocks and the degradation increment in the interval $(t-1,t]$.
% Let $m$ and $x$ denote the number of shocks and the corresponding cumulative degradation during the period $t-1 \in \N_0$, respectively.
Then, the hyperparameters can be updated as
\begin{alignat}{2}
    \label{eq:hyperParameter_update}
    \left(\alpha_t, \beta_t, a_t, b_t \right)=& \left(\alpha_{t-1}+M_{A(t)}, \beta_{t-1}+1, a_{t-1}+M_{A(t)}, b_{t-1}+X_{A(t)} \right) \nonumber\\
    =& \left(\alpha_0+\sum_{j=1}^{A(t)} M_j, \beta_0+A(t), a_0+\sum_{j=1}^{A(t)} M_j, b_0+S_t \right)
\end{alignat}
% \begin{alignat*}{2}
%     &\left(\alpha_t, \beta_t, a_t, b_t \right)=\left(\alpha_{t-1}+M_{A(t)}, \beta_{t-1}+1, a_{t-1}+X_{A(t)}, b_{t-1}+M_{A(t)} \right), \quad && \textit{if the component is not in state 0},\\
%     &\left(\alpha_t, \beta_t, a_t, b_t \right)=\left(\alpha_0, \beta_0, a_0, b_0 \right), \quad && \textit{if the component is in state 0}.
% \end{alignat*}

We note that this analysis extends to many choices of $f_Z$ from the one-parameter exponential distributions (we do so for the exponential distribution in Appendix \ref{expCompoundingApp}) if their conjugate priors satisfy the conditions outlined in Section \ref{subsect:degrmodelHeterogeneous}. 
% , supported on $\R_+$,
% \[
%     f_Z(z; \theta_Z ) = e^{\theta_Z \psi(z) - H_1(\theta_Z)+H_0(z)}.
% \]

% We illustrate the Poisson cumulative damage models using two examples for the shock damage $Z_{m,\tau}$.
% In the first example, $Z_{m,\tau}$ follows a discrete distribution, specifically a geometric distribution with parameter $p$.
% In the second example, $Z_{m,\tau}$ follows a continuous distribution, namely an exponential distribution with parameter $\omega$.

$\hfill \blacklozenge$
\end{example}

\section{Oracle's Optimal Policies}
\label{section:OracleOptimal}
In this section, we address the optimal maintenance policy for the Oracle, i.e., $\pi^*(\O)$, whose degradation data at each epoch $t \in \N_0$ contains the parameters of the degradation process for all operating components up to $t$, namely $\theta_{N(1)}, \dots, \theta_{N(t)}$. 
Throughout this section, the discussion refers to filtration $\O$, and therefore we omit it from the notation.
For example, we write $g^*:=g^*(\O)$ and $\pi^*:=\pi^*(\O)$.
The reader may, however, interpret $\O$ implicitly wherever the context requires.

% \subsection{Oracle's Optimal Policies}
% \label{section:OracleOptimal}

A replacement policy $\pi$ is called stationary if and only if $\pi_t=\pi_{t+1}$ for all $t$.
The critical observation to find the optimal policy of the Oracle is that the future evolution of degradation at epoch $t$ depends on $\Ob_t$ only through $\theta_{N(t)}$, i.e.,  $\{S_{N(t),A(t)},\theta_{N(t)}\}_{t\in\N}$ is a Markov process for any stationary policy $\pi$. 
The process $\{S_{N(t),A(t)},\theta_{N(t)}\}_{t\in\N}$ is positive recurrent because any policy eventually replaces a component (in the worst case after failure) and expected time between failures are finite (see \eqref{eq:ET_bounds} in Appendix \ref{section:Proof_lemma:CostRateHetero}).

A stationary replacement policy $\pi$ is a parameter-specific threshold policy if there exists a threshold function $\xi:\Theta \to [0,L]$ such that at each epoch $t \in \N_0$ the operating component is replaced if and only if the degradation level is strictly greater than $\xi(\theta_{N(t)})$, i.e.,
\begin{equation*}
    \pi_t(S_{N(t), A(t)})=
    \begin{dcases}
        \mbox{continue} & \qquad \mbox{if } S_{N(t), A(t)} \leq \xi(\theta_{N(t)}),\\
        \mbox{replace} & \qquad \mbox{if } S_{N(t), A(t)} > \xi(\theta_{N(t)}).
    \end{dcases}
\end{equation*}
If an Oracle's optimal policy $\pi^*$ exists, then there exists a bounded function $v: R_+ \times \Theta \to \R$ and the optimal cost-rate $g^* \in \R$ such that
\begin{equation}
\label{eq:OPTEq}
g^* + v(s,\theta) = \ell(s) +
\begin{cases}
\min \Big\{
     \E[v(s + X_{1,1}, \theta) \mid \theta_1 = \theta],\;
    c_p + \bar{v}(\theta)
\Big\}, & s \le L,\, \theta \in \Theta, \\[0.6em]
c_p + c_f + \bar{v}(\theta), & s > L,\, \theta \in \Theta,
\end{cases}
\end{equation}
where we use the shorthand for the expected continuation value after replacement,
\[
\bar{v}(\theta)
    = \int_{\Theta} \E[v(X_{1,1}, \theta_1) \mid \theta_1 = \yb]\,
      dF_\theta(\yb),
\]
and $\pi^*$ denotes the policy that satisfies the optimality equations~\eqref{eq:OPTEq}.
It can be readily shown that a parameter-specific threshold policy is optimal under the discounted total cost criterion.
The following theorem states that even under the more challenging -- due to the general state space we are dealing with -- cost-rate criterion, there exists an optimal policy which is likewise a parameter-specific threshold policy.
\begin{theorem}
    \label{thm:ExistNStrucOPTHetero}
    % There exists a bounded function $v:[0,\infty) \times \Theta \to \R$ and an optimal cost-rate $g^* \in \R$ such that,
    \begin{enumerate}[(a)]
        % \item \label{thm:ExistNStrucOPTHeteroOPTEq} $g^* + v(s, \theta) = \begin{cases}
        %         \min \Big\{\ell(s)+\E[v(s+X_{1,1},\theta) \mid \theta_1=\theta], \ell(s)+ c_p + \displaystyle \int_{\Theta} \E[v(X_{1,1},\theta) \mid \theta_1= \yb] dF_\theta(\yb) \Big\}, & \forall s \leq L, \theta \in \Theta \\
        %         \ell(s)+c_p+c_f + \displaystyle \int_{\Theta} \E[v(X_{1,1},\theta) \mid \theta_1 = \yb] dF_\theta(\yb), & \forall s > L, \theta \in \Theta,
        %     \end{cases}$
        \item \label{thm:ExistNStrucOPTHeteroOPTPolicy}  There exists a (cost-rate) optimal policy $\pi^*$, where $g(\pi^*) = g^*$,
        \item \label{thm:ExistNStrucOPTHeteroThreshPolicy}  For each $\theta \in \Theta$, there is an optimal parameter-specific threshold policy, i.e.,
        \begin{equation*}
            \pi^*_t(S_{N(t), A(t)})=
            \begin{dcases}
                \mbox{continue} & \qquad \mbox{if } S_{N(t), A(t)} \leq \xio(\theta_{N(t)}),\\
                \mbox{replace} & \qquad \mbox{if } S_{N(t), A(t)} > \xio(\theta_{N(t)}),
            \end{dcases}
        \end{equation*}
        where $\xio: \Theta \to [0,L]$.
    \end{enumerate}
\end{theorem}

\noindent
We henceforth restrict attention to the class of parameter-specific threshold policies, denoted by $\Pi$, as Theorem \ref{thm:ExistNStrucOPTHetero} establishes that this set contains an optimal policy.

Now, let $\Xi$ denote the set of all threshold functions $\xi$.
Observe that the replacement of an operating component under the parameter-specific threshold policy $\pi\in \Pi$ is determined solely via the threshold function $\xi \in \Xi$; thus, it is straightforward to see that $\Xi$ and $\Pi$ are equivalent.
With slight abuse of notation, we write the cost-rate of a policy $\pi \in \Pi$ with threshold function $\xi \in \Xi$ as $g(\xi)$. Accordingly, for the optimal threshold function $\xio$, we have $g(\xio) = g^*$.

We will now study how the cost rate of any threshold policy $\xi$ can be evaluated. To this end let $T(x,\theta)$ be the lifetime of a component with degradation parameter $\theta$ if it is replaced when the degradation exceeds $x \in \R_+$:
\[
    T(x,\theta):= \inf\{\tau\in \N: S_{1,\tau}>x, \theta_1=\theta\}.
\]
Similarly, let $C(x,\theta)$ denote the costs incurred by a components with degradation parameter $\theta$ that is replaced when it degradation exceeds $x \in \R_+$:
\[
    C(x,\theta):= c_p \ind_{S_{1,T(x,\theta)} \leq L} + (c_p+c_f)\ind_{S_{1,T(x,\theta)} > L} + \sum_{\tau=1}^{T(x,\theta)} \ell(S_{1,\tau}), \qquad \theta_1=\theta.
\]
Then the following holds.
\begin{lemma}
    \label{lemma:CostRateHetero}
    Under the policy $\pi \in \Pi$ characterized by the threshold function $\xi \in \Xi$, with probability one we have
    \begin{alignat}{2}
    \label{eq:SLLNDef}
        g(\xi) =& \frac{\displaystyle\int_\Theta \E[C(\xi(\yb),\yb)] dF_\theta(\yb)}{\displaystyle\int_\Theta \E[T(\xi(\yb),\yb)] dF_\theta(\yb)} \nonumber\\
        = &\frac{c_p+ \displaystyle\int_\Theta\left(c_f \P\{S_{1,T(\xi(\theta_1), \theta_1)}>L \mid \theta_1=\yb\}+ \E \left[\sum_{\tau=1}^{T(\xi(\theta_1),\theta_1)} \ell(S_{1,\tau}) \Big| \theta_1=\yb \right] \right) dF_\theta(\yb)}{\displaystyle\int_\Theta \E[T(\xi(\yb),\yb)] dF_\theta(\yb)}.
    \end{alignat}
\end{lemma}
The intuition behind Lemma~\ref{lemma:CostRateHetero} is that any two components are probabilistically equivalent, and the corresponding stopping decisions follow the same probabilistic structure.
Equation \eqref{eq:SLLNDef} can be equivalently written in terms of the amount by which the degradation of a component $i$ exceeds $\xi(\theta_i)$ upon replacement. We denote this amount by the random variable $\YT_i(\xi(\theta_i))$ where $\YT_i(x)$ is expressed by $\YT_i(x):= S_{i,T_i}-x$.
% When the index $i$ is omitted, $\YT(\cdot)$ refers to $\YT_1(\cdot)$
Figure \ref{fig:samplepath} illustrates a sample degradation path for the first component, failure level $L$, replacement threshold $\xi(\theta_1)$, excess degradation $\YT_1(\xi(\theta_1))$, and $\phi\left(\xi(\theta_1) \right)$, where $\phi(x):= L-x$ represents the distance of the point $x \in [0,L]$ and $L$.
\begin{figure}[H]
    \centering
    \includegraphics[width=.7\textwidth]{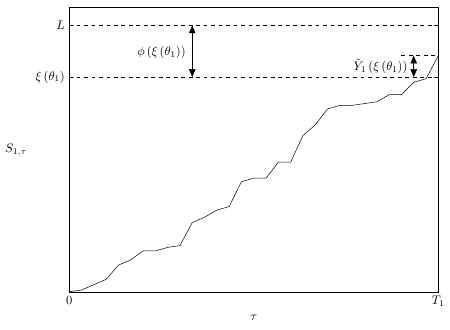}
    \caption{A sample degradation path.}
    \label{fig:samplepath}
\end{figure}
Then we have
\begin{equation}
    \label{eq:SLLN}
    g(\xi) = \frac{c_p+ \displaystyle\int_\Theta\left(c_f \P \left\{\YT_1(\xi(\theta_1))>\phi \left(\xi(\theta_1) \right) \mid \theta_1=\yb \right\}+ \E \left[\sum_{\tau=1}^{T(\xi(\theta_1),\theta_1)} \ell(S_{1,\tau})\mid \theta_1=\yb \right] \right) dF_\theta(\yb)}{\displaystyle\int_\Theta \frac{1}{\mu(\yb)} \left(\xi(\yb)+\E \left [\YT_1(\xi(\theta_1)) \mid \theta_1=\yb \right] \right) dF_\theta(\yb)},
\end{equation}
Note that $S_{i,T_i}=\sum_{\tau=1}^{T_i} X_{i,\tau}=\xi(\theta_i)+\YT_i(\xi(\theta_i))$ so that by Wald's identity $\E[T(\xi,\theta_i)]=(\xi(\theta_i)+\E[\YT_i (\xi(\theta_i))])/\mu(\theta_i)$.

Equation \eqref{eq:SLLNDef} (and equivalently \eqref{eq:SLLN}) implies that any function $\xi \in \Xi$ satisfies
\[
    \displaystyle\int_\Theta \E[C(\xi(\yb),\yb)] dF_\theta(\yb) - g^* \displaystyle\int_\Theta \E[T(\xi(\yb),\yb)] dF_\theta(\yb) \geq0,
\]
where equality holds if and only if $\xi$ corresponds to the optimal threshold function, i.e., $\xi=\xio$.
This property can be exploited to develop an iterative procedure to compute the optimal cost-rate and corresponding policies.

Define $\hT:\R \times [0,L] \times \Theta \to \R$ as
\begin{alignat}{2}
    \label{eq:hTDef}
    &\hT(\lambda, x, \theta) :=  \nonumber\\ &c_p + c_f \P \left\{\YT_1(x)>\phi \left(x \right) \Big \vert \theta_1=\theta \right\}+ \E \left[\displaystyle\sum_{\tau=1}^{T(x,\theta_1)} \ell(S_{1,\tau}) \Bigg \vert \theta_1=\theta \right] -  \frac{\lambda}{\mu(\yb)} \left(\xi(\yb)+\E \left[\YT_1(x) \Big \vert \theta_1=\theta \right]  \right),
\end{alignat}
and let the threshold $\xiT_{\lambda}(\theta)$ minimize $\hT(\lambda, x, \theta )$ for a given $\lambda \in \R$ and $\theta \in \Theta$. It follows that, for $\lambda \in \R$, the threshold function $\xiT \in \Xi$ satisfies
\[
    \xiT_\lambda \in \arg \min_{\xi \in \Xi} h(\lambda, \xi, \theta), \qquad \forall\theta \in \Theta.
\]
Notice that $\xiT_{g^*}$ determines the optimal threshold function $\xio$, and by optimality
\[
    \displaystyle\int_\Theta \hT(g^*,\xiT_{g^*},\yb) dF_\theta(\yb)=0.
\]
We build on the seminal results of \cite{Aven1986} to present the following lemma, which ensures sequential convergence to $g^*$.

\begin{theorem}
    \label{thm:lambda*}
    \begin{enumerate}[(a)]
        \item \label{thm:lambda*_fixedpoint} $g^*$ is the unique solution to the equation $\lambda = g(\xiT_\lambda)$.
        \item \label{thm:lambda*_algorithm} If $\lambda_1 \in \R$,
            % \[
            %     \lambda_1 = c_p + \int_\Theta \left(c_f\P\{X_{1,1}>L \mid \theta_1=\yb\} + \E[\ell(X_{1,1}) \mid \theta_1=\yb] \right) dF_\theta(\yb),
            % \]
            and $\lambda_{j+1} = g(\xiT_{\lambda_j})$, for $j\in\{2,3,\dots\}$, then
            \[
                \lim_{j \to \infty} \lambda_j = g^*.
            \]
    \end{enumerate}
\end{theorem}
Theorem \ref{thm:lambda*} provides a fixed-point iteration to compute $g^*$ and, consequently, the optimal replacement thresholds.
Notice that the functions in \eqref{eq:SLLN} and \eqref{eq:hTDef} can, in general, be numerically evaluated through simulation.

\section{Asymptotically Optimal Adaptive Replacement Policies}
\label{section:AOARP}

\subsection{Estimated Oracle's Optimal Policy (EOP)}
\label{subsection:EOP}
In this section, we study a non-Oracle decision maker who seeks to determine the optimal policy under the filtration $\F$, i.e., $\pi^*(\F)$.
At each epoch $t\in \N_0$, the (true) parameter of the operating component $\theta_{N(t)}$ is unknown to the decision maker and can only be estimated conditional on the degradation data $\Ib_t$.
Since only the information related to the operating component is relevant for estimation, we focus on an arbitrary component without loss of generality.
Let $\theta_{N(t_0+\tau)} \in \Theta$, denote the (true) parameter of the component installed at $t_0$ and aged $\tau$, with $t_0,\tau \in \N_0$.
The decision maker estimates $\theta_{N(t_0+\tau)}$, based on the degradation data $\Ib_{t_0+\tau}$, using a point estimator $\thetaH(\Ib_{t_0+\tau})$.
% which is $\F_{t_0+\tau}-$measurable
For instance, the decision maker first uses the prior distribution $F_\theta$, which represents the baseline belief $F_\theta(\cdot \mid \Ib_0)$, to update her knowledge of $\theta_{N(t_0+\tau)}$ by obtaining a posterior distribution $F_{\theta}(\cdot \mid \Ib_{t_0+\tau})$.
She then estimates $\theta_{N(t_0+\tau)}$ using the posterior mean, which is the Bayes estimator under the squared error loss function.
Note that the information before $t_0$ is irrelevant for posterior updating and estimation.
Therefore, we suppress the index of the operating component and $t_0$, for notational convenience, whenever the context permits.
We also assume that $\thetaH (\Ib_\tau) \in \Theta$ with probability 1, without loss of generality.
The point estimator $\thetaH (\Ib_\tau)$ is said to be (weakly) consistent if $\thetaH (\Ib_\tau) \xrightarrow{p} \theta$.
% , and strongly consistent if $\thetaH (\Ib_\tau) \xrightarrow{a.s.} \theta$.
In the following, we restrict attention to consistent point estimators and refer to them simply as point estimators.
% It is also well known in Bayesian statistics that posterior distributions are consistent in the sense that $F_\theta (\cdot \mid \Ib_\tau) \xrightarrow{w}\delta_{\theta}$ almost surely, where $\delta_{\theta}$ denotes the degenerate distribution at the true parameter value $\theta$ \citep[cf. Corollary A.2.][]{Diaconis1986}.

At each epoch, the point estimate of the parameter can be used to compute a replacement threshold based on the full-information (Oracle) optimal policy.
If the estimate were equal to the true parameter, the resulting threshold would coincide with the optimal full-information threshold.
In particular, the optimal threshold $\xio \left( \theta \right)$ can be estimated by $\xio \left( \thetaH (\Ib_\tau) \right)$.
This requires us first to calculate $g^*$ using Theorem \ref{thm:lambda*}, then to obtain $\xio$ as the minimizer of $\hT(g^*,x,\theta)$ according to Equation \eqref{eq:hTDef}, and finally to compute $\xio \left( \thetaH (\Ib_\tau) \right)$.
% If $h(\lambda,\xi,\theta)$ can be evaluated using Equation \eqref{eq:hDef}, then Theorem \ref{thm:lambda*} allows us to find $g^*$ and by that determine the corresponding optimal threshold function $\xi_{g^*}$.
Unfortunately, applying Equations \eqref{eq:SLLN} and \eqref{eq:hTDef} is generally difficult, as analytical expressions for $\P \left\{\YT_1(\xi(\theta_1))\leq x \mid \theta_1=\theta \right\}$, $\E \left[\sum_{\tau=1}^{T(\xi(\theta),\theta)} \ell(S_{1,\tau}) \Big\vert \theta_1=\theta \right]$, and $\E \left[\YT_1(\xi(\theta_1)) \Big \vert \theta_1=\theta \right]$ are not readily available.
However, we construct computationally efficient approximations of these functions and prove their asymptotic optimality as the failure threshold grows large.

Let $Y(\theta)$, $\theta \in \Theta$ be a random variable such that 
\begin{equation}
    \label{eq:YProb}
    \P\{Y (\theta) \leq x\} =
    \begin{dcases}
        \frac{1}{\mu(\theta)} \int_0^x \left(1 - F_X (y,\theta)\right)dy & \mbox{if $F_X$ is continuous,} \\
        \frac{d}{\mu(\theta)} \sum_{j=0}^{\left\lfloor x/d \right\rfloor-1} (1-F_X(jd, \theta)) & \mbox{if $F_X$ is $d-$lattice,}
    \end{dcases}
\end{equation}
where empty sums are 0 by definition and 
\begin{equation}
    \label{eq:EY}
    \E[Y(\theta)] =
    \begin{dcases}
        \frac{\sigma^2(\theta) + \mu^2(\theta)}{2 \mu(\theta)} & \mbox{if $F_X$ is continuous,} \\
        \frac{\sigma^2(\theta) + \mu^2(\theta)}{2 \mu(\theta)} +\frac{d}{2} & \mbox{if $F_X$ is $d-$lattice.}
    \end{dcases}
\end{equation}
Define $r:[0,L] \times \Theta \to \R_+$ as
\begin{equation}
    \label{eq:r_ell}
    r (x,\theta) = \begin{dcases} \E \left[\ell(X_{1,1})\mid \theta_1=\theta \right]+
        \dfrac{1}{\mu(\theta)}\displaystyle\int_{0}^{x} \E \left[\ell(X_{1,1}+y)\mid \theta_1=\theta \right] dy, & \textit{if $F_x$ is continuous,} \\
        \E \left[\ell(X_{1,1})\mid \theta_1=\theta \right]+ \dfrac{d}{\mu(\theta)}\displaystyle\sum_{j=0}^{\lfloor x/d \rfloor -1}  \E \left[\ell(X_{1,1}+jd) \Big \vert \theta_1=\theta \right], & \textit{if $F_x$ is $d$-lattice}. \\
    \end{dcases}
\end{equation}
We will later show that $Y$ is the asymptotic expansion of $\YT_i$, for all $i\in \N$ as the replacement thresholds tend to infinity.
Similarly, $r(\xi(\theta), \theta)$ provides the asymptotic expansion of $\E \left[\sum_{\tau=1}^{T(\xi(\theta),\theta)} \ell(S_{1,\tau}) \Big\vert \theta_1=\theta \right]$ in the same limiting regime.
Based on these asymptotic expansions, we approximate the cost-rate function $g(\xi)$ by $\gamma(\xi)$ given by
\begin{equation}
    \label{eq:gamma}
    \gamma(\xi) := \frac{c_p + \displaystyle \int_\Theta \left(c_f \P \left\{Y(\yb) > \phi \left(\xi(\yb) \right) \right\} + r(\xi(\yb), \yb) \right) dF_\theta(\yb)}{\displaystyle \int_\Theta \frac{1}{\mu(\yb)} \left(\xi(\yb) + \E[Y(\yb)] \right) dF_\theta(\yb)},
\end{equation}
(cf. Equation \eqref{eq:SLLN}). Morover, we approximate the function $\hT$ by $h$ expressed as
\begin{equation}
    \label{eq:hDef}
    h(\lambda, x, \theta) := c_p + c_f \P \left\{Y(\theta) > \phi \left(x \right) \right\}+ r(x, \theta) - \frac{\lambda}{\mu(\theta)} \left(x + \E[Y(\theta)] \right),
\end{equation}
(cf. Equation \eqref{eq:hTDef}).

Let $\xi_\lambda$ minimize $h$ for fixed $\lambda \in \R$, i.e., $\xi_\lambda(\theta) \in \displaystyle \arg \min_{x\in [0,L] }h(\lambda, x, \theta)$, $\theta \in \Theta$.
We propose to use the heuristic threshold functions $\xi_{\lambda}(\theta)$ instead of $\xiT_\lambda$.
While the computation of $\hT(\lambda, x, \theta)$ as per Equation \eqref{eq:hTDef} could pose challenges, $h(\lambda, x, \theta)$ can be calculated by Equation \eqref{eq:hDef} using the distribution function and the first two moments of $X_{1,1}$, having parameter $\theta$.
The next theorem explains that the heuristic threshold functions $\xi_\lambda(\theta)$ can be simply calculated using conventional search methods. This allows us to approximate the optimal cost-rate $g^*$ with the approximated optimal cost-rate $\gamma^*$ through an iterative approach similar to Theorem \ref{thm:lambda*}.

Define $Dh:\R \times [0,L] \times \Theta \to \R$ as
\begin{multline}
    \label{eq:Dh}
    Dh(\lambda,x,\theta) = \\
    {\small
    \begin{dcases}
        c_f \left(1 - F_X(\phi(x) \mid\theta)\right) + \E \left[\ell \left(x+X_{1,1} \right) \mid \theta_1=\theta  \right] - \lambda, & \textit{if $F_x$ is continuous,}\\
        c_f d \left(1 - F_X \left(d \Bigg\lfloor \frac{\phi(x)}{d}-1 \Bigg\rfloor \Bigg \vert \theta \right)\right)+ \E \left[\ell \left(d \Bigg \lfloor \frac{x+X_{1 \,1}}{d} -1\Bigg \rfloor \right) \Bigg \vert \theta_1=\theta \right] - \lambda, & \textit{if $F_x$ is $d$-lattice}.
    \end{dcases}
    }
\end{multline}
$Dh$ represents the rate of change of $h$ when $x$ is increased by an infinitesimal amount if $F_X$ is continuous, or by one unit if $F_X$ is $d-$lattice.
For any fixed $\lambda \in \R$ and $\theta \in \Theta$, let $\mathcal{D}_{\lambda,\theta}$ denote the set of the roots of $Dh(\lambda,x,\theta)$, i.e.,
\[
    \mathcal{D}_{\lambda,\theta}:= \begin{dcases}
        \left\{x \in [0,L]: Dh(\lambda,x,\theta)=0 \right\}, & \textit{if $F_x$ is continuous,}\\
        \left\{jd: j \in \N, j \leq \dfrac{L}{d}, Dh(\lambda,x,\theta)=0 \right\}, & \textit{if $F_x$ is $d$-lattice}.
    \end{dcases}
\]
If $Dh(\lambda,x,\theta)$ in strictly negative for all values of $x \in [0,L]$, then $\mathcal{D}_{\lambda,\theta}:=\{L\}$. Conversely, if $Dh(\lambda,x,\theta)$ is strictly positive for all $\mathcal{D}_{\lambda,\theta}:=\{0\}$.
Clearly, $\mathcal{D}_{\lambda,\theta}$ is a closed and bounded set, and therefore it admits a well-defined minimum element.

\begin{theorem}
    \label{thm:gamma*}
    \begin{enumerate}[(a)]
        \item \label{thm:gamma*_FixedPoint} $\gamma^*$ is the unique solution to the equation $\lambda = \gamma(\xi_\lambda)$.
        \item \label{thm:gamma*_algorithm} Choose $\lambda_1 \in \R$, and $\lambda_{j+1} = \gamma(\xi_{\lambda_j}), j=1,2,\dots$ Then
        \[
            \lim_{j \to \infty} \lambda_j = \gamma^*.
        \]
        \item \label{thm:gamma*_threshold} $Dh(\lambda, x, \theta)$ is non-decreasing in $x$; and $\xi_\lambda(\theta)$ is given by
        \[
            \xi_\lambda(\theta) = \\
            {
            \begin{dcases}
                \min \left\{x \in [0,L]: x \in \mathcal{D}_{\lambda,\theta} \right\}, & \textit{if $F_x$ is continuous,}\\
                \min \left\{jd: j\in \N,j\leq \frac{L}{d}, jd \in \mathcal{D}_{\lambda,\theta} \right\}, & \textit{if $F_x$ is $d$-lattice}.
            \end{dcases}
            }
        \]
    \end{enumerate}
\end{theorem}
By Theorem \ref{thm:gamma*}, the heuristic threshold $\xi_\lambda(\theta)$ can be obtained through the application of any root finding technique, such as bisection, on $Dh$.
Consequently, we can efficiently compute the approximated optimal cost-rate $\gamma^*$ via the iterative procedure stated in Part \ref{thm:gamma*_algorithm} of Theorem \ref{thm:gamma*}.
The resulting approximated optimal threshold function, $\xia(\theta) :=\xi_{\gamma^*} (\theta)$, $\theta \in \Theta$, minimizes $\gamma(\xi)$.
Finally, the integrals over the parameter space in the numerator and denominator of Equation \eqref{eq:gamma} can be efficiently evaluated using numerical techniques, such as Monte Carlo or Gaussian quadrature methods, once $\xi_\lambda$ is specified.
\begin{remark}
    When $\ell(x)$ is constant for all $x \in \R_+$, the computation of $\phi(\xia_\lambda(\theta))=L-\xi_\lambda(\theta)$ reduces to identifying a specific quantile of $F_X(\cdot \mid\theta)$, similar to the classical newsvendor problem. In particular, for continuous degradation processes we have that 
    \begin{equation}
        \label{eq:NewsVendor}
        \phi(\xia_\lambda(\theta))=F^{-1}_X(1-\lambda/c_f \mid \theta),
    \end{equation}
\end{remark}
except in the rare case that $F^{-1}_X(1-\lambda/c_f \mid \theta)\geq L$ in which case it is capped there.
Now, we propose our Estimated Oracle's Optimal Policy (EOP) with its adaptive thresholds expressed as
\begin{equation}
    \label{eq:adaptivethreshold}
    \xip \left( \Ib_\tau \right):=\xia \left( \thetaH (\Ib_\tau) \right)
\end{equation}
% \begin{subequations}
%     \label{eq:adaptivethreshold}
%     \begin{alignat}{2}
%         &\xip \left( \Ib_\tau \right):=\xia \left( \thetaH (\Ib_\tau) \right) && \qquad \textit{if $\theta$ is estimated using a point estimator,} \label{eq:adaptivethresholdPoint}\\
%         &\xib \left( \Ib_\tau \right):=\displaystyle \int_\Theta \xia \left( \yb \right)dF_\theta(\yb \mid \Ib_\tau) && \qquad \textit{if $\theta$ is estimated using a posterior.}\label{eq:adaptivethresholdBayesian}
%     \end{alignat}
% \end{subequations}
% \[
%     \xiH_{\tau} \left( \theta \right):= \E \left[ \xio \left( \thetaH (\Ib_\tau) \right) \Big \vert \I_{t_0+\tau} \right].
% \]
% It is worth mentioning that in \eqref{eq:adaptivethresholdBayesian}, the integral can be calculated efficiently using numerical methods.
% From a Bayesian perspective, $\xip$ can be viewed as the Bayes estimator of $\xia(\theta)$ under the mean squared error loss.
% We refer to both $\xip(\Ib_\tau)$ and $\xib(\Ib_\tau)$ collectively as our adaptive replacement thresholds, denoted by $\xiH_{\tau}$, unless it is necessary to distinguish between the two. 
In summary, EOP thresholds can be computed as follows:
\begin{enumerate}
    \item Compute $\gamma^*$ using $F_\theta(\cdot \mid \Ib_0)$, as specified in Theorem \ref{thm:gamma*},
    \item For a component aged $\tau$ compute $\xip (\Ib_\tau)$, as given in Equation \eqref{eq:adaptivethreshold},
    \item Continue operating the component while $S_{\tau} \leq \xip (\Ib_\tau)$; replace otherwise.
\end{enumerate}
EOP is broadly applicable across various degradation and estimation processes, as its parameter estimation and decision making steps are decoupled and operate independently.
For practical implementation, the policy can be explained to practitioners in terms of an estimator of the safety margin around the failure level, defined by $\phi \left(\xip (\Ib_\tau)\right)$.
This estimator provides an indicator of how close the system is to its failure level, so the decision maker can take informed replacement actions while accounting for uncertainty in both the degradation and parameter estimates.

\begin{example}[Compound Poisson Degradation, continued]
\label{ex:exampleDef_1}
We continue with the compound Poisson degradation model from Example \ref{ex:exampleDef}  to illustrate the practical implementation of our adaptive threshold policies.
We first specify a functional form for $\ell$ and derive the corresponding expression for $r$ using Equation \eqref{eq:r_ell}.
Next, for each example, we demonstrate how to obtain closed-form expressions for the distribution function and expectation of $Y$ based on Equations \eqref{eq:YProb} and \eqref{eq:EY}, as well as for $\xi_\lambda$ and $\gamma^*$ according to Theorem \ref{thm:gamma*}.
Finally, we show how to determine $\xip(\Ib_\tau)$ using the posterior mean following Equation \eqref{eq:adaptivethreshold} for both examples.

% We consider a quadratic cost function for $\ell$, i.e., $\ell(x)=a_2x^2+a_1x$, which implies that
% \begin{equation}
%     \label{eq:ell_quadratic}
%     \E[\ell(x+X_{1,1})\mid \theta_1=\theta]= a_2 x^2+ (2a_2\mu(\theta)+a_1)x + a_2 (\sigma^2(\theta) + \mu^2(\theta))+a_1\mu(\theta)
% \end{equation}
% In this case, from Equation \eqref{eq:r_ell} it can be verified that if $F_X$ is continuous
% \begin{equation}
%     \label{eq:r_ell_continuous}
%     r(x,\theta)= \frac{a_2}{3\mu(\theta)}x^3 + \frac{2a_2 \mu(\theta)+a_1}{2\mu(\theta)} x^2+ \frac{a_2 \left(\sigma^2(\theta)+\mu^2(\theta)\right) + a_1\mu(\theta)}{\mu(\theta)} \left(x+\mu(\theta) \right),
% \end{equation}
% and if $F_X$ is $d-$lattice we have
% \begin{alignat}{2}
%     \label{eq:r_ell_lattice}
%     r(x,\theta)= &
%     \frac{a_2d}{6\mu(\theta)} \left(x^\prime-1\right) x^\prime \left(2x^\prime-1\right)+ \frac{(2a_2\mu(\theta)+a_1)d}{2\mu(\theta)}(x^\prime-1)x^\prime \nonumber\\
%     &+ \frac{\left(a_2 \left(\sigma^2(\theta)+\mu^2(\theta) \right) + a_1 \mu(\theta) \right)d}{\mu(\theta)} \left(x^\prime+\frac{\mu(\theta)}{d}\right),
% \end{alignat}
% where $x^\prime = x/d$ and $x$ is a multiple of d.

We first describe the computation of $\xi_\lambda(\nu,p)$ based on Theorem \ref{thm:gamma*}\ref{thm:gamma*_threshold}.
Let $F_{NB}(\cdot \mid l,p)$ denote the distribution function of a negative binomial random variable with parameters $l$ and $p$. Then
\begin{alignat}{2}
    \label{eq:cdf_CP_geom}
    F_X(x \mid \nu, p) = & \sum_{j=0}^{x} f_X(j \mid \nu, p) = f_M(0 \mid \nu) + \sum_{j=0}^{x} \sum_{l=1}^\infty f_M \left(l\mid \nu \right) f_{NB}(j\mid l,p) \nonumber\\
    = & f_M(0 \mid \nu)+\sum_{l=1}^\infty f_M \left(l\mid \nu \right) F_{NB} (x \mid l,p).
\end{alignat}
We use Theorem \ref{thm:gamma*}\ref{thm:gamma*_threshold} together with Equation \eqref{eq:cdf_CP_geom} to derive a relation for $\xi_\lambda$.
% \begin{alignat}{2}
%     \label{eq:xi_lambda_CP_geom}
%     \xi_\lambda(\nu,p)= &\min_{x \in \{0,\dots,L\}} x\\
%     \textit{s.t.} \quad &  a_2 x^2+ \left( \frac{2a_2\nu(1-p)}{p} +a_1 \right)x - c_f \sum_{l=1}^\infty f_M \left(l\mid \nu \right) F_{NB} (x-1 \mid l,p) \nonumber\\
%     &= \lambda -c_f (1 - f_M(0 \mid \nu)) -\frac{a_2\nu(1-p)(\nu(1-p)+2-p)}{p^2} - \frac{a_1\nu(1-p)}{p} \nonumber
% \end{alignat}
% The optimization problem \eqref{eq:xi_lambda_CP_geom} can be solved simply using the bisection method.
% As discussed earlier, if Equation \eqref{eq:xi_lambda_CP_geom} does not admit a solution, $\xi_\lambda(\nu,p)$ takes its value at one of the boundary points, i.e., $\xi_\lambda(\nu,p) \in \{0,L\}$.

Next, we describe the calculation of $\gamma^*$ using Theorem \ref{thm:gamma*}\ref{thm:gamma*_algorithm}. The key steps are to evaluate $\P\{Y(\nu,p)\leq x\}$, $\E[Y(\nu,p)]$, and $r(x,\nu,p)$ that allow the calculation of $\gamma(\xi)$ according to Equation \eqref{eq:gamma}.
By Equation \eqref{eq:YProb}, when $F_X$ is $1-$lattice and $x \in \N$, we have
\begin{alignat}{2}
    \label{eq:YProb1Lattice}
    \P\{Y(\nu,p)\leq x\}= & \frac{1}{\mu(\theta)} \sum_{j=0}^{x-1}(1-F_X(j\mid\theta)) =  \frac{1}{\mu(\theta)} \sum_{j=0}^{x-1} \left(1- \sum_{l=0}^j f_X(l \mid \theta) \right) \nonumber\\
    = & \frac{1}{\mu(\theta)} \left(x- \sum_{j=0}^{x-1} (x-j)f_X(j \mid \theta) \right) = \frac{1}{\mu(\theta)} \left(x- x F_X(x-1 \mid \theta)+ \sum_{j=0}^{x-1} jf_X(j \mid \theta) \right).
\end{alignat}
We use Equation \eqref{eq:pmf_CP_geom} to derive a relation for the last term in \eqref{eq:YProb1Lattice}. Let $\mu_{NB}(l,p):= \dfrac{l(1-p)}{p}$ denote the mean of a negative binomial random variable with parameters $l$ and $p$, then
\begin{alignat}{2}
    \label{eq:PartialMean_CP_geom}
    \sum_{j=0}^{x-1} jf_X(j \mid \nu, p) = & \sum_{j=1}^{x-1} \sum_{l=1}^\infty f_M \left(l\mid \nu \right)jf_{NB}(j\mid l,p)
    = \sum_{j=1}^{x-1} \sum_{l=1}^\infty f_M \left(l\mid \nu \right)j \binom{j+l-1}{j} (1-p)^j p^l \nonumber\\
    = &\sum_{j=1}^{x-1} \sum_{l=1}^\infty f_M \left(l\mid \nu \right) \frac{l(1-p)}{p} \binom{j+l-1}{j-1} (1-p)^{j-1} p^{l+1} \nonumber\\
    = & \sum_{j=1}^{x-1} \sum_{l=1}^\infty f_M \left(l\mid \nu \right) \mu_{NB}(l,p) f_{NB}(j-1\mid l+1,p) \nonumber\\
    = & \sum_{l=1}^\infty f_M \left(l\mid \nu \right) \mu_{NB}(l,p) F_{NB}(x-2\mid l+1,p)
\end{alignat}
Combining \eqref{eq:CP_moments_geom}, \eqref{eq:YProb1Lattice}, \eqref{eq:cdf_CP_geom}, and \eqref{eq:PartialMean_CP_geom} gives us
\begin{alignat}{2}
    \label{eq:YDist_CP_geom}
    &\P\{Y(\nu,p)\leq x\}= \nonumber\\
    &\frac{p}{\nu(1-p)} \left(x \left(1-f_M(0 \mid \nu)\right) + \sum_{l=1}^\infty f_M(l \mid \nu) \Big(\mu_{NB}(l,p)F_{NB}(x-2\mid l+1,p) - xF_{NB} (x-1 \mid l,p) \Big)\right)
\end{alignat}
By Equations \eqref{eq:CP_moments_geom} and \eqref{eq:EY}
\begin{equation}
    \label{eq:EY_CP_geom}
    E[Y(\nu,p)] = \frac{(\nu (1-p)/p)^2 + \nu(1-p)(2-p)/p^2}{2\nu (1-p)/p} +\frac{1}{2} = \frac{\nu(1-p)+2}{2p}.
\end{equation}
% By \eqref{eq:r_ell_lattice} and \eqref{eq:CP_moments_geom} $r(x,\nu,p)$ is given by
% \begin{alignat}{2}
%     \label{eq:r_ell_CP_geom}
%     r(x,\nu,p)= &
%     \frac{a_2p}{6\nu(1-p)} \left(x-1\right) x \left(2x-1\right)+ \frac{2a_2\nu(1-p)+a_1p}{2\nu(1-p)}(x-1)x \nonumber\\
%     &+ \frac{a_2 \left(\nu(1-p)+2-p \right) + a_1 p }{p} \left(x+\frac{\nu(1-p)}{p}\right).
% \end{alignat}
Using \eqref{eq:gamma} together with \eqref{eq:YDist_CP_geom} and \eqref{eq:EY_CP_geom}, one can compute $\gamma(\xi)$ efficiently.
Subsequently, $\gamma^*$ is determined using the iterative algorithm described in Theorem \ref{thm:gamma*}\ref{thm:gamma*_algorithm}.
Once $\gamma^*$ is computed, $\xia(\nu,p)= \xi_{\gamma^*}(\nu,p)$ can be calculated using Theorem \ref{thm:gamma*}\ref{thm:gamma*_threshold}.

Next, we describe how to compute the EOP threshold using the available degradation data $\Ib_\tau=(S_\tau, M_\tau,\tau)$ for the first component, where the decision epoch coincides with the component’s age.

We use Equation \eqref{eq:hyperParameter_update} to update the posterior and use the posterior mean as the point estimator of the parameter, that is
\begin{alignat}{2}
    \label{eq:posteriorMean_CP_geom}
    \thetaH(\Ib_\tau)= & \left(\hat{\nu}(\Ib_\tau), \thetaH_Z(\Ib_\tau) \right)= \left(\frac{\alpha_\tau}{\beta_\tau}, \frac{a_\tau}{a_\tau+b_\tau} \right) \nonumber\\
    =& \left(\frac{\alpha_{\tau-1}+M_\tau}{\beta_{\tau-1}+1}, \frac{a_{\tau-1}+M_\tau}{a_{\tau-1}+b_{\tau-1}+M_\tau+X_\tau} \right)= \left(\frac{\alpha_0+\sum_{j=1}^\tau M_j}{\beta_0+\tau}, \frac{a_0+\sum_{j=1}^\tau M_j}{a_0+b_0+\sum_{j=1}^\tau M_j+S_\tau} \right),
\end{alignat}
and
\[
   \xip(\Ib_\tau)= \xia  \left(\hat{\nu}(\Ib_\tau), \thetaH_Z(\Ib_\tau) \right).
\]
% Additionally, $\xiH^b_t$ can be calculated using the updated posterior as follows.
% Similar to Section \ref{subsubsection:exampleDef}, let $f_\nu (\cdot \mid \alpha,\beta)$ denote the density of a gamma distribution with shape parameter $\alpha$ and rate parameter $\beta$, and let $f_{\theta_Z}(\cdot \mid a,b)$ denote the density of a beta distribution with parameters $a$ and $b$.
% Then
% \[
%     \xiH^b_t(\nu,p) = \int_0^1 \int_0^\infty \xi_{\gamma^*}(x,y) f_\nu(x \mid \alpha_t,\beta_t) f_{\theta_Z}(y \mid a_t,b_t) dx dy,
% \]
% where the updated hyperparameters are given by Equation \eqref{eq:hyperParameter_update}.
% This double integral can be evaluated efficiently using numerical methods such as Monte Carlo approximation.
$\hfill \blacklozenge$
\end{example}

\subsection{Scaling Regime and Main Asymptotic Optimality Results}
Our analysis so far provides a tractable policy that fully exploits the available data for decision making in an otherwise intractable maintenance problem.
Beyond tractability, a key performance measure of any policy is its optimality gap, which we examine in the arguments that follow.
We restrict our study to large failure thresholds for several reasons.
First, as the average time to failure increases, exact methods quickly become intractable, and heuristics such as the one we propose become particularly attractive due to their tractability advantages.
Second, from a practical perspective, new equipment often operates for a relatively long time before failing.
Finally, in terms of learning effectiveness, short failure times leave little opportunity for accurate parameter estimation, which lead to high estimation errors.
In such cases, the problem can be reduced to a homogeneous population, as posterior distributions just before replacements are nearly identical across components.

Next, we study the EOP thresholds $\xip(\Ib_\tau)$ under the scaling regime described in Section \ref{subsection:Outline}, and formally state our main result that $\reg_k(\xip(\Ib_\tau))$ vanishes as the scaling parameter increases.
% We will show that the Oracle's optimal thresholds $\xio (\theta)$ become large, in this regime, for all $\theta \in \Theta$, while the corresponding optimal cost-rates remain bounded away from both zero and infinity.
% We will also prove that $\xia$ is asymptotically optimal for arbitrarily large $k$ which will be formally presented in the next section.
Let, for $k>0$, $\Xi_k$ be the set of all threshold functions $\xi:\Theta \to [0,L(k)]$. Define $g_k: \Xi_k \to \R_+$ and $\gamma_k: \Xi_k \to \R_+$ by
\begin{equation}
    \label{eq:g_k}
    g_k \left(\xi \right):= \frac{c_p(k) + \displaystyle\int_\Theta \left(c_f(k) \P \left\{\YT_1 \left(\xi(\theta_1) \right) > \phi_k \left(\xi(\theta_1) \right) \Bigg \vert \theta_1=\yb \right\} + \E \left[\sum_{\tau=1}^{T \left(\xi(\theta_1),\theta_1 \right)} \ell(S_{1,\tau},k) \Bigg\vert \theta_1=\yb \right]  \right) dF_\theta(\yb)}{\displaystyle \int_\Theta \frac{1}{\mu(\yb)} \left(\xi(\yb) + \E \left[\YT_1 \left(\xi(\theta_1) \right) \Bigg \vert \theta_1=\yb \right] \right) dF_\theta(\yb)}
\end{equation}
and
\begin{equation}
    \label{eq:gamma_k}
    \gamma_k \left(\xi \right):= \frac{c_p(k) + \displaystyle \int_\Theta \left(c_f(k) \P \left\{Y(\yb) > \phi_k \left(\xi(\yb) \right) \right\} + r_k\left(\xi, \yb \right) \right) dF_\theta(\yb)}{\displaystyle \int_\Theta \frac{1}{\mu(\yb)} \left(\xi(\yb) + \E \left[Y(\yb) \right] \right) dF_\theta(\yb)},
\end{equation}
where $\phi_k(x):= L(k) - x$, $x \in [0,L(k)]$, and
$r_k:[0,L(k)] \times \Theta \to \R_+$ as
\begin{equation}
    \label{eq:rkDef}
    r_k (x,\theta) = \begin{cases} \E \left[\ell(X_{1,1},k)\mid \theta_1=\theta \right]+
        \dfrac{1}{\mu(\theta)}\displaystyle\int_{0}^{x} \E \left[\ell(X_{1,1}+y,k)\mid \theta_1=\theta \right] dy, & \textit{if $F_x$ is continuous,} \\
        \E \left[\ell(X_{1,1},k)\mid \theta_1=\theta \right]+ \dfrac{d}{\mu(\theta)}\displaystyle\sum_{j=0}^{\big\lfloor x/d \big\rfloor}  \E \left[\ell(X_{1,1}+jd,k) \mid \theta_1=\theta \right], & \textit{if $F_x$ is $d$-lattice}. \\
    \end{cases}
\end{equation}
Define the minimizers of $g_k$ and $\gamma_k$ by
\[
    \xio_k \in \arg \min\limits_{\xi \in \Xi_k} g_k \left(\xi \right), \qquad \mbox{and}, \qquad \xia_k \in \arg \min\limits_{\xi \in \Xi_k} \gamma_k \left(\xi \right),
\]
Then our EOP thresholds in the scaling regime are defined as
\begin{equation}
    \label{eq:xiH_k}
    \xip_k(\Ib_\tau) :=\xia_k \left( \thetaH (\Ib_\tau) \right).
\end{equation}
% For a component with parameter $\theta$, let $\Tcal_k(\theta)$ be a random variable that denotes the age of the operating component at replacement under the sequence of EOP thresholds $\left\{\xip_k(\Ib_\tau) \right\}_{\tau=1}^{\Tcal_k(\theta)}$.
% We denote by $\xipklast(\theta):= \xip_k \left(\Ib_{\Tcal_k(\theta)} \right)$ the last EOP threshold by which the replacement actually occurs.
% Observe that $\xipklast(\theta)$ is a random variable whose distribution is determined by the true parameter $\theta$.
% Define $\xipklast$ as the set of all such $\xipklast(\theta)$, i.e., $\xipklast:= \left \{\xipklast(\theta): \theta \in \Theta \right\}$.
% The cost-rate of EOP policies under the scaling parameter $k>0$ is a random variable determined by $ g_k \left(\xipklast \right)$.
% Specifically, $\xip_{(k)}(\theta)$ is the last adaptive threshold when decision making is based on \eqref{eq:xiH_kPoint}, and $\xiH^b_{(k)}(\theta)$ is the last adaptive threshold when decision making is based on \eqref{eq:xiH_kBayesian}.
% Let $m \in \{p, a.s.\}$ denote the mode of consistency of the point estimator, that is, $\thetaH_\tau \xrightarrow{m} \theta$.
Next, we present the main theoretical result of this paper.
\begin{theorem}
    \label{thm:costconvergancewithlearning}
    As $k \to \infty$, the cost-rate of the EOP thresholds converge to the cost-rate of the Oracle's optimal policy, that is
    \[
        \lim_{k\to \infty} g_k \left(\xip_k \right) - g_k \left(\xio_k \right) = \lim_{k \to \infty} \reg(\xip_k) =0.
        % g_k \left(\xipklast \right) \xrightarrow{m} g_k \left(\xio_k \right).
        % \qquad \textit{and,} \qquad \gamma_k \left(\xipklast \right) \xrightarrow{m} \gamma_k \left(\xia_k \right).
    \]
\end{theorem}
% \begin{theorem}
%     \label{thm:costconvergancewithlearning}
%     As $k \to \infty$
%     \begin{enumerate}[(a)]
%         \item \label{thm:costconvergancewithlearning_pointEstimator} $g_k \left(\xip_{(k)} \right) \xrightarrow{m} g_k \left(\xio_k \right)$, and $\gamma_k \left(\xip_{(k)} \right) \xrightarrow{m} \gamma_k \left(\xia_k \right)$;
%         \item \label{thm:costconvergancewithlearning_posterior} $g_k \left(\xiH^b_{(k)} \right) \xrightarrow{a.s.} g_k \left(\xio_k \right)$, and $\gamma_k \left(\xiH^b_{(k)} \right) \xrightarrow{a.s.} \gamma_k \left(\xia_k \right)$.
%     \end{enumerate}
% \end{theorem}
Theorem \ref{thm:costconvergancewithlearning} establishes that the cost-rates of our EOP thresholds are asymptotically optimal as $k \to \infty$.
% In particular, the cost-rate achieved by a decision maker who relies on point estimates of the true parameter $\theta$ converges to that of an Oracle with full information of $\theta$, in the same sense as the consistency of the estimator for $\theta$.
% When the decision maker uses the posterior distribution to compute $\xiH^b$ , the convergence is almost surely at the cost of additional but still tractable computational effort.

\section{Asymptotic Analysis}
\label{section:AsymptoticAnalysis}
% In this section, we present the argument for an operating components which is installed at epoch $t_0 \in \N_0$.
% We suppress the index of the component and $t_0$, for notational convenience, following the argument in the first paragraph of Section \ref{subsection:EOP}.
We next outline the main steps of the proof of Theorem \ref{thm:costconvergancewithlearning}.
\begin{enumerate}
    \item We first show that, under the scaling regime, both $\xio_k(\theta)$ and $\xia_k(\theta)$ diverge to infinity for all $\theta \in \Theta$ as $k \to \infty$, while the cost-rate function $g_k(\xi)$ and its approximation $\gamma_k(\xi)$ remain bounded for all threshold functions $\xi$.
    We further show that, as $k$ increases, the gap between $g_k(\xio_k)$ and $g_k(\xia_k)$ vanishes, i.e., $ g_k(\xia_k) - g_k(\xio_k) \to 0$.
    The proof proceeds by establishing the convergence of each function appearing in the expression of $g_k$ in \eqref{eq:g_k} to its corresponding counterpart in $\gamma_k$ as defined in \eqref{eq:gamma_k}.
    \item We next show that, under the EOP threshold function $\xip_k$, the lifetime of an operating component also diverges, as $k \to \infty$.
    Specifically, let $\Tcal_k (\theta_1 )$ denote the lifetime of the first component under the sequence of EOP thresholds $\left\{\xip_k(\Ib_\tau) \right\}_{\tau=1}^{\Tcal_k (\theta_1)}$, where $k$ is the scaling parameter.
    Let $\xipklast (\theta_1)$ denote the random variable corresponding to the last EOP threshold at which the component is replaced by component 2, i.e., $\xipklast (\theta_1):= \xip_k \left(\Ib_{\Tcal_k (\theta_1)} \right)$.
    We show that $\Tcal_k (\theta_1 ) \to \infty$ as $k \to \infty$, which implies that for large $k$ the decision maker's error in estimating the true parameter $\theta_1$ from the degradation data becomes negligible.
    The same argument holds for all components $i \in \N$.
    We use this result together with the (weak) consistency of the estimator to prove that the cost-rate of replacing at $\xipklast(\theta_i)$, i.e., $g(\xipklast)$, converges to $g_k(\xio_k)$ as $k$ becomes large.
\end{enumerate}
While this section presents the main asymptotic results, the full mathematical proofs are detailed in the Appendix.

\subsection{Asymptotics of the Oracle's Optimal Thresholds}
Our construction of the Oracle’s approximate optimal threshold function $\xia$ relies on the asymptotic expansion of the overshoot random variable from classical renewal theory.
In our setting, this result takes the following form:
\begin{lemma} \citep[Theorems 2.6.1 and 2.6.2][]{Gut2009}
    \label{lemma:Y_1 asymptotics}
    For all $i \in \N$ and $\theta_i \in \Theta$, as $\xi(\theta_i) \rightarrow \infty$,
    \begin{enumerate}[(a)]
        \item $\P \{\YT_i(\xi(\theta_i)) \leq x\} \rightarrow \P\{Y(\theta_i)\leq x\}$ for all $x\geq 0$, \label{eq:limitprob}
        \item $\E [\YT_i(\xi(\theta_i))] \rightarrow \E[Y(\theta_i)]$. \label{eq:limitexpect}
    \end{enumerate}
\end{lemma}
We now present the convergence results for different functions used to construct $g_k$ in \eqref{eq:g_k}.

\begin{lemma}
    \label{lemma:elementwise_convergence}
    Let for the sequence of threshold functions $\left \{\xi^{(k)} \in \Xi_k \right\}_{k>0}$, we have $\lim_{k\to \infty} \xi^{(k)}(\theta)/k >0$ for all $\theta \in \Theta$. Then, as $k \to \infty$
    \begin{enumerate}[(a)]
        \item \label{lemma:elementwise_convergence_per_theta}For all $\theta \in \Theta, $ $
            \E \left[\displaystyle\sum_{\tau=1}^{T \left(\xi^{(k)}(\theta_1) ,\theta_1 \right)} \ell(S_{1,\tau},k) \Bigg\vert \theta_1=\theta \right] - r_k \left(\xi^{(k)}(\theta) ,\theta \right)  \to 0,
        $
        \item \label{lemma:elementwise_convergence_E_Theta}$
            \displaystyle \int_\Theta\E \left[\displaystyle\sum_{\tau=1}^{T \left(\xi^{(k)}(\theta_1) ,\theta_1 \right)} \ell(S_{1,\tau},k) \Bigg\vert \theta_1=\yb \right] dF_\theta(\yb) - \int_\Theta r_k \left(\xi^{(k)}(\yb) ,\yb \right) dF_\theta(\yb)  \to 0,
        $
        \item \label{lemma:elementwise_convergence_Prob}$
            \displaystyle\int_\Theta \P \left\{\YT_1(\xi^{(k)}(\theta_1)) > \phi_k \left(\xi^{(k)}(\theta_1) \right) \Big \vert \theta_1=\yb \right\} dF_\theta(\yb) - \displaystyle\int_\Theta \P \left\{Y(\yb) > \phi_k \left(\xi^{(k)}(\yb) \right) \right\} dF_\theta(\yb) \to 0
        $
        \item \label{lemma:elementwise_convergence_CycleLength}$
             \displaystyle \frac{1}{k} \int_\Theta \frac{\xi^{(k)}(\yb) + \E \left[\YT_1(\xi^{(k)}(\theta_1)) \Big \vert \theta_1 = \yb \right]}{\mu(\yb)}  dF_\theta(\yb) -  \frac{1}{k}\displaystyle\int_\Theta \frac{\xi^{(k)}(\yb) + \E \left[Y(\yb) \right]}{\mu(\yb)}  dF_\theta(\yb) \to 0
        $
    \end{enumerate}
    
\end{lemma}
Lemma \ref{lemma:elementwise_convergence} indicates that as the scaling parameter $k$ increases, each function in \eqref{eq:g_k} converges to its counterpart in \eqref{eq:gamma_k}, when the threshold functions scale at least linearly with respect to $k$.
Intuitively, this leads to the convergence of the gap between $g_k$ and $\gamma_k$, which is formalized in the next lemma.

\begin{lemma}
    \label{lemma:gasymptotic}
    For the sequence of $\left \{\xi^{(k)} \in \Xi_k \right\}_{k>0}$ that, for all $\theta \in \Theta$, satisfies $\lim_{k\to\infty} \xi^{(k)}(\theta)/k>0$ 
    \[
        \lim_{k \to \infty} \left(g_k \left(\xi^{(k)} \right) - \gamma_k \left(\xi^{(k)} \right) \right)=0.
    \]
\end{lemma}
Lemmas \ref{lemma:elementwise_convergence} and \ref{lemma:gasymptotic} apply only to threshold functions whose growth is asymptotically linear in $k$.
A natural question is whether the minimizers of $g_k$ and $\gamma_k$, i.e.,
\[
    \xio_k \in \arg \min\limits_{\xi \in \Xi_k} g_k \left(\xi \right), \qquad \mbox{and}, \qquad \xia_k \in \arg \min\limits_{\xi \in \Xi_k} \gamma_k \left(\xi \right),
\]
fall into this class.
The following lemma establishes that both threshold functions indeed satisfy this property.

\begin{lemma}
    \label{lemma:threshold_order}
    For $\theta \in \Theta$
    \begin{enumerate}[(a)]
        \item \label{lemma:threshold_order_xi} $\lim_{k \to \infty} \xio_k(\theta)/k \in (0,\LT ]$,
        \item \label{lemma:threshold_order_xiT} $\lim_{k \to \infty} \xia_k(\theta)/k \in (0,\LT ]$,
    \end{enumerate}
\end{lemma}
We can now apply Lemma \ref{lemma:gasymptotic} and \ref{lemma:threshold_order} to establish the following convergence result for the cost-rates of $g_k$ and $\gamma_k$.
\begin{theorem}
    \label{thm:costconvergancenolearning}
    In the scaling regime, the approximated optimal threshold function $\xia$ performs as well as the optimal threshold function $\xio$ and the approximated optimal cost-rate $\gamma(\xia)$ converges to the actual optimal cost-rate $g(\xio)$, i.e. as $k \to \infty$, 
    \begin{enumerate}[(a)]
        \item $g_k (\xia_k) - g_k (\xio_k) \to 0$. \label{thm:optimalitygap}
        \item $  g_k (\xia_k) - \gamma_k (\xia_k)  \to 0$. \label{thm:predictiongap}
    \end{enumerate}
\end{theorem}
Theorem \ref{thm:costconvergancenolearning}\ref{thm:optimalitygap} shows that the optimality gap resulting from the use of the threshold function $\xia_k$ in place of the optimal threshold $\xio_k$ vanishes as $k$ increases.
Part \ref{thm:predictiongap} shows that if $g_k \left(\xia_k\right)$ is predicted by $\gamma_k \left(\xia_k \right)$, the prediction gap tends to zero for sufficiently large $k$.

\subsection{Asymptotic Analysis of the EOP Thresholds}
We begin our analysis with the first component; the same reasoning applies to any operating component.
By Lemma \ref{lemma:threshold_order} together with Equations \eqref{eq:xiH_k} one can verify that $\xip_k(\Ib_\tau)$ diverge as $k$ increases, i.e., for all $\tau\in \N_0$,
\begin{equation}
    \label{eq:xiH_k_lim}
    \lim_{k\to \infty}\xip_k(\Ib_\tau) =  \lim_{k\to \infty}\xia_k \left(\thetaH(\Ib_\tau) \right)  = \infty, \qquad \mbox{almost surely.}
\end{equation}
By combining the divergence result in \eqref{eq:xiH_k_lim} with the strong law for counting processes \citep[Theorem 2.5.1 in][]{Gut2009}, we obtain that, for the true parameter $\theta_1 \in \Theta$,
\begin{equation}
    \label{eq:Tcal_lim}
    \lim_{k \to \infty}\Tcal_k(\theta_1)=\infty \qquad \textit{almost surely.}
\end{equation}
Equation \eqref{eq:Tcal_lim} implies that the number of observations also diverges as $k$ approaches infinity.
Consequently, the errors in estimating the true parameter, $\theta_1$, by the point estimators $\thetaH(\Ib_\tau)$ vanish due to consistency.
The following lemmas state this observation in detail and establish the corresponding convergence results for the EOP thresholds, the total operating cost, and the probability of failure at replacement as $k$ becomes large.
% We define the last estimator of $\theta$ prior to replacement under the scaling parameter $k$ as $\thetaH_{(k)}:= \thetaH_{\Tcal_k\left( \theta \right)}$.
Let $\thetaH(\Ib_\tau) \xrightarrow{p} \theta_1$, for $\theta_1 \in \Theta$.

\begin{lemma}
    \label{lemma:PointEstimateConsistency}
    As $k \to \infty$,
    \begin{enumerate}[(a)]
        \item \label{lemma:PointEstimateConsistency_theta} For $\theta_1 \in \Theta$, $\thetaH \left(\Ib_{\Tcal_k(\theta_1)} \right) \xrightarrow{p} \theta_1$;
        \item \label{lemma:PointEstimateConsistency_xi} For $\theta_1 \in \Theta$, $\dfrac{1}{k} \xipklast(\theta_1) \xrightarrow{p} \dfrac{1}{k}\xia_k(\theta_1)$;
        \item \label{lemma:PointEstimateConsistency_Exi} $\dfrac{1}{k} \displaystyle \int_\Theta \frac{1}{\mu(\yb)} \left(\xipklast(\yb)+ \E \left[ Y(\xipklast \left(\yb) \right)  \right]\right)dF_\theta(\yb) \xrightarrow{p} \dfrac{1}{k} \displaystyle \int_\Theta \frac{1}{\mu(\yb)} \left(\xia_k \left(\yb \right) +\E \left[Y \left(\xia_k \left(\yb \right) \right) \right]\right) dF_\theta(\yb)$;
        \item \label{lemma:PointEstimateConsistency_ell} $\dfrac{1}{k} \displaystyle \int_\Theta r_k \left(\xipklast(\yb), \yb \right) dF_\theta(\yb) \xrightarrow{p} \dfrac{1}{k} \displaystyle \int_\Theta r_k \left(\xia_k(\yb) , \yb \right) dF_\theta(\yb)$;
        \item \label{lemma:PointEstimateConsistency_prob} $\displaystyle \int_\Theta \P \left\{Y(\yb)>\phi_k \left(\xipklast (\yb) \right) \right\} dF_\theta(\yb) \xrightarrow{a.s.} \displaystyle \int_\Theta \P \left\{Y(\yb)>\phi_k \left(\xia_k(\yb) \right) \right\} dF_\theta(\yb)$.
    \end{enumerate}
\end{lemma}
Lemma \ref{lemma:PointEstimateConsistency} leads to Theorem \ref{thm:costconvergancewithlearning}.
% The next lemma presents results similar to Lemma \ref{lemma:PointEstimateConsistency} when $\xiH^b_{(k)}$ is concerned.
% \begin{lemma}
%     \label{lemma:PosteriorEstimateConsistency}
%     As $k \to \infty$,
%     \begin{enumerate}[(a)]
%         \item \label{lemma:PosteriorEstimateConsistency_xi} For all $\theta \in \Theta$, $\dfrac{1}{k} \xiH^b_{(k)}(\theta) \xrightarrow{a.s.} \dfrac{1}{k}\xi^*_k(\theta)$;
%         \item \label{lemma:PosteriorEstimateConsistency_Exi} $\dfrac{1}{k} \displaystyle \int_\Theta \frac{1}{\mu(\yb)} \xiH^b_{(k)}(\yb) dF_\theta(\yb) \xrightarrow{a.s.} \dfrac{1}{k} \displaystyle \int_\Theta \xi^*_k(\yb) dF_\theta(\yb)$;
%         \item \label{lemma:PosteriorEstimateConsistency_ell} $\dfrac{1}{k} \displaystyle \int_\Theta \frac{1}{\mu(\yb)} r_k \left(\xiH^b_{(k)}, \yb \right) dF_\theta(\yb) \xrightarrow{a.s.} \dfrac{1}{k} \displaystyle \int_\Theta r_k \left(\xi^*_k , \yb \right) dF_\theta(\yb)$;
%         \item \label{lemma:PosteriorEstimateConsistency_prob} $\displaystyle \int_\Theta \P \left\{Y(\yb)>\phi_k \left(\xiH^b_{(k)} (\yb) \right) \right\} dF_\theta(\yb) \xrightarrow{a.s.} \displaystyle \int_\Theta \P \left\{Y(\yb)>\phi_k \left(\xi^*_k(\yb) \right) \right\} dF_\theta(\yb)$.
%     \end{enumerate}
% \end{lemma}
% Lemma \ref{lemma:PosteriorEstimateConsistency} leads to Part \ref{thm:costconvergancewithlearning_posterior} of Theorem \ref{thm:costconvergancewithlearning}.

\section{Case Study Results}
\label{section:casestudyresults}
In this section, we evaluate the performance of our EOP thresholds using real degradation data from the case study on Interventional X-ray (IXR) systems presented in \cite{Drent2023}.
Below, we provide a brief summary of this case study and refer the interested reader to \cite{Drent2023} for a more detailed discussion.

IXR systems are advanced imaging platforms used in hospitals to support minimally invasive procedures, such as cardiac or vascular interventions.
Their most critical and costly replacement components are X-ray tubes, which gradually lose efficiency over time as the tungsten filaments inside them evaporate through high-temperature heating.
During operation, these filaments emit electrons that are directed toward a target to generate X-ray images.
Unexpected tube failures can force procedures to stop or be rescheduled, delay critical patient care, and generate substantial corrective maintenance expenses.

Filament condition monitoring is performed primarily through real-time measurements of electrical resistance.
The measurements are then processed to generate a one-dimensional degradation indicator for the filaments, which is stored in the database of the IXR system.
Major manufacturers of IXR systems, such as Philips Healthcare, require maintenance strategies that exploit these real-time measurements to prevent tube failure while maximizing their useful life.

\cite{Drent2023} provide time series degradation data for 52 filaments obtained from lab experiments by Philips Healthcare. The set of these 52 time series is denoted $\J$ and its cardinality $\vert\J \vert = 52$.
They describe how the data are preprocessed, including statistical tests to assess filament heterogeneity and to verify that electric shocks follow a Poisson process.
They also justify modeling the damage process as geometric.
Next, they show how, under this assumption, the policy adapted to $\F$ can be computed as the optimal solution of a Partially Observable Markov Decision Process (POMDP).
The resulting policy is optimal within the assumed model, though not necessarily for the specific case data.
This policy is called the Integrated Bayes Policy (IBP).
The IBP specifies adaptive replacement thresholds that depend on the information available at time $t$, denoted by $\xiI(\Ib_t)$.

The condition-based maintenance problem in this case study aligns with Example 1, for which our EOP thresholds $\xip$ can be efficiently applied.
Notice that $\xip(\Ib_t)$ is significantly more computationally efficient than $\xiI(\Ib_t)$, often by more than an order of magnitude. In fact, $\xiI(\Ib_t)$ becomes numerically intractable as the failure level increases.
The only computationally intensive step in applying $\xip(\Ib_t)$ is the calculation of $\gamma^*$ via Theorem \ref{thm:gamma*}\ref{thm:gamma*_algorithm}, which requires numerical integration. We perform this using Monte Carlo methods.
However, $\gamma^*$ needs to be computed only when the hyperparameters are updated
The runtime to update the posterior and to evaluate $\xip(\Ib_t)$ is negligible.
In contrast, solving a POMDP is NP-hard, so exact solutions might become computationally intractable for some realistic problem sizes.

In the following, we evaluate the performance of $\xip$ relative to $\xiI$, which is the current state-of-the-art strategy, using the 52 degradation time-series provided by \cite{Drent2023}.
This assessment automatically benchmarks our approach against the other state-of-the-art methods presented in \cite{Drent2023}.
We focus on cost savings percentages, defined as $\%SAV = \dfrac{\bar{g}(\xiI)-\bar{g}(\xip)}{\bar{g}(\xiI)} \times100$ where $\bar{g}(\cdot)$ denotes the average cost function.
Note that $(\alpha_0,\beta_0,a_0,b_0)$ is not known a priori in this problem.
We therefore use a sample from $\J$, called the training set, to estimate these hyperparameters via maximum likelihood estimation (MLE), as described in Appendix C of \cite{Drent2023}.
Since $\%SAV$ depends on the estimated $(\alpha_0,\beta_0,a_0,b_0)$, we generate multiple training sets from $\J$, denoted by $\J_{\scriptsize\mbox{training}}$.
We choose the sizes of the training sets $\Big\vert\J_{\scriptsize\mbox{training}}\Big\vert$ from $\{5,10,15\}$, and for each size, we randomly draw 150 subsets from $\J$, each of which provides an estimate of the hyperparameters.
This procedure gives us 450 bootstrapping instances.
For each instance, we compute $\%SAV$ based on the training set, denoted $\J_{\scriptsize\mbox{test}}:=\J \setminus\J_{\scriptsize\mbox{training}}$. In line with \cite{Drent2023}, we set $c_p$ and $c_f$ equal to 1 and 4, respectively, and $\ell(\cdot)\equiv 0$.
We compute the integrals in Equation \eqref{eq:gamma} using Monte Carlo sampling with $5\times 10^4$ samples, which gives narrow 95\% confidence intervals.
The results of our study are presented in Table \ref{tab:CaseStudy}.

\begin{table}[!htbp]
% \fontsize{8pt}{9pt}\selectfont
\scriptsize 
\centering
\caption{\textsf{Results of the case study.}}
\label{tab:CaseStudy}
\begin{tabular}{lll}
\toprule
$\Big\vert\J_{\scriptsize\mbox{training}}\Big\vert$ && $\%SAV$ \\
\midrule
5                    && -0.35   \\
10                   && 1.09 \\
15                   && 1.17 \\
\midrule
Total                && 0.63 \\
\bottomrule
\end{tabular}
\end{table}

Table \ref{tab:CaseStudy} shows that our asymptotically optimal adaptive threshold policy slightly reduces the average maintenance cost-rate by 0.63\% compared to the integrated Bayes policy across the bootstrapping instances.
The results indicate that the integrated Bayes policy performs slightly better for small training sets, i.e., when $\Big\vert\J_{\scriptsize\mbox{training}}\Big\vert=5$, whereas the asymptotically optimal adaptive threshold policy slightly outperforms the integrated Bayes policy for larger training sets, i.e., when $\Big\vert\J_{\scriptsize\mbox{training}}\Big\vert=10$ or 15.
We observe that these results remain consistent when using smaller sample sizes to compute the integrals in Equation \eqref{eq:gamma}, for example, as small as $10^4$.
These results demonstrate that our asymptotically optimal adaptive threshold policy is not only substantially faster than the integrated Bayes policy but also achieves strong cost-rate performance on real data.
From a managerial perspective, this combination of computational and cost-rate efficiency makes it highly suitable for real-time maintenance decision making.

% [TODO reflections on the result. also emphasis on how heterogeneity change replacement decisions (cases close to the 2 point distribution)]

\section{Simulation Study Results}
\label{section:Simulation}
In this section, we present the results of an extensive simulation study.
The objectives of this study are threefold.
\begin{enumerate}
    \item To evaluate the performance of the asymptotically optimal adaptive threshold policy relative to the optimal policy under the same filtration, when computation of the latter is feasible;
    \item To compare the performance of the asymptotically optimal adaptive threshold policy with the Oracle optimal policy;
    \item To analyze the optimality and prediction gaps associated with the Oracle’s asymptotically optimal policy.
\end{enumerate}
In this study, we assume that decision makers have access to the true hyperparameters $(\alpha_0, \beta_0, a_0, b_0)$.
This assumption allows us to isolate the impact of the policies of interest on the cost rates from the effect of estimation errors in the prior, which would arise if the prior were learned from historical data.
The latter effect is examined in Section~\ref{section:casestudyresults}, where we observed that the EOP performs remarkably well under such conditions.

Our study addresses two settings in terms of degradation state spaces: discrete, i.e. $1-$lattice, and continuous.
The discrete setting is characterized by a compound Poisson process with geometric compounding, whereas the continuous setting is characterized by a compound Poisson process with exponential compounding.
Example 1 details the computation of the EOP thresholds in the discrete setting, and Example 2 extends this to the continuous setting.
The optimal policy $g^*(\F)$ can be computed in a tractable manner using the integrated Bayes approach introduced in \cite{Drent2023}; however, this is feasible only for the discrete setting.
Hence, we do not compute $g^*(\F)$ for the continuous state space.

A tractable implementation of the integrated Bayes approach requires truncating both the cumulative number of shocks and the ages of the components, whereas the asymptotically optimal adaptive threshold policy does not.
This truncation, however, affects the resulting updated posteriors.
The cost-rate of the EOP is calculated under both the full observations, denoted $g(\xip)$, and the truncated observations, denoted $g(\xiptrunc)$, to assess the impact of truncation.

We are interested in the regret (relative to the Oracle) of the adaptive policies $\xi \in \left\{\xip, \xiptrunc, \xiI \right\}$, defined as $\%\reg(\xi):= 100\cdot\left(g(\xi)-g^*(\O) \right)/g^*(\O)$, which serves as a common performance metric across all policies and both discrete and continuous settings.
As noted earlier, $\%\reg \left( {\xiI} \right)$ is only available for discrete state space.
We are also interested in the optimality gap and prediction gap of $\xia$ with respect to $\xio$, defined as $\%OPT := 100\cdot\left(g \left(\xia \right) - g^*(\O) \right)/g^*(\O)$ and $\%PRED := 100\cdot\left(\gamma^*-g \left(\xia \right) \right)/g \left(\xia \right)$, respectively.
$\%OPT$ measures the relative additional cost incurred by using $\xia$ instead of $\xio$ as the Oracle's optimal decision, while
$\%PRED$ quantifies how far the approximated Oracle’s optimal cost-rate $\gamma^*$ deviates from the true value $g \left(\xia \right)$.

We construct a comprehensive testbed for our simulation study, shown in Table \ref{tab:testbed}, and conduct experiments across all combinations of the parameter values, with $c_p = 1$.
The shock arrival process follows a gamma prior with $\alpha_0 = \beta_0$, yielding a mean of 1.
For instances with a discrete state space, the damage process follows a Beta prior, whereas for instances with a continuous state space, it follows a Gamma prior, with $a_0 = b_0$ in both cases, which gives means of 0.5 and 1, respectively.
We vary the coefficients of variation to assess the impact of the degree of heterogeneity across instances.
We consider three operating cost functions $\ell_0(x)=0$, $\ell_1(x)= \dfrac{c_p \min\{x,L\}}{L^2/2+L+1}$, and $\ell_2(x)= \dfrac{c_p (\min\{x,L\})^2}{L^3/3+L^2+L+1}$.
The selection of $\ell_1$ and $\ell_2$ is such that the expected total operating cost at failure is approximately the same value as $c_p$ when $\mu(\theta)=1$, and $\sigma^2(\theta)$ is negligible.

\begin{table*}[!htbp]
  \caption{\textsf{Input parameters for numerical studies.}}
	\fontsize{8pt}{9pt}\selectfont
\scriptsize
  \label{tab:testbed}
  \begin{tabularx}{\textwidth}{l X l l }
    \toprule
    % \hline
    & Input parameter & No. of Choices & Values\\
		\midrule
    1 & Coefficient of variation for the prior of the shock arrival process, $cv^2_\nu$ & 2 & $0.3$, $0.6$\\
    2 & Coefficient of variation for the prior of the damage process, $cv^2_{\theta_Z}$ & 2 & 0.04, $0.08$\\
    3 & Failure threshold, $L$, & 3 & 10,20,30\\
    4 & Additional corrective maintenance cost, $c_f$ & 2 & 4,9\\
    5 & Operating cost, $\ell$ & 3 & $\ell_0$, $\ell_1$, $\ell_2$\\
    \bottomrule
  \end{tabularx}
\end{table*}

We compute $\gamma^*$ using Theorem \ref{thm:gamma*}\ref{thm:gamma*_algorithm} with Monte Carlo integration using $10^5$ samples per integral.
The value of $g^*$ is obtained from Theorem \ref{thm:lambda*}\ref{thm:lambda*_algorithm}, with Monte Carlo integration using $10^5$ samples per integral and $4 \times10^3$ simulated components per expectation function.
This ensures that the confidence intervals for the computed $\gamma^*$ and $g^*$ remain small.
We simulate $10^4$ components for each policy $\xi \in \{\xip, \xiptrunc, \xiI\}$ and compute the cost-rate across all components.
Each simulation is repeated 10 times to keep the confidence intervals for the performance metric small.
The policy $\xiI$ is obtained by solving the optimality Equations (10) in \cite{Drent2023}, modified to account for the operating cost, with a discount factor of 0.99 and a truncation of the number of shocks and components’ age at 40, 50, and 60 when the failure level is 10, 20, and 30, respectively.

The results of the simulation study are reported in Tables \ref{tab:simulation_discrete} and \ref{tab:simulation_continuous}.
\begin{table*}[!htbp]
  \fontsize{8pt}{9pt}\selectfont
  \scriptsize
  \parbox[t]{.9\linewidth}{
  \centering
  \caption{\textsf{Results of the simulation study - discrete state space}}
  \label{tab:simulation_discrete}
  \begin{tabular}{l l l l l l l l l l l l l l l l l l l l l l l }
        \toprule
        &&&& \multicolumn{3}{c}{$\%\reg \left(\xip \right)$} && \multicolumn{3}{c}{$\%\reg \left(\xiptrunc\right)$} && \multicolumn{3}{c}{$\%\reg \left(\xiI \right)$} && \multicolumn{3}{c}{$\%OPT$} && \multicolumn{3}{c}{$\%PRED$}\\
        \cmidrule{5-7} \cmidrule{9-11} \cmidrule{13-15} \cmidrule{17-19} \cmidrule{21-23}
        Input && Values && Min & Mean & Max && Min & Mean & Max && Min & Mean & Max && Min & Mean & Max && Min & Mean & Max \\ \midrule
        \multirow{3}{*}{$L$} && 10 && 3.76 & 11.05 & 24.93 &  & 3.19 & 11.54 & 24.20  &  & 2.69 & 11.16 & 26.75 &  & 0.03 & 3.03 & 8.69 &  & -13.88 & -8.04 & -0.11 \\ 
        && 20 && 1.14 & 5.68  & 15.53 &  & 0.85 & 6.19  & 16.24 &  & 0.97 & 6.17  & 16.97 &  & 0.02 & 1.44 & 4.53 &  & -8.47  & -4.87 & -0.27 \\ 
        && 30 && 0.47 & 3.06  & 9.55  &  & 0.35 & 3.50   & 10.46 &  & 1.06 & 4.00     & 11.31 &  & 0.03 & 0.90  & 3.04 &  & -5.79  & -3.37 & 0.24  \\ \\
        \multirow{2}{*}{$cv^2_\nu$} && 0.3 && 0.75 & 7.55  & 24.93 &  & 1.27 & 7.61  & 22.83 &  & 1.06 & 6.82  & 21.89 &  & 0.14 & 2.04 & 8.69 &  & -11.12 & -4.64 & -0.09 \\
        && 0.6  && 0.47 & 5.64  & 19.53 &  & 0.35 & 6.55  & 24.20  &  & 0.97 & 7.40   & 26.75 &  & 0.02 & 1.54 & 6.14 &  & -13.88 & -6.21 & 0.24  \\ \\
        \multirow{2}{*}{$cv^2_{\theta_Z}$} && 0.04 &  & 0.47 & 5.82  & 19.74 &  & 0.35 & 6.51  & 20.42 &  & 0.97 & 6.45  & 20.23 &  & 0.18 & 1.73 & 6.74 &  & -13.88 & -5.28 & -0.09 \\
        && 0.08 &  & 0.69 & 7.37  & 24.93 &  & 0.59 & 7.65  & 24.20  &  & 1.06 & 7.77  & 26.75 &  & 0.02 & 1.84 & 8.69 &  & -13.58 & -5.57 & 0.24  \\ \\
        \multirow{2}{*}{$c_f$} && 4    && 0.47 & 4.99  & 16.55 &  & 0.35 & 5.46  & 16.66 &  & 0.97 & 5.86  & 20.59 &  & 0.02 & 1.58 & 7.07 &  & -13.88 & -5.44 & 0.06  \\
        && 9    && 0.83 & 8.20   & 24.93 &  & 0.66 & 8.7   & 24.20  &  & 1.07 & 8.36  & 26.75 &  & 0.05 & 2.00    & 8.69 &  & -13.58 & -5.41 & 0.24  \\ \\
        \multirow{3}{*}{$\ell$}&& $\ell_0$    &  & 4.53 & 11.36 & 24.93 &  & 5.66 & 12.76 & 24.20  &  & 5.31 & 13.09 & 26.75 &  & 1.89 & 3.97 & 8.69 &  & -1.76  & -0.63 & 0.24  \\
        && $\ell_1$    &  & 0.47 & 3.88  & 14.06 &  & 0.35 & 3.98  & 13.20  &  & 0.97 & 3.81  & 12.26 &  & 0.02 & 0.52 & 2.45 &  & -13.29 & -7.51 & -3.96   \\
        && $\ell_2$    &  & 0.65 & 4.54  & 14.98 &  & 0.53 & 4.50   & 13.88 &  & 1.09 & 4.43  & 11.79 &  & 0.18 & 0.87 & 2.70  &  & -13.88 & -8.14 & -4.58 \\ \\ \midrule
        Total &&      &  & 0.47 & 6.59  & 24.93 &  & 0.35 & 7.08  & 24.20  &  & 0.97 & 7.11  & 26.75 &  & 0.02 & 1.79 & 8.69 &  & -13.88 & -5.43 & 0.24  \\
        \bottomrule
    \end{tabular}
  }
\end{table*}
\begin{table*}[!htbp]
  \fontsize{8pt}{9pt}\selectfont
  \scriptsize
  \parbox[t]{.9\linewidth}{
  \centering
  \caption{\textsf{Results of the simulation study - continuous state space}}
  \label{tab:simulation_continuous}
  \begin{tabular}{l l l l l l l l l l l l l l l }
        \toprule
        &&&& \multicolumn{3}{c}{$\%\reg \left(\xip \right)$} && \multicolumn{3}{c}{$\%OPT$} && \multicolumn{3}{c}{$\%PRED$}\\
        \cmidrule{5-7} \cmidrule{9-11} \cmidrule{13-15} 
        Input && Values && Min & Mean & Max && Min & Mean & Max && Min & Mean & Max  \\ \midrule
        \multirow{3}{*}{$L$} && 10 && 2.42 & 6.11 & 12.72 &  & 0.28 & 0.77 & 1.82 &  & -1.34 & -0.58 & -0.11 \\ 
        && 20 && 0.70 & 2.51 & 6.81  &  & 0.11 & 0.66 & 2.55 &  & -0.48 & -0.18 & 0.19 \\ 
        && 30 && 0.08 & 1.38 & 4.55  &  & 0.00   & 0.48 & 1.74 &  & -0.73 & -0.12 & 0.39  \\ \\
        \multirow{2}{*}{$cv^2_\nu$} && 0.3 && 0.08 & 3.53 & 12.72 &  & 0.00   & 0.50 & 1.63 &  & -1.19 & -0.30 & 0.09 \\
        && 0.6  && 0.28 & 3.14 & 12.61 &  & 0.11 & 0.77 & 2.55 &  & -1.34 & -0.29 & 0.39  \\ \\
        \multirow{2}{*}{$cv^2_{\theta_Z}$} && 0.04 &  & 0.08 & 3.00   & 10.05   &  & 0.00   & 0.79 & 2.55 &  & -0.74 & -0.19 & 0.39 \\
        && 0.08 && 0.13 & 3.67 & 12.72 &  & 0.00   & 0.49 & 1.74 &  & -1.34 & -0.40 & 0.19  \\ \\
        \multirow{2}{*}{$c_f$} && 4    && 0.08 & 2.45 & 7.69  &  & 0.00   & 0.47 & 1.99   &  & -0.78 & -0.21 & 0.28  \\
        && 9    && 0.29 & 4.22 & 12.72 &  & 0.02   & 0.81 & 2.55 &  & -1.34 & -0.37 & 0.39  \\ \\
        \multirow{3}{*}{$\ell$}&& $\ell_0$    &  & 2.12 & 5.71 & 12.72 &  & 0.05   & 1.14 & 2.55 &  & -1.30 & -0.24 & 0.39  \\
        && $\ell_1$    &  & 0.28 & 2.17 & 8.02    &  & 0.00   & 0.41 & 0.93 &  & -1.34 & -0.38 & -0.09   \\
        && $\ell_2$    &  & 0.08 & 2.12 & 8.04    &  & 0.00   & 0.37 & 0.86 &  & -1.17 & -0.26 & 0.00 \\ \\ \midrule
        Total &&      &  & 0.08 & 3.33 & 12.72 &  & 0.00   & 0.64 & 2.55 &  & -1.34 & -0.29 & 0.39  \\
        \bottomrule
    \end{tabular}
  }
\end{table*}
The results in Tables \ref{tab:simulation_discrete} and \ref{tab:simulation_continuous} indicate that EOP thresholds perform well in both discrete and continuous settings, with a mean $\%\reg(\xip)$ of 6.59\% and 3.3\%, respectively.
The regret of EOP relative to the integrated Bayes method is negligible, implying that EOP performs as well as the optimal policy in our experiments.
As the failure threshold $L$ increases from 10 to 30, mean $\%\reg(\xip)$ decreases from 11.05\% to 3.06\% in the discrete setting and from 6.1\% to 1.4\% in the continuous setting, consistent with our main convergence result (Theorem \ref{thm:costconvergancewithlearning}).
Moreover, $\%\reg(\xip)$ increases with the corrective maintenance cost $c_f$, because higher $c_f$ leads to smaller optimal thresholds, which in turn increase the approximation gap between $\xia$ and $\xio$ and amplify estimation errors of $\thetaH(\Ib_\tau)$.
Both tables show that the EOP thresholds effectively adapt to situations where the system incurs operating costs.
We observe from the results that the optimality gap of using $\xia$ instead of $\xio$ is small, indicating that $\xia$ provides an excellent approximation to the Oracle’s optimal thresholds.
Finally, the prediction gaps are more significant when the state space is discrete than in the continuous one, where they are negligible.

\section{Conclusion}
\label{section:Conclusion}
In this paper, we have studied the condition-based maintenance (CBM) problem for a heterogeneous population of components that undergo non-negative i.i.d. degradation processes with parameters unknown to the decision maker.
At equidistant time epochs, the decision maker receives real-time degradation information and must choose between replacing the operating component or allowing it to continue functioning.
The cost of replacing a failed component is considerably higher than that of a healthy component, and the system experiences a non-negative operating cost that rises with the degradation level.
At each time epoch, the degradation data is collected and used to learn the unknown parameters of the degradation process via a consistent estimator.
The decision maker then determines the maintenance action to minimize the cost-rate (long-run average cost).
In general, this problem can be solved using a partially observable Markov decision process (POMDP).
It is well known that POMDPs suffer from the curse of dimensionality.

We have proposed the Estimated Oracle's Optimal Policy (EOP) that chooses the CBM actions based on estimates of the optimal decision of an Oracle who has the full information of the true parameters of the degradation processes.
Furthermore, we have introduced a scaling regime in which the failure level, cost parameters, and total operating costs grow jointly.
This scaling regime captures the practical settings where the components' lifetime and the maintenance costs are much larger than the frequency of degradation measurements.
This common setting also represents conditions under which conventional POMDP approaches become intractable as the state space grows large.
We have proved that regret of the EOP, i.e., the difference between its cost-rate and that of the Oracle, converges to zero in our scaling regime, which also implies its asymptotic optimality.
We have evaluated the performance of the EOP on real degradation data of IXR systems relative to the state-of-the-art Integrated Bayes policy \citep{Drent2023}, which computes replacement policies using a POMDP model. Our results show that employing the EOP, on average, achieves a 0.63\% cost reduction compared to the Integrated Bayes policy.
We have conducted an extensive numerical experiment including discrete and continuous state spaces to test the performance of the EOP the Oracle's optimal policy.
We have observed that the relative regret of our EOP is small in both types of state spaces, with the average relative regret of 6.59\% in discrete state space and 3.3\% in continuous state space.
We have not been able to statically distinguish the regret of the EOP from the POMDP's (on average 7.11\%) which could only be computed in the discrete setting.
To the best of our knowledge, our EOP is the first data-driven CBM policy that can accommodate a wide range of degradation and parameter learning processes, provides an asymptotic performance guarantee, and has achieved remarkable results on real-life degradation data as well as numerical experiments across both discrete and continuous state spaces.

% Acknowledgments here
% \ACKNOWLEDGMENT{%
% Enter the text of acknowledgments here
% }% Leave this (end of acknowledgment)

% References here (outcomment the appropriate case) 
\bibliographystyle{informs2014} % outcomment this and next line in Case 1
\bibliography{bib} % if more than one, comma separated

% Appendix here
% Options are (1) APPENDIX (with or without general title) or 
%             (2) APPENDICES (if it has more than one unrelated sections)
% Outcomment the appropriate case if necessary
%
% \begin{APPENDIX}{<Title of the Appendix>}
% \end{APPENDIX}
%
%   or 

\newpage 

\begin{APPENDICES}
\section{Exponential compounding}\label{expCompoundingApp}
In this section, we adapt the analysis introduced in Example \ref{ex:exampleDef} to the case of exponential compounding. 

To this end, let $Z_{m,\tau}$ be supported on $\R_+$ and have an exponential distribution with parameter $\omega>0$, that is, $f_Z(x \mid \omega):=\P \left\{Z_{m,\tau}=x \mid \omega \right\} = \omega e^{-\omega x}$, where $x \geq 0$.
Accordingly, the probability mass function of $X_\tau$ can be expressed, using the law of total probability, as 
\begin{equation}
    \label{eq:pmf_CP_exp}
    f_X \left(x\mid\nu,\omega \right) = \begin{cases}
        f_M \left(0 \mid \nu \right) + \displaystyle\sum_{j=1}^\infty f_M \left(j \mid \nu \right)f_{Er}(0 \mid j,\omega), & x=0\\
        \displaystyle\sum_{j=1}^\infty f_M \left(j \mid \nu \right)f_{Er}(x \mid j,\omega), & x >0,
    \end{cases}
\end{equation}
where $f_{Er}(\cdot \mid j, \omega)$ is the Erlang-$j$ probability density function with the rate $\omega$, i.e., $f_{Er}(x \mid j, \omega):= \omega^j x^{j-1} e^{-\omega x}/(j-1)!$.
By Equations \eqref{eq:CP_moments} the moments of $X_\tau$ are given by
\begin{equation}
    \label{eq:CP_moments_exp}
    \mu(\nu,\omega) = \frac{\nu}{\omega}, \qquad \sigma^2(\nu,\omega) = \frac{2\nu}{\omega^2}.
\end{equation}

The degradation data observed by the decision maker at epoch $t \in \N_0$ includes the history of replacement times, the number of shocks, and the size of the damages induced by each shock up to $t$.
That is,
\[
    \Ib_t = \left \{N(j),A(j),M_{N(j),A(j)},Z_{N(j),1, A(j)}, \dots, Z_{N(j), M_{N(j),A(j)}, A(j)} \right\}_{j=0}^t.
\]
This degradation data also enables the computation of $X_{N(t),A(t)}$ and $S_{N(t),A(t)}$.
The functions $F_{\nu}$ and $F_{\theta_Z}$, as well as $F_X$ for every realization of $(\nu, \theta_Z)$ are also known.
The decision maker may sequentially estimate $(\nu, \theta_Z)$ in a Bayesian manner, based on the degradation data $\Ib_t$ using the estimators
$\left(\nuH \left( \Ib_t \right),\thetaH_{Z} \left( \Ib_t \right) \right)$, with density functions $f_\nu(\cdot \mid \Ib_t)$ and $f_{\theta_Z} (\cdot \mid \Ib_t)$, respectively.

Returning to the example, we notice that $f_Z(x \mid \omega)$ also belongs to the one-parameter exponential family. This property allows us to select conjugate priors for each of these (likelihood) functions, so that the posterior distributions remain within the same family, and the corresponding parameters can be updated efficiently, conditional on the degradation data $\Ib_t$ \citep[cf. Section 5.1.5.][]{Ghosh2007}.
The selection of the conjugate priors is as follows.

Again, we have that $f_{\nu} (\cdot) = f_\nu(\cdot \mid \Ib_0)$ corresponds to a gamma distribution with shape $\alpha_0$ and scale $\beta_0$. Now that the shock size is exponentially distributed, the density function $\theta_Z$ follows a gamma distribution. The selected $f_{\nu}$ and $f_{\theta_Z}$ are conjugate priors corresponding to likelihood functions $f_M$ and $f_Z$.
In both cases, $f_{\theta_Z}$ is parametrized by $a_0$ and $b_0$.
One can simply verify that the chosen $f_{\nu}$ and $f_{\theta_Z}$ satisfy all conditions specified in Section \ref{subsect:degrmodelHeterogeneous}.

Then, $\nuH \left( \Ib_t \right)$ follows a gamma distribution with parameters $\alpha_t$ and $\beta_t$.
Furthermore, $\thetaH_{Z} \left( \Ib_t \right)$ follows a gamma distribution with parameters $a_t$ and $b_t$.

We refer to $(\alpha_t, \beta_t, a_t, b_t)$ as hyperparameters.
Let $M_{A(t)}$ and $X_{A(t)}$ be the number of shocks and the degradation increment in the interval $(t-1,t]$.
% Let $m$ and $x$ denote the number of shocks and the corresponding cumulative degradation during the period $t-1 \in \N_0$, respectively.
Then, the hyperparameters can be updated as
\begin{alignat}{2}
    \label{eq:hyperParameter_update_exp}
    \left(\alpha_t, \beta_t, a_t, b_t \right)=& \left(\alpha_{t-1}+M_{A(t)}, \beta_{t-1}+1, a_{t-1}+M_{A(t)}, b_{t-1}+X_{A(t)} \right) \nonumber\\
    =& \left(\alpha_0+\sum_{j=1}^{A(t)} M_j, \beta_0+A(t), a_0+\sum_{j=1}^{A(t)} M_j, b_0+S_t \right)
\end{alignat}
% \begin{alignat*}{2}
%     &\left(\alpha_t, \beta_t, a_t, b_t \right)=\left(\alpha_{t-1}+M_{A(t)}, \beta_{t-1}+1, a_{t-1}+X_{A(t)}, b_{t-1}+M_{A(t)} \right), \quad && \textit{if the component is not in state 0},\\
%     &\left(\alpha_t, \beta_t, a_t, b_t \right)=\left(\alpha_0, \beta_0, a_0, b_0 \right), \quad && \textit{if the component is in state 0}.
% \end{alignat*}

The entire analysis in our examples can be replicated for many choices for $f_Z$ from the one-parameter exponential distributions if their conjugate priors satisfy the conditions outlined in Section \ref{subsect:degrmodelHeterogeneous}.

In the following, we adapt the analysis from Example \ref{ex:exampleDef_1} by replacing the geometric compounding distribution with an exponential distribution with a similar analytical framework.
We omit some intermediate derivation steps, as the reader can follow from the geometric compounding example.
Let $F_{Er}(\cdot \mid l,\omega)$ be the Erlang$-l$ distribution function with rate $\omega$. Then $F_X(\cdot \mid \nu, \omega)$ is given by
\begin{equation}
    \label{eq:cdf_CP_exp}
    F_X(x \mid \nu, \omega)= f_M(0 \mid \nu) +\sum_{l=1}^\infty f_M(l \mid \nu) F_{Er} (x\mid l, \omega).
\end{equation}
% Similar to Equation \eqref{eq:xi_lambda_CP_geom}, $\xi_\lambda$ can be obtained from the following problem
% \begin{alignat}{2}
%     \label{eq:xi_lambda_CP_exp}
%     \xi_\lambda(\nu,\omega)= &\min_{x \in \{0,\dots,L\}} x\\
%     \textit{s.t.} \quad &  a_2 x^2+ \left( \frac{2a_2\nu}{\omega} +a_1 \right)x - c_f \sum_{l=1}^\infty f_M \left(l\mid \nu \right) F_{Er} (x \mid l, \omega) \nonumber\\
%     &= \lambda -c_f (1 - f_M(0 \mid \nu)) -\frac{a_2\nu(\nu+2)}{\omega^2} - \frac{a_1\nu}{\omega} \nonumber
% \end{alignat}
Let the mean of the Erlang$-l$ distribution be $\mu_{Er}(l,\omega):=l/\omega$. Then 
\begin{alignat}{2}
    \label{eq:YDist_CP_exp}
    &\P\{Y(\nu,\omega)\leq x\}= \nonumber\\
    &\frac{\omega}{\nu} \left(x \left(1-f_M(0 \mid \nu)\right) + \sum_{l=1}^\infty f_M(l\mid \nu) \Big( \mu_{Er}(l,\omega) F_{Er}(x\mid l+1,\omega) - xF_{NB} (x \mid l,\omega) \Big)\right).
\end{alignat}
and 
\begin{equation}
    \label{eq:EY_CP_exp}
    E[Y(\nu,\omega)] = \frac{\nu+2}{\omega}.
\end{equation}
% Furthermore by \eqref{eq:r_ell_continuous},
% \begin{equation}
%     \label{eq:r_ell_CP_exp}
%     r(x,\nu,\omega)= \frac{a_2 \omega}{3\nu}x^3+ \frac{2a_2 \nu+a_1 \omega}{2\nu} x^2+ \frac{a_2 \left(\nu+2\right) + a_1 \omega}{\omega} \left(x+\frac{\nu}{\omega} \right).
% \end{equation}
Having these results, one can compute $\xi_{\gamma^*}$ using Theorem \ref{thm:gamma*}.
Next, the hyperparameters can be updated as per \eqref{eq:hyperParameter_update}, that is
\begin{alignat}{2}
    \label{eq:posteriorMean_CP_exp}
    \thetaH(\Ib_\tau)= & \left(\hat{\nu}(\Ib_\tau), \thetaH_Z(\Ib_\tau) \right)= \left(\frac{\alpha_\tau}{\beta_\tau}, \frac{a_\tau}{b_\tau} \right) \nonumber\\
    =& \left(\frac{\alpha_{\tau-1}+M_\tau}{\beta_{\tau-1}+1}, \frac{a_{\tau-1}+M_\tau}{b_{\tau-1}+X_\tau} \right)= \left(\frac{\alpha_0+\sum_{j=1}^\tau M_j}{\beta_0+\tau}, \frac{a_0+\sum_{j=1}^\tau M_j}{b_0+S_\tau} \right),
\end{alignat}
The difference between \eqref{eq:posteriorMean_CP_exp} and \eqref{eq:posteriorMean_CP_geom} lies in the distribution of $\theta_Z$: when the compounding distribution is exponential, $f_{\theta_Z}$ is the density of a gamma distribution, whereas under geometric compounding, $\theta_Z$ follows a beta distribution.
Finally, our EOP thresholds are expressed as follows
\[
   \xip(\Ib_\tau)= \xia  \left(\hat{\nu}(\Ib_\tau), \thetaH_Z(\Ib_\tau) \right).
\]
% and
% \[
%     \xiH^b_t(\nu,\omega) = \int_0^\infty \int_0^\infty \xi_{\gamma^*}(x,y) f_\nu(x \mid \alpha_t,\beta_t) f_{\theta_Z}(y \mid a_t,b_t) dx dy.
% \]

\section{Example: Gamma process}
This section characterizes the EOP thresholds for a Gamma degradation process by deriving explicit expressions under the assumption that $\ell(x) = 0$ for all $x$.
Let the increments $\{X_{1,\tau}\}_{\tau\in \N}$ be independently drawn from a gamma distribution with a known shape parameter $\alpha>0$ and unknown rate parameter $\beta>0$, with density
\[
    f_X(x \mid \alpha, \beta) = \frac{\beta^\alpha}{\Gamma(\alpha)} x^{\alpha-1} e^{-\beta x},
\]
where $\Gamma(\cdot)$ denotes the gamma function.
We impose a conjugate prior on $\beta$, specified as a gamma distribution with shape parameter $a_0>0$ and rate parameter $b_0>0$.
That is $(a_0,b_0)$ are the initial hyperparameters which can be updated as per
\begin{equation}
    \label{eq:hyperparameterupdate_gammaProcess}
    (a_t,b_t) = (a_{t-1}+\alpha, b_{t-1}+X_{A(t)})= (a_0+\alpha A(t), b_0+S_{N(t),A(t)}).
\end{equation}
In this section, $F_\theta(\cdot)$ denotes $F_\theta(\cdot \mid a_0, b_0)$, and we explicitly condition on the parameters of $F_\theta$ when parameters other than $(a_0,b_0)$ are considered.
Let $\hat{\beta}_t$ be the posterior mean of $\beta$, then
\begin{equation}
    \label{eq:beta_Estimate}
    \hat{\beta}_t = \frac{a_0+\alpha A(t)}{b_0+S_{N(t),A(t)}}.
\end{equation}
The key step is to derive a recursive relation that determines $\lambda_{j+1}$ from $\lambda_j$, as stated in Theorem \ref{thm:gamma*}\ref{thm:gamma*_FixedPoint}. 
This relation enables efficient computation of $\gamma^*$ that leads to the identification of $\xia$, that is,
\[
    \xia(\beta) = L - \min \left\{F_X^{-1}(1- \gamma^*/c_f \mid \alpha, \beta),L \right\} = L -  \min \left  \{ \frac{1}{\beta} q(\gamma^*),L \right\},
\]
with $q(x):= F_X^{-1}(1- x/c_f \mid \alpha, 1)$ $ x\geq0$, as in Equation \eqref{eq:NewsVendor}, and
\[
    \xip(\Ib_t) = L - \min \left \{ \frac{a_0+\alpha A(t)}{b_0+S_{N(t),A(t)}} q(\gamma^*) , L \right\},
\]
according to Equations \eqref{eq:adaptivethreshold} and \eqref{eq:beta_Estimate}.
In the following, we first state the recursive relation and then present the main steps to derive it.
$\lambda_{j+1}$ can be expressed as follows 
\begin{equation}
    \label{eq:lambda_recursion_gamma}
    \lambda_{j+1} = \dfrac{c_p+c_f-\left(\lambda_j \left(q(\lambda_j) - \alpha\right)+c_f(\alpha-q(\lambda_j) f_X \left(q(\lambda_j) \mid \alpha, 1 \right)) \right) (1 - F_\theta(y_j)) - \alpha c_f \left( I(\lambda_j)- \dfrac{L^\alpha b_0^{a_0} F_\theta(y_j \mid \alpha+a_0, L+b_0)}{B(\alpha, a_0) (L+b_0)^{\alpha+a_0}}\right)}{\dfrac{ La_0}{b_0} (1-F_\theta(y_j \mid a_0+1,b_0)
    -q(\lambda_j)(1-F_\theta(y_j))
    +\dfrac{\alpha+1}{2}},
\end{equation}
where the real function $I$ is defined by
\[
    I(x):= \frac{1}{L} \int_0^{q(x)} \left(\frac{u}{\alpha} \left(1-F_X \left(u \mid\alpha,1 \right)\right)+F_X \left(u \mid\alpha,1 \right)\right) dF_\theta(u),
\]
$y_j:= q(\lambda_j)/L$, and $B(\cdot,\cdot)$ is the beta function, i.e.,
\[
    B(\alpha, a_0):= \frac{\Gamma(\alpha) \Gamma(a_0)}{\Gamma(\alpha+a_0)}.
\]
The numerical computation of $I(x)$ is efficient, as the integrand remains bounded over the domain of the integral, including the boundaries.
$\lambda_{j+1}$ can be approximated when $L$ is large as
\begin{equation}
    \label{eq:lambda_recursion_gamma_approx}
    \lambda_{j+1} \approx \dfrac{c_p+c_f \left( 1 + q(\lambda_j) f_X \left(q(\lambda_j) \mid \alpha, 1 \right) -\alpha \right)+ \lambda_j \left( \alpha - q(\lambda_j) \right)}{\dfrac{ La_0}{b_0}
    -q(\lambda_j)
    +\dfrac{\alpha+1}{2}},
\end{equation}

In the remainder of this section we present the derivation of Equation \eqref{eq:lambda_recursion_gamma}.
We first derive explicit expressions for $\P\{Y(\beta) \leq x\}$ and $\E[Y(\beta)]$ according to Equations \eqref{eq:YProb} and \eqref{eq:EY}, respectively.
We then use these results together with Equations \eqref{eq:gamma} and \eqref{eq:NewsVendor} to derive the recursive equation.

We use integration by part to find
\begin{alignat*}{2}
    \P\{Y(\beta) \leq x\}=& \frac{1}{\mu(\beta)} \int_0^x \left(1-F_X(y \mid \alpha, \beta) \right) dy = \frac{1}{\mu} \left(x - x F_X(x \mid \beta) + \int_0^x y f_X(y  \mid \alpha, \beta) dy\right) \\
    =&  \frac{1}{\mu(\beta)} \left(x \left(1-F_X(x \mid \alpha,\beta)\right) +
    \int_0^x \frac{\beta^{\alpha}}{\Gamma(\alpha)} y^{\alpha} e^{-\beta y} dy \right)\\
    = & 
    \frac{1}{\mu(\beta)} \left(x \left(1-F_X(x \mid \alpha,\beta)\right) +
    \frac{\alpha}{\beta}\int_0^x \frac{\beta^{\alpha+1}}{\Gamma(\alpha+1)} y^{\alpha} e^{-\beta y} dy \right)\\
    = & 
    \frac{1}{\mu(\beta)} \left(x \left(1-F_X(x \mid \alpha,\beta)\right) +
    \mu(\beta) \left(\int_0^x \frac{\beta^{\alpha}}{\Gamma(\alpha)} y^{\alpha-1} e^{-\beta y} dy - \frac{\beta^\alpha}{\Gamma(\alpha+1)}x^\alpha e^{-\beta x} \right) \right)\\
    =& \frac{x}{\mu(\beta)}  \left(1-F_X(x \mid\alpha,\beta)\right) +
    F_X(x \mid \alpha, \beta) - \frac{x}{\alpha} f_X(x \mid \alpha, \beta) .
\end{alignat*}
$E[Y(\beta)]$ can be expressed by
\[
    \E[Y(\beta)]= \frac{\alpha+1}{2 \beta}.
\]
By Theorem \ref{thm:gamma*}\ref{thm:gamma*_FixedPoint} and Equation \eqref{eq:gamma}
\[
    \lambda_{j+1} = \frac{c_p+c_f - c_f \displaystyle \int_0^\infty \P \left\{Y(y) \leq \phi \left(\xia_{\lambda_j}(y) \right) \right\}dF_\theta(y)}{\displaystyle \int_0^\infty \frac{1}{\mu(y)} \left(L-\phi \left( \xia_{\lambda_j}(y) \right) + \E[Y(y)]
    \right) dF_\theta(y)},
\]
where by Equation \eqref{eq:NewsVendor}
\[
    \phi \left(\xia_{\lambda_j}(y) \right) = \min \left\{F_X^{-1} \left(1- \frac{\lambda_j}{c_f} \mid \alpha, y \right), L \right\} = \min \left\{\frac{1}{y} q(\lambda_j), L \right\}.
\]
It follows that
\begin{equation}
    \label{eq:lambda_j1}
    \scriptsize
    \lambda_{j+1} =
    \frac{c_p+c_f - c_f \displaystyle \int_0^\infty \left(\frac{y}{\alpha} \phi \left( \xia_{\lambda_j}(y) \right) \left(1-F_X \left(\phi \left( \xia_{\lambda_j}(y) \right) \mid\alpha,y \right)\right) +
    F_X \left(\phi \left( \xia_{\lambda_j}(y) \right) \mid \alpha, y \right) - \frac{\phi \left( \xia_{\lambda_j}(y) \right)}{\alpha}f_X \left(\phi \left( \xia_{\lambda_j}(y) \right) \mid \alpha, y \right) \right)  dF_\theta(y)}{\displaystyle \int_0^\infty \frac{y}{\alpha} \left(L-\phi \left( \xia_{\lambda_j}(y) \right) + \frac{\alpha+1}{2y}
    \right) dF_\theta(y)}.
\end{equation}
Observe that $\phi \left(\xia_{\lambda_j}(\beta) \right)$ can be expressed by
\begin{equation*}
    \phi \left(\xia_{\lambda_j}(\beta) \right)= \begin{dcases}
        L, & \text{if } 0 <\beta \leq y_j,\\
        \frac{q(\lambda_j)}{\beta}, & \text{if }\beta> y_j.
    \end{dcases}
\end{equation*}
Thus,
\begin{alignat}{2}
    \label{eq:lambda_j1_term1}
    &\displaystyle \int_0^\infty \left(\frac{y}{\alpha} \phi \left( \xia_{\lambda_j}(y) \right) \left(1-F_X \left(\phi \left( \xia_{\lambda_j}(y) \right) \mid\alpha,y \right)\right)+F_X \left(\phi \left( \xia_{\lambda_j}(y) \right) \mid\alpha,y \right)\right) dF_\theta(y) \nonumber \\
    =&\displaystyle \int_0^{y_j} \left(\frac{y}{\alpha} \phi \left( \xia_{\lambda_j}(y) \right) \left(1-F_X \left(\phi \left( \xia_{\lambda_j}(y) \right) \mid\alpha,y \right)\right)+F_X \left(\phi \left( \xia_{\lambda_j}(y) \right) \mid\alpha,y \right)\right) dF_\theta(y)  \nonumber \\
    + & \int_{y_j}^\infty \left(\frac{y}{\alpha} L \left(1-F_X \left(L \mid\alpha,y \right)\right)+F_X \left(L \mid\alpha,y \right)\right) dF_\theta(y)\nonumber \\
    =& \int_0^{y_j} \left(\frac{Ly}{\alpha} \left(1-F_X \left(y L \mid\alpha,1 \right)\right)+F_X \left(y L \mid\alpha,1 \right)\right) dF_\theta(y) + \displaystyle \int_{y_j}^\infty \left(\frac{y}{\alpha} \frac{q(\lambda_j)}{y}  \frac{\lambda_j}{c_f}+ \left(1-\frac{\lambda_j}{c_f} \right) \right) dF_\theta(y) \nonumber\\
    = & \int_0^{y_j} \left(\frac{y L}{\alpha} \left(1-F_X \left(y L \mid\alpha,1 \right)\right)+F_X \left(y L \mid\alpha,1 \right)\right) dF_\theta(y) + \left(\frac{\lambda_j \left(q(\lambda_j) - \alpha\right)}{\alpha c_f}+1 \right) (1 - F_\theta(y_j)) \nonumber\\
    = & \frac{1}{L} \int_0^{q(\lambda_j)} \left(\frac{u}{\alpha} \left(1-F_X \left(u \mid\alpha,1 \right)\right)+F_X \left(u \mid\alpha,1 \right)\right) dF_\theta(u) + \left(\frac{\lambda_j \left(q(\lambda_j) - \alpha\right)}{\alpha c_f}+1 \right) (1 - F_\theta(y_j))
\end{alignat}
Next, observe that for $0 <\beta \leq y_j$
\[
    \phi \left( \xia_{\lambda_j}(\beta) \right)f_X \left(\phi \left( \xia_{\lambda_j}(\beta) \right) \mid \alpha, \beta \right)= L f_X(L \mid \alpha, \beta),
\]
and for $\beta>y_j$
\[
    \phi \left( \xia_{\lambda_j}(\beta) \right)f_X \left(\phi \left( \xia_{\lambda_j}(\beta) \right) \mid \alpha, \beta \right)= \frac{q(\lambda_j)}{\beta} \frac{\beta^\alpha}{\Gamma(\alpha)} \left(\frac{q(\lambda_j)}{\beta}\right)^{\alpha-1} e^{-\beta \frac{q(\lambda_j)}{\beta}}
    = q(\lambda_j) f_X \left(q(\lambda_j) \mid \alpha, 1 \right),
\]
% \begin{alignat*}{2}
%     &\phi \left( \xia_{\lambda_j}(\beta) \right)f_X \left(\phi \left( \xia_{\lambda_j}(\beta) \right) \mid \alpha, \beta \right)= \frac{q(\lambda_j)}{\beta} \frac{\beta^\alpha}{\Gamma(\alpha)} \left(\frac{q(\lambda_j)}{\beta}\right)^{\alpha-1} e^{-\beta \frac{q(\lambda_j)}{\beta}} \\
%     =& q(\lambda_j) f_X \left(q(\lambda_j) \mid \alpha, 1 \right),
% \end{alignat*}
which implies that
\begin{alignat}{2}
    \label{eq:lambda_j1_term2}
    & \displaystyle \int_0^\infty \frac{\phi \left( \xia_{\lambda_j}(y) \right)}{\alpha}f_X \left(\phi \left( \xia_{\lambda_j}(y) \right) \mid \alpha, y \right) dF_\theta(y) \nonumber
    \\
    =& \displaystyle \int_0^{y_j} \frac{\phi \left( \xia_{\lambda_j}(y) \right)}{\alpha}f_X \left(\phi \left( \xia_{\lambda_j}(y) \right) \mid \alpha, y \right) dF_\theta(y) + \displaystyle \int_{y_j}^\infty \frac{\phi \left( \xia_{\lambda_j}(y) \right)}{\alpha}f_X \left(\phi \left( \xia_{\lambda_j}(y) \right) \mid \alpha, y \right) dF_\theta(y) \nonumber
    \\
    =&  \frac{L^\alpha b_0^{a_0}}{\Gamma(\alpha) \Gamma(a_0)} \displaystyle \int_0^{y_j} y^{\alpha+a_0-1}e^{-(b_0+L) y}dy + \frac{q(\lambda_j)}{\alpha} f_X \left(q(\lambda_j) \mid \alpha, 1 \right)(1-F_\theta(y_j)) \nonumber
    \\
    =& \frac{L^\alpha b_0^{a_0} F_\theta(y_j \mid \alpha+a_0, L+b_0)}{B(\alpha, a_0) (L+b_0)^{\alpha+a_0}}  + \frac{q(\lambda_j)}{\alpha} f_X \left(q(\lambda_j) \mid \alpha, 1 \right)(1-F_\theta(y_j)).
\end{alignat}
Finally,
\begin{alignat}{2}
    \label{eq:lambda_j1_term3}
    &\displaystyle \int_0^\infty \frac{y}{\alpha} \left(L-\phi \left( \xia_{\lambda_j}(y) \right) + \frac{\alpha+1}{2y}
    \right) dF_\theta(y) \nonumber\\
    =& \displaystyle \int_0^{y_j} \frac{y}{\alpha} \left(L-\phi \left( \xia_{\lambda_j}(y) \right) + \frac{\alpha+1}{2y}
    \right) dF_\theta(y) + \displaystyle \int_{y_j}^\infty \frac{y}{\alpha} \left(L-\phi \left( \xia_{\lambda_j}(y) \right) + \frac{\alpha+1}{2y}
    \right) dF_\theta(y) \nonumber \\
    =& \displaystyle \int_0^{y_j} \frac{y}{\alpha} \left(L-L + \frac{\alpha+1}{2y}
    \right) dF_\theta(y) + \displaystyle \int_{y_j}^\infty \frac{y}{\alpha} \left(L-\frac{q(\lambda_j)}{y} + \frac{\alpha+1}{2y}
    \right) dF_\theta(y) \nonumber \\
    =& \frac{1}{\alpha}\left(\frac{ La_0}{b_0} (1-F_\theta(y_j \mid a_0+1,b_0)
    -q(\lambda_j)(1-F_\theta(y_j))
    +\frac{\alpha+1}{2}\right).
\end{alignat}
Equation \eqref{eq:lambda_j1} together with \eqref{eq:lambda_j1_term1}, \eqref{eq:lambda_j1_term2}, and \eqref{eq:lambda_j1_term3} provide \eqref{eq:lambda_recursion_gamma}. 

\section{Lemma \ref{lemma:ExpectedLifetime}: On the Finite Expected Lifetime of Components}
\begin{lemma}
    \label{lemma:ExpectedLifetime}
    Under any policy $\pi$, the expected lifetime and the expected maintenance cost of an operating component $i \in \N$ are finite, i.e., $\E[T_i]<\infty$, and $\E[C_i]<\infty$. 
\end{lemma}
\proof{Proof of Lemma \ref{lemma:ExpectedLifetime}.}
We upperbound $\E[T_i]$ using the expected lifetime of the component under the policy that only replaces upon failure. Note that $T_i(L,\theta_i)$ is the time until failure of the $i$-th component.
Let $U_i$ be the renewal measure of $F_X(\cdot\mid \theta_i)$ defined by $U_i(x):= \sum_{\tau=0}^\infty F_X^{*\tau}(x \mid \theta_i)$ with $F_X^{*\tau}(\cdot\mid\theta_i)$ denoting the $\tau$-fold convolution of $F_X(\cdot\mid\theta_i)$, and $F_X^{*0}$ the distribution function degenerate at the origin.
Next, For all $x \geq 0$,
\begin{equation}
    \label{eq:BoundsU}
    \frac{1}{\mu(\theta_i)}x\leq U_i(x) \leq \frac{1}{\mu(\theta_i)}x + \frac{2\E \left[Y(\theta_i) \right]}{\mu(\theta_i)},
\end{equation}
where $\E[Y(\theta_i)]$ is given by \eqref{eq:EY}.
The left inequality in \eqref{eq:BoundsU} follows from the non-negativity of the residual life after $x$ \citep[cf. V.6.][]{Asmussen2003}, and the right inequality in \eqref{eq:BoundsU} is Lorden's inequality \citep[][]{Lorden1970}.
Moreover, from renewal theory,
\begin{equation}
    \label{eq:meanT}
    \E[T \left(L,\theta_i \right)] = U_i(L)+1
\end{equation}
\citep[cf. Theorem V.2.4.][]{Asmussen2003}.
Equation \eqref{eq:BoundsU} together with \eqref{eq:meanT} imply
\begin{equation}
    \label{eq:ET_bounds}
    0 \leq \E[\E[T \left(L,\theta_i \right) \mid \theta_i]] \leq L \E \left[ \frac{1}{\mu(\theta_i)} \right]  + 2 \E\left[ Y(\theta_i) \right]+1 < \infty,
\end{equation}
where the last inequality follows from Assumption \ref{assum:converge}.
Next,
\begin{equation}
    \label{eq:finiteCycleCost}
    \E[C_i] \leq c_p+c_f+\E \left[ \sum_{\tau=1}^{T(L,\theta_i)} \ell(L)\right]= c_p+c_f+\ell(L) \E[T(L,\theta_i)]<\infty.
\end{equation}
The first inequality in \ref{eq:finiteCycleCost} follows from the definition of $C_i$ according to Equation \eqref{eq:Cnindicator} and the upperbound on $\ell(\cdot)$ as per Assumption \ref{assum:ellbounded}, and the second inequality from \eqref{eq:ET_bounds}.
\Halmos
\endproof

\section{Proof of Theorem \ref{thm:ExistNStrucOPTHetero}}
Consider the following Bellman optimality equations under the discounted total cost criterion with discount factor $\alpha \in (0,1)$
\begin{equation}
    \label{eq:discOPTEq}
    v_\alpha(s,\theta) =
    \begin{cases}
        \min \Big\{\ell(s)+\alpha\E[v_\alpha(s+X_{1,1},\theta) \mid \theta_1=\theta], \ell(s)+ c_p + \displaystyle \alpha\int_{\Theta} \E[v_\alpha(X_{1,1},\theta_1) \mid \theta_1= \yb] dF_\theta(\yb) \Big\}, & \forall s \leq L, \theta \in \Theta \\
        \ell(s)+c_p+c_f + \displaystyle \alpha \int_{\Theta} \E[v_\alpha(X_{1,1},\theta) \mid \theta_1 = \yb] dF_\theta(\yb), & \forall s > L, \theta \in \Theta.
    \end{cases}
\end{equation}
Define $m_\alpha := \displaystyle\inf_{(s,\theta) \in [0,L] \times \Theta} v_\alpha(s,\theta)$, and $w_\alpha:= v_\alpha(s, \theta)-m_\alpha$.
We intend to use the results of \cite{schal1993} to show that
\begin{equation}
    \label{eq:convergenceAverageCost}
    \lim_{\alpha \to 1} (1-\alpha)m_\alpha = \lim_{\alpha \to 1} (1-\alpha)v_\alpha(s, \theta)=g^*, \qquad \forall (s,\theta) \in [0,L] \times \Theta.
\end{equation}
In particular, our aim is to use Theorem 3.8 in \cite{schal1993}.
We restate this Theorem as the following lemma, in a form that is more convenient for our analysis.
\begin{lemma} \citep[adapted from Theorem 3.8.][]{schal1993}
    \label{lemma:Thm3.8Schal}
    For all states $(s,\theta) \in [0,L] \times \Theta$, assume that (A) $\displaystyle\sup_{\alpha<1} w_\alpha(s,\theta)<\infty$, and (B) the decision space is finite.
    Then there exists a stationary policy $\pi^*$ that is average optimal in the sense that $g(\pi^*)=g^*$.
    Furthermore, the convergence result \eqref{eq:convergenceAverageCost} holds.
\end{lemma}
The critical step is to prove that condition (A) in Lemma \ref{lemma:Thm3.8Schal} holds for our problem.
To do this, we use Lemma 4.1 in \cite{schal1993} which is restated (a customized version) as follows.
For some $\eta>0$, let $\mathcal{A}_{S,\eta} \subseteq [0,L]$ and $\mathcal{A}_{\theta,\eta} \subseteq \Theta$ be such that $v(s,\theta) \leq m_\alpha+\eta $ for all $s \in \mathcal{A}_{S,\eta}$, and $\theta \in \mathcal{A}_{\theta,\eta}$.
Notice that $v(s,\theta)$ is non-decreasing in $s$ and that $\mathcal{A}_{S,\eta}$ and $\mathcal{A}_{\theta,\eta}$ cannot be empty as $\eta$ is positive and  $m_\alpha$ is the infimum of $v(s,\theta)$ over the state space.
Define the stopping time $n_\alpha:= \displaystyle \inf_{i\in \N} \left\{\displaystyle i:S_{T_i+1} \in \mathcal{A}_{s,\eta}, \theta_{T_i+1} \in\mathcal{A}_{\theta,\eta} \right\}$.
\begin{lemma} \citep[adapted from Lemma 4.1.][]{schal1993}
    \label{lemma:Lemma4.1Schal}
    For $\eta \geq 0$, $\alpha<1$, $(s,\theta) \in [0,L] \times \Theta$: $w_\alpha(s,\theta) \leq \eta + \inf_\pi \E [\sum_{i=1}^{n_\alpha} C(S_{T_i(\pi)+1},\theta_{T_i(\pi)+1}) ]$.
\end{lemma}
Observe that the sequences $\{S_{T_i+1}\}_{i\in \N}$ and $\{\theta_{N(T_i+1)}\}_{i\in \N}$ are i.i.d.
Therefore, $n_\alpha$ is a geometric random variable and has a finite expectation.
Define $t_\alpha = \displaystyle \sum_{i=1}^{n_\alpha} T_i$.
By Wald's identity and Lemma \ref{lemma:ExpectedLifetime}, $\E[t_\alpha]= \E[n_\alpha]\E[T_i]<\infty$.
Now, By Lemma \ref{lemma:Lemma4.1Schal}, $w_\alpha (s, \theta) \leq \eta + (c_p+c_f+\ell(L)) \E[t_\alpha]<\infty$ which proves that condition (A) in Lemma \ref{lemma:Thm3.8Schal} holds.
Condition (B) in Lemma \ref{lemma:Thm3.8Schal} obviously holds, since we have binary decisions for each state.
Thus, Part \ref{thm:ExistNStrucOPTHeteroOPTPolicy} of Theorem \ref{thm:ExistNStrucOPTHetero} follows from Lemma \ref{lemma:Thm3.8Schal}.

Finally, there exists a threshold policy that satisfies the optimality Equations \eqref{eq:discOPTEq} that complete the proof for Part \ref{thm:ExistNStrucOPTHeteroThreshPolicy}. \Halmos

\section{Proof of Lemma \ref{lemma:CostRateHetero}}
\label{section:Proof_lemma:CostRateHetero}
Let $\pi$ be a policy that determines the replacement actions based on the available information of the operating component.
Under this structure, maintenance interventions disregard any history of actions or observations associated with previously installed components.
Consequently, the sequences $\{C_i\}_i\in \N$ and $\{T_i\}_{i \in \N}$ are i.i.d., as each component is drawn from the same population.
The result then follows directly from the renewal–reward theorem (cf. Theorem VI.3.1. of \cite{Asmussen2003}) combined with Lemma \ref{lemma:ExpectedLifetime}.

\begin{remark}
    % Equation \eqref{eq:KSLLN_Application} holds even when replacement decisions differ for components with the same parameter, as long as these decisions depend only on the history of the operating component.
    % In that case, the application of the Kolmogorov law of large numbers on the numerator and denominator of \eqref{eq:KSLLN_Application} yields
    % \begin{equation}
    %     \label{eq:GeneralSLLNDef}
    %     g(\xi) = \frac{\displaystyle \int_\Theta\E[\E[C(\xi(\yb),\yb)\mid \xi(\yb)]] dF_\theta(\yb)}{\displaystyle \int_\Theta\E[\E[T(\xi(\yb),\yb)\mid \xi(\yb)]] dF_\theta(\yb)}.
    % \end{equation}
    This argument is more general than the Lemma’s statement and remains valid for any replacement policy that depends only on the information of the operating component. In that case, $g(\xi)$ can be expressed as
    \begin{equation}
        \label{eq:GeneralSLLNDef}
        g(\xi) = \frac{\displaystyle \int_\Theta\E[\E[C(\xi(\yb),\yb)\mid \xi(\yb)]] dF_\theta(\yb)}{\displaystyle \int_\Theta\E[\E[T(\xi(\yb),\yb)\mid \xi(\yb)]] dF_\theta(\yb)},
    \end{equation}
    (cf. Equation~\ref{eq:SLLNDef}).
    We will use this result later when we introduce our adaptive policies, in which the replacement decisions are made based on the available degradation data of the current component. \Halmos
\end{remark}

\section{Proof of Theorem \ref{thm:lambda*}}
\label{section:Proof_thm:lambda*}
The proof of Theorem \ref{thm:lambda*} is adapted from the Appendix of \cite{Aven1986}.
We present a tailored version for completeness.
\begin{lemma}
    \label{lemma:sub_lambda*}
    \begin{enumerate}[(a)]
        \item \label{lemma:sub_lambda*_aboveg*} For $\lambda> g^*$, $\displaystyle \int_\Theta \hT(\lambda, \xi_\lambda, \yb) dF_\theta(\yb)<0$, and equivalently, $g(\xi_\lambda)<\lambda$.
        % \item \label{lemma:sub_lambda*_belowg*} For $\lambda< g^*$, $\displaystyle \int_\Theta \hT(\lambda, \xi_\lambda, \yb) dF_\theta(\yb)>0$, and equivalently, $g(\xi_\lambda)>\lambda$.
        \item \label{lemma:sub_lambda*_monoconcavity} $\displaystyle \int_\Theta \hT(\lambda, \xi_\lambda, \yb) dF_\theta(\yb)$ is non-increasing and concave in $\lambda$.
        \item \label{lemma:sub_lambda*_distance} For $\lambda_1, \lambda_2 \in \R$ satisfying $g^*< \lambda_1 < \lambda_2$, we have
        \[
            (\lambda_2 - \lambda_1) \displaystyle\int_\Theta \E \left[T(\xi_{\lambda_1},\yb)  \right] dF_\Theta(\yb)<-\displaystyle\int_\Theta\E \left[\hT \left(\lambda_2, \xi_{\lambda_2}, \yb \right) \right]dF_\Theta(\yb)<0.
        \]
    \end{enumerate}
\end{lemma}
By the arguments in Appendix \ref{section:Proof_lemma:CostRateHetero} and the cost structure, for all $\xi \in \Xi$
\[
    0< \displaystyle\int_\Theta \E \left[C(\xi(\yb),\yb) dF_\Theta(\yb) \right]< \infty, \qquad \mbox{and} \qquad 1\leq \displaystyle\int_\Theta \E \left[T(\xi(\yb),\yb) dF_\Theta(\yb) \right]<\infty.
\]
Therefore, $0 < g(\xi)<\infty$, and for all $\lambda \in \R$ we have $\Bigg \vert \displaystyle\int_\Theta\E \left[\hT \left(\lambda, \xi, \yb \right) \right]dF_\Theta(\yb) \Bigg \vert <\infty$.
\proof{Proof of Lemma \ref{lemma:sub_lambda*}.}
\begin{enumerate}[(a)]
    \item Let $\xi \in \Xi$ satisfy $g^* \leq g(\xi)<\lambda$. Then,
    \[
        \displaystyle \int_\Theta \hT(\lambda, \xi_\lambda, \yb) dF_\theta(\yb) \leq \displaystyle \int_\Theta \hT(\lambda, \xi, \yb) dF_\theta(\yb) < \displaystyle \int_\Theta \hT(g(\xi), \xi_\lambda, \yb) dF_\theta(\yb) =0.
    \]
    By Equation \eqref{eq:SLLN}, $g(\xi_\lambda) < \lambda$.
    % \item The proof of part \ref{lemma:sub_lambda*_belowg*} follows from $\lambda <g^* \leq g(\xi_\lambda)$.
    \item We first prove the monotonicity.
    Let $\lambda_1 \leq \lambda_2$, then
    \[
         \displaystyle \int_\Theta \hT(\lambda_2, \xi_{\lambda_2}, \yb) dF_\theta(\yb) \leq \displaystyle \int_\Theta \hT(\lambda_2, \xi_{\lambda_1}, \yb) dF_\theta(\yb) \leq \displaystyle \int_\Theta \hT(\lambda_1, \xi_{\lambda_1}, \yb) dF_\theta(\yb).
    \]
    Next we show the concavity. Let for $\lambda_1,\lambda_2 \in \R$ and $\alpha \in [0,1]$, $\lambda= \alpha\lambda_1+(1-\alpha)\lambda_2$, then
    \begin{alignat*}{2}
        \displaystyle \int_\Theta \hT(\lambda, \xi_\lambda, \yb) dF_\theta(\yb) = & \alpha \displaystyle \int_\Theta \hT(\lambda_1, \xi_\lambda, \yb) dF_\theta(\yb) + (1-\alpha) \displaystyle \int_\Theta \hT(\lambda_2, \xi_\lambda, \yb) dF_\theta(\yb)\\
        \geq & \alpha \displaystyle \int_\Theta \hT(\lambda_1, \xi_{\lambda_1}, \yb) dF_\theta(\yb) + (1-\alpha) \displaystyle \int_\Theta \hT(\lambda_2, \xi_{\lambda_2}, \yb) dF_\theta(\yb).
    \end{alignat*}
    \item Let $g^* < \lambda_1<\lambda_2$, then
    \begin{alignat*}{2}
        \displaystyle \int_\Theta \hT(\lambda_2, \xi_{\lambda_2}, \yb) dF_\theta(\yb)\leq &
        \displaystyle \int_\Theta \hT(\lambda_2, \xi_{\lambda_1}, \yb) dF_\theta(\yb) = \displaystyle \int_\Theta \E[C({\lambda_1}, \xi_{\lambda_1}, \yb)]-\lambda_2  \E[T({\lambda_1}, \xi_{\lambda_1}, \yb)] dF_\theta(\yb)\\
        = &\displaystyle \int_\Theta \E[C({\lambda_1}, \xi_{\lambda_1}, \yb)]-\lambda_1  \E[T({\lambda_1}, \xi_{\lambda_1}, \yb)] -(\lambda_2-\lambda_1)  \E[T({\lambda_1}, \xi_{\lambda_1}, \yb)] dF_\theta(\yb)\\
        = & \displaystyle \int_\Theta \hT({\lambda_1}, \xi_{{\lambda_1}}, \yb) dF_\theta(\yb) - \displaystyle \int_\Theta (\lambda_2-\lambda_1)  \E[T({\lambda_1}, \xi_{\lambda_1}, \yb)] dF_\theta(\yb) \\
        < & - \displaystyle \int_\Theta (\lambda_2-\lambda_1)  \E[T({\lambda_1}, \xi_{\lambda_1}, \yb)] dF_\theta(\yb)
    \end{alignat*}
    Notice that by Part \ref{lemma:sub_lambda*_aboveg*}, both $\displaystyle \int_\Theta \hT(\lambda_2, \xi_{\lambda_2}, \yb) dF_\theta(\yb)$ and $\displaystyle \int_\Theta \hT(\lambda, \xi_{\lambda}, \yb) dF_\theta(\yb)$ are strictly positive which completes the proof. \Halmos
\end{enumerate}
\endproof

\proof{Proof of Theorem \ref{thm:lambda*}.}
Let $\lambda_1 \in \R$, then $\lambda_2=g(\xi_{\lambda_1})\geq g^*$.
We can verify by induction and using part \ref{lemma:sub_lambda*_aboveg*} of Lemma \ref{lemma:sub_lambda*} that for all $j\geq2$, $g^* \leq \lambda_{j+1} = g(\xi_{\lambda_j})<\lambda_j$.
If $\lambda_j=g^*$, then $\lambda_{j+1} = \lambda_j=g^*$.
Therefore, there exists a $\lambda_\infty\geq g^*$ such that $\lambda_\infty \leq \lambda_j$ for all $j \in \N$ and $\displaystyle\lim_{j \to \infty} \lambda_j -\lambda_\infty=0$.
Notice that by Equation \eqref{eq:ET_bounds}, $\displaystyle\int_\Theta \E \left[T(\xi_{\lambda_j},\yb)  \right] dF_\Theta(\yb)$ is bounded for all $j\in \N$.
Thus, by Parts \ref{lemma:sub_lambda*_monoconcavity} and \ref{lemma:sub_lambda*_distance} of Lemma \ref{lemma:sub_lambda*},
\[
    \lim_{j \to \infty} \displaystyle\int_\Theta\E \left[\hT \left(\lambda_j, \xi_{\lambda_j}, \yb \right) \right]dF_\Theta(\yb)=0,
\]
which establishes $\lambda_\infty=g^*$. \Halmos

\section{Proof of Lemma \ref{lemma:elementwise_convergence}}

We provide the proof for the continuous degradation process; the argument for the $d$-lattice case is similar, with integration over the degradation level replaced by summation. 

\noindent
Proof of Part \ref{lemma:elementwise_convergence_per_theta}:
For notational convenience, we omit the conditioning on $\theta_1$, $\theta$, and the index $1$.
We have,
\begin{alignat}{2}
    \label{eq:StateDepTotalCost}
    \E \left[\sum_{\tau=1}^{T \left(\xi \right)} \ell(S_\tau,k) \right] = & \sum_{\tau=1}^\infty \E \left[\ell(S_\tau,k); T \left(\xi \right) \geq \tau \right] \nonumber\\
    =&\int_0^\infty \ell(x,k)dF_X(x) + \sum_{\tau=1}^\infty\int_0^{\xi}\int_0^\infty \ell(s+x,k)dF_X(x)dF_X^{*\tau}(s) \nonumber\\
    =& \E[\ell(X,k)] + \int_0^{\xi}\E[\ell(X+s,k)]dU(s)
\end{alignat}
Define for $k>0$
\begin{equation}
    \label{eq:wnDef}
    w_k(s):=\E \left [\ellT \left(\frac{X+s}{k} \right) \right],
\end{equation}
and observe that $w_k(\cdot)\leq \ellT \left(\LT \right)$.
Then Equation \eqref{eq:StateDepTotalCost} gives
\begin{equation}
    \label{eq:StateDepTotalCost_n}
    % \frac{1}{k}
    \E \left[\sum_{\tau=1}^{T\left(\xi^{(k)}\right)} \ell(S_\tau,k) \right]= w_k(0) + \int_0^{\xi^{(k)}}w_k(s)dU(s).
\end{equation}
Note that $w_k(\cdot)$ is not directly Riemann integrable, since $\displaystyle \int_0^\infty w_k(s)ds$ diverges, for every $k>0$. Thus the key renewal theorem \citep[cf. Theorem V.4.7.][]{Asmussen2003} cannot be applied immediately, on the RHS of Equation \eqref{eq:StateDepTotalCost_n}.
We, therefore, adopt a more elaborate approach to establish the result.

By renewal theory there exists a function $U_c:\R_+ \to \R_+$ such that
\[
    U(s)=\frac{s}{\mu} + \frac{\E[Y]}{\mu} + U_c(s),
\]
with
\begin{subequations}
    \begin{alignat}{2}
        \vert U_c(s)\vert\leq \frac{\E[Y]}{\mu}, \label{eq:U_c_UB}\\
        \lim_{s \to \infty} U_c(s)=0, \label{eq:U_c_Limit}
    \end{alignat}
\end{subequations}
\citep[cf. Propositions V.6.1 and V.6.2][]{Asmussen2003}.
Then we have
\begin{equation}
    \label{eq:IntDecompOnU}
    \int_{0}^{\xi^{(k)}} w_k(s) dU(s) - \frac{1}{\mu}\int_{0}^{\xi^{(k)}} w_k(s) ds =
    \int_{0}^{\xi^{(k)}} w_k(s) dU_c(s).
\end{equation}
We use integration by part to obtain
\begin{equation}
    \label{eq:Int_gap}
    \int_{0}^{\xi^{(k)}} w_k(s) dU_c(s) = w_k\left(\xi^{(k)} \right)U_c\left(\xi^{(k)} \right) - w_k(0)U_c(0)- \int_{0}^{\xi^{(k)}} U_c(s) dw_k(s).
\end{equation}
By the lemma's assumption as $k \to \infty$, $\xi^{(k)} \to \infty$. Thus, by \eqref{eq:U_c_Limit}, and boundedness of $w_k(\cdot)$
\begin{equation}
    \label{eq:Int_gap_1}
    \lim_{k\to \infty}w_k\left(\xi^{(k)} \right)U_c \left(\xi^{(k)} \right)=0
\end{equation}
Next, observe that, $\ellT \left(\dfrac{X}{k} \right)\leq \ellT(\LT)$, and
\[
    \lim_{k \to \infty}\ellT \left(\frac{X}{k} \right) U_c (0) = 0, 
\]
since $U_c (0)$ is bounded by \eqref{eq:U_c_UB}.
Therefore by Lebesgue's dominated convergence theorem
\begin{equation}
    \label{eq:Int_gap_2}
    \lim_{k\to \infty}w_k(0)U_c (0) = \lim_{k\to \infty} \E \left[ \ellT \left(\frac{X}{k} \right) \right] U_c (0)=0
\end{equation}
\citep[cf. Theorem 16.4][]{Billingsley1995}.

The remainder of the proof is devoted to showing that the final term on the RHS of \eqref{eq:Int_gap} vanishes as $k$ approaches infinity.
We organize the remainder of the proof into two steps.

\textbf{Step I:}
We construct a sequence $\left\{s_k \in \left[0,\LT \right] \right\}_{k \gg 0}$ such that as $k \to \infty$, we have $s_k \to 0$, and $ks_k \to \infty$:\\
By continuity of $\ellT$, for every $\epsilon_0>0$ there exists $s_0 \in (0,\LT]$ such that for all $s \in [0,s_0)$, we have $\ellT(s) < \epsilon_0$.
Additionally, by Assumption \ref{assum:ell0}, for every $\epsilon_r>0$ there exists $s_r \in (0,\LT]$ such that for all $s \in [0,s_r)$, we have $\ellT(s)/s < \epsilon_r/2$.
Choose $k>0$ such that $1/k< \min\{\epsilon_0,\epsilon_r/2\}$.
Then there exist $s_k \in \left(0,\dfrac{\min\{s_0,s_r\}}{2} \right)$ such that $\ellT(2s_k)= 1/k$ and $ks_k=s_k/\ellT(2s_k) > 1/\epsilon_r$.

This construction implies, also, that $\xi^{(k)}/ks_k \to \infty$ as $k \to \infty$, because by the lemma's assumption $ \lim_{k\to \infty}\xi^{(k)}/k>0$.
Hence, we restrict attention to sufficiently large $k$ with $\xi^{(k)} \geq ks_k$.

\textbf{Step II:}
We partition the integration domain of the RHS of \eqref{eq:Int_gap} into $[0,ks_k]$ and $(ks_k, \xi^{(k)}]$.
First, for any $s \in [0,ks_k)$ we find an upper bound for $w_k(s)$ as follows
\begin{alignat*}{2}
    \label{eq:EellTBounds}
    w_k(s) =& \int_0^\infty \ellT \left(\frac{x+s}{k} \right) dF_X(x) = \int_0^{ks_k} \ellT \left(\frac{x+s}{k} \right) dF_X(x) + \int_{ks_k}^\infty \ellT \left(\frac{x+s}{k} \right) dF_X(x) \nonumber\\
    \leq & \ellT(2s_k)F_X(ks_k)+\ellT \left(\LT \right) \left(1-F_X(ks_k) \right)
    \leq \frac{1}{k} + \ellT \left(\LT \right) \left(1- F_X \left(ks_k \right) \right).
\end{alignat*}
This together with \eqref{eq:U_c_UB} leads us to
\begin{equation}
    \label{eq:Step_aII_1}
    \Bigg \vert\int_0^{ks_k}U_c(s) dw_k(s) \Bigg \vert \leq\kappa_1(k),
\end{equation}
where
\begin{equation}
    \label{eq:kappa1}
    \kappa_1(k):= \frac{ \E [Y ]}{\mu}\left( \frac{1}{k} + \ellT \left(\LT \right) \left(1- F_X \left(ks_k \right) \right) \right).
\end{equation}

For the interval $\left(ks_k, \xi^{(k)}\right]$ we have
\begin{equation}
    \label{eq:Step_aII_2}
    \Bigg \vert\int_{ks_k}^{\xi^{(k)}} U_c(s) dw_k(s) \Bigg \vert \leq\kappa_2(k),
\end{equation}
where
\begin{equation}
    \label{eq:kappa2}
    \kappa_2(k):= \Bigg \vert \max_{s \in \left(ks_k,  \xi^{(k)}\right]}  U_c(s)  \Bigg\vert  \ellT \left(\LT \right).
\end{equation}

Now by \eqref{eq:Step_aII_1} through \eqref{eq:kappa2}
\begin{equation}
     \label{eq:intUcdwn_Bound}
     \Bigg \vert\int_0^{\xi^{(k)}} U_c(s) dw_k(s) \Bigg \vert \leq \Bigg \vert\int_0^{ks_k}U_c(s) dw_k(s) \Bigg \vert+\Bigg \vert\int_{ks_k}^{\xi^{(k)}} U_c(s) dw_k(s) \Bigg \vert  \leq \kappa_1(k)+ \kappa_2(k).
\end{equation}
By construction in Part I along with \eqref{eq:U_c_Limit},
\begin{equation}
    \label{eq:Int_gap_3}
    \lim_{k\to \infty}\kappa_1(k)= \lim_{k\to \infty}\kappa_2(k)=0.
\end{equation}
Equations \eqref{eq:IntDecompOnU}, \eqref{eq:Int_gap_1}, \eqref{eq:Int_gap_2}, and \eqref{eq:Int_gap_3} give
\[
    \lim_{k\to \infty} \left(\int_{0}^{\xi^{(k)}} w_k(s) dU(s) - \frac{1}{\mu}\int_{0}^{\xi^{(k)}} w_k(s) ds \right)=0,
\]
which completes the proof for Part \ref{lemma:elementwise_convergence_per_theta} of the lemma.

Proof of Part \ref{lemma:elementwise_convergence_E_Theta}:
By \eqref{eq:StateDepTotalCost_n}, \eqref{eq:IntDecompOnU}, and \eqref{eq:Int_gap} it is sufficient to show that the following limit exists and equals zero:
\[
    \lim_{k \to \infty}
    \int_\Theta\left( \Big\vert w_{k,\yb}\left(\xi^{(k)} (\yb) \right)U_c\left(\xi^{(k)}(\yb),\yb \right) \Big\vert + \Big\vert w_{k,\yb}(0,\yb)U_c(0,\yb ) \Big\vert+ \Big\vert \kappa_1(k,\yb)+\kappa_2(k,\yb) \Big\vert\right) dF_\theta(\yb),
\]
where $w_{k,\theta}(s)$, $U_c(\cdot,\theta)$, $\kappa_1(k, \theta)$, and $\kappa_2(k,\theta)$, are defined as in the proof of Part \ref{lemma:elementwise_convergence_per_theta}, except here they are parameterized by $\theta \in \Theta$.
Define $K: \Theta \to \R_+$ by
\[
    K(\theta):= \frac{\E\left[ Y(\theta) \right]}{\mu(\theta)}=\frac{\sigma^2(\theta)+\mu^2(\theta)}{2 \mu^2(\theta)}.
\]
For $k \gg 0$ and $\theta \in \Theta$:
\begin{subequations}
\begin{gather}
    \Big\vert w_{k,\theta}\left(0 \right)U_c\left(0,\theta \right) \Big\vert \leq \ellT \left(\LT \right) K(\theta),\label{eq:convUBS_1}\\
    \Big\vert w_{k,\theta}\left(\xi^{(k)}(\theta) \right)U_c\left(\xi^{(k)}(\theta),\theta \right) \Big\vert \leq \ellT \left(\LT \right) K(\theta),\label{eq:convUBS_2}\\
    \Big \vert \kappa_1(k,\theta)+\kappa(n,\theta) \Big \vert \leq 2 \ellT \left(\LT \right) K(\theta).\label{eq:convUBS_3}
\end{gather}
\end{subequations}
By Assumption \ref{assum:converge}, $K$ is integrable, i.e. $\displaystyle\int_\Theta K(\yb) dF_\theta(\yb)$ is finite.
Hence, the right-hand sides of \eqref{eq:convUBS_1} through \eqref{eq:convUBS_3} are integrable.
Moreover, by the arguments in the proof of Part \ref{lemma:elementwise_convergence_per_theta}, the left-hand sides of \eqref{eq:convUBS_1} through \eqref{eq:convUBS_3} converge to zero almost everywhere as $n$ grows to infinity.
Thus by Lebesgue's dominated convergence theorem we have
\begin{multline*}
    \lim_{k \to \infty}
    \int_\Theta\left( \Big\vert w_{k,\yb}\left(\xi^{(k)}(\yb) \right)U_c\left(\xi^{(k)}(\yb),\yb \right) \Big\vert + \Big\vert w_{k,\yb}(0,\yb)U_c(0,\yb ) \Big\vert+ \Big\vert \kappa_1(k,\yb)+\kappa_2(k,\yb) \Big\vert\right) dF_\theta(\yb)=\\
    \int_\Theta \lim_{k \to \infty}\left( \Big\vert w_{k,\yb}\left(\xi^{(k)}(\yb) \right)U_c\left(\xi^{(k)}(\yb),\yb \right) \Big\vert + \Big\vert w_{k,\yb}(0,\yb)U_c(0,\yb ) \Big\vert+ \Big\vert \kappa_1(k,\yb)+\kappa_2(k,\yb) \Big\vert\right) dF_\theta(\yb)=0.
\end{multline*}
The last equality follows from the arguments in the proof of Part \ref{lemma:elementwise_convergence_per_theta}.

Proof of Part \ref{lemma:elementwise_convergence_Prob}:
Since any probability distribution function is bounded, the result follows directly from the bounded convergence theorem \citep[cf. Theorem 16.5.][]{Billingsley1995} and Lemma \ref{lemma:Y_1 asymptotics}.

Proof of Part \ref{lemma:elementwise_convergence_CycleLength}:
Note that for any $k>0$ and $\theta \in \Theta$, we have $\xi^{(k)}\in [0,k\LT]$. Thus, by Lorden's inequality \citep{Lorden1970} we have for $k>1$ and $\theta \in \Theta$
\begin{equation}
    \label{eq:CLUB}
    \Bigg\vert\frac{1}{\mu(\theta)} \left(\frac{\xi^{(k)}(\theta)}{k} + \frac{\E\left[\YT_1 \left(\xi^{(k)}(\theta_1) \right) \Big \vert \theta_1=\theta \right]}{k}  \right) \Bigg\vert \leq \frac{1}{\mu(\theta)} \left( \LT+ 2\E[Y(\theta)] \right).
\end{equation}
By Assumption \ref{assum:converge}, the RHS of \eqref{eq:CLUB} is integrable.
Therefore, applying Lebesgue’s dominated convergence theorem together with Lemma \ref{lemma:Y_1 asymptotics} provides the result.
\Halmos

\endproof

\section{Proof of Lemma \ref{lemma:gasymptotic}}
Observe that for $k>0$ and $\xi^{(k)} \in \Xi^{(k)}$
\begin{equation}
    \label{eq:g_kFinitePositive}
    0<\frac{\cT_p}{\LT \displaystyle\int_\Theta\frac{1}{\mu(\yb)}dF_\theta(\yb) + 2\displaystyle\int_\Theta\frac{\E \left[\YT(\yb) \right]}{\mu(\yb)}dF_\theta(\yb) } \leq g_k\left(\xi^{(k)}\right) \leq \frac{\cT_p + \cT_f + \ellT(\LT)}{\LT \displaystyle\int_\Theta\frac{1}{\mu(\yb)}dF_\theta(\yb) }<\infty
\end{equation}
The lower bound on $g_k$ in \eqref{eq:g_kFinitePositive} results from minimizing the numerators and maximizing the denominators in \eqref{eq:g_k} and applying Lorden's inequality \citep{Lorden1970}.
The upper bound follows from the cost-rate of replacing at failure.
The upper bound is finite and the lower bound is strictly positive under Assumptions \ref{assum:converge} and \ref{assum:MeanPositivity}, respectively.
It follows from \eqref{eq:g_kFinitePositive} that $\lim_{k\to \infty} g_k \left(\xi^{(k)}\right)$ exists and is positive for any sequence $\left\{\xi^{(k)} \right\}_{k>0}$.
By Equation \eqref{eq:g_k} we have the following.
\begin{alignat}{2} 
    \label{eq:gnlimDef}
    g_k \left(\xi^{(k)} \right)= &  \frac{\cT_p + \displaystyle\int_\Theta \left(\cT_f \P \left\{\YT_1 \left(\xi^{(k)}(\theta_1) \right) > \phi_k \left(\xi^{(k)}(\theta_1) \right) \Big \vert \theta_1 = \yb \right\} + \frac{1}{k}\E \left[\sum_{\tau=1}^{T \left(\xi^{(k)}(\theta_1),\theta_1 \right)} \ell(S_{1,\tau},k) \Bigg\vert \theta_1=\yb \right]  \right) dF_\theta(\yb)}{\displaystyle \int_\Theta \frac{1}{\mu(\yb)} \left(\frac{\xi^{(k)}(\yb)}{k} + \frac{\E \left[\YT_1(\xi^{(k)}(\theta_1)) \Big \vert \theta_1 = \yb \right]}{k} \right) dF_\theta(\yb)}
\end{alignat}
The equality in \eqref{eq:gnlimDef} follows from dividing both the numerator and denominator by $k$.
By Lemma \ref{lemma:elementwise_convergence} and the growth rate of $\xi^{(k)}$ from Lemma's conditions, the limit of the numerator and denominator of are finite and positive as $k$ approaches infinity.
Therefore, $\lim_{k\to \infty} g_k \left(\xi^{(k)} \right)$ equals the ratio of the limits of the numerator and denominators in the RHS of \ref{eq:gnlimDef}.
Hence
\begin{alignat}{2} 
    \label{eq:gnlim}
    % \lim_{k\to \infty} g_k \left(\xi^{(k)} \right)=& 
    % \frac{\lim_{k\to \infty}\left(c_p(k) + \displaystyle\int_\Theta \left(c_f(k) \P \left\{\YT_1 \left(\xi^{(k)}(\theta_1) \right) > \phi_k \left(\xi^{(k)}(\theta_1) \right) \Big \vert \theta_1 = \yb \right\} + \E \left[\sum_{\tau=1}^{T \left(\xi^{(k)}(\theta_1),\theta_1 \right)} \ell(S_{1,\tau},k) \Bigg\vert \theta_1=\yb \right]  \right) dF_\theta(\yb) \right)}{\lim_{k\to \infty}\left(\displaystyle \int_\Theta \frac{1}{\mu(\yb)} \left(\frac{\xi^{(k)}(\yb)}{k} + \frac{\E \left[\YT_1(\xi^{(k)}(\theta_1)) \Big \vert \theta_1 = \yb \right]}{k} \right) dF_\theta(\yb) \right)} \nonumber\\
    \lim_{k\to \infty} g_k \left(\xi^{(k)} \right)=&
    \frac{\lim_{k\to \infty}\left(\cT_p + \displaystyle\int_\Theta \left(\cT_f \P \left\{Y(\yb) > \phi_k \left(\xi^{(k)}(\yb) \right) \right\} + \frac{1}{k} r_k \left(\xi^{(k)} ,\yb \right) \right) dF_\theta(\yb) \right)}{\lim_{k\to \infty}\left(\displaystyle \int_\Theta \frac{1}{\mu(\yb)} \left(\frac{\xi^{(k)}(\yb)}{k} + \frac{\E \left[Y(\yb) \right]}{k} \right) dF_\theta(\yb) \right)} \nonumber\\
    =&
    \lim_{k\to \infty}\frac{c_p(k) + \displaystyle\int_\Theta \left(c_f(k) \P \left\{Y(\yb) > \phi_k \left(\xi^{(k)}(\yb) \right) \right\} +  r_k \left(\xi^{(k)} ,\yb \right) \right) dF_\theta(\yb) }{\displaystyle \int_\Theta \frac{1}{\mu(\yb)} \left(\xi^{(k)}(\yb) + \E \left[Y(\yb) \right] \right) dF_\theta(\yb) } \nonumber\\
    =& \lim_{k\to \infty} \gamma_k \left(\xi^{(k)} \right)
\end{alignat}
The first equality in \eqref{eq:gnlim} follows from Lemma  \ref{lemma:elementwise_convergence}.
\Halmos

\section{Proof of Lemma \ref{lemma:threshold_order}}
We present the proof for Part \ref{lemma:threshold_order_xi}; the argument for Part \ref{lemma:threshold_order_xiT} is analogous.

Observe that $0\leq \liminf_{k \to \infty} \xio_k/k \leq \LT$.
Therefore, it is sufficient to show that $\liminf_{k \to \infty} \xio_k/k \not= 0$.
In the remainder of the argument, we fix $\theta \in \Theta$,
and suppress notation involving $\theta$ or conditioning on it, as well as the index for the component number.
For $k>0$ define
\begin{equation*}
    \hT_k(\lambda, \xi) := \cT_f \P \left\{\YT(\xi^{(k)}) > \phi_k \left(\xi^{(k)} \right) \right\} + \frac{1}{k}\E \left[\sum_{\tau=1}^{T(\xi^{(k)})} \ellT \left(\frac{S_{\tau}}{k} \right) \right]
    -\frac{\lambda}{\mu} \left(\frac{\xi^{(k)}}{k} + \frac{\E \left[\YT(\xi^{(k)}) \right]}{k}  \right).
\end{equation*}
By \eqref{eq:StateDepTotalCost} we have
\begin{multline}
    \label{eq:h_k_Def}
    \hT_k(\lambda, \xi) := \cT_f \P \left\{\YT(\xi^{(k)}) > k \LT - \xi^{(k)} \right\} \\
    + \frac{1}{k}\E\left[\ellT\left(\frac{X}{k}\right) \right] 
    + \frac{1}{k}\int_0^{\xi^{(k)}}\E\left[\ellT\left(\frac{X+s}{k}\right)  \right]dU(s)
    -\frac{\lambda}{\mu} \left(\frac{\xi^{(k)}}{k} + \frac{\E \left[\YT(\xi^{(k)}) \right]}{k}  \right).
\end{multline}
By Theorem \ref{thm:lambda*}, $\xio_n$ minimizes $\hT_k(g^*_k, \xi)$. In order to reach a contradiction, assume the contrary of the lemma, i.e., $\lim\inf_{k \to \infty} \xio_n/k = 0$.
This implies $\lim_{k \to \infty} \hT_k(g^*_k, \xio_n)=0$. The convergence of $\dfrac{1}{k}\displaystyle\int_0^{\xi^{(k)}}\E\left[\ellT\left(\frac{X+s}{k}\right)  \right]dU(s) \to \dfrac{1}{k \mu}\displaystyle\int_0^{\xi^{(k)}}\E\left[\ellT\left(\frac{X+s}{k }\right)  \right]ds$ follows from the argument in the proof of Lemma \ref{lemma:elementwise_convergence}\ref{lemma:elementwise_convergence_per_theta}.
Now consider the threshold function $\xi^{(k,c)}$ that satisfies
\[
    0< \lim_{k \to \infty}\E \left[\ellT\left(\frac{X+\xi^{(k,c)}}{k} \right) \right]<\lim_{k \to \infty} \frac{g^*_k}{\mu}.
\]
In this case $0<\lim_{k \to \infty} \xi^{(k,c)}/k\leq \LT$, as $\LT(0)=0$.
Then,
\begin{alignat*}{2}
    \lim_{k \to \infty} \hT_k(g^*_k, \xi^{(k,c)})=& \lim_{k \to \infty} \frac{1}{k \mu}\int_0^{\xi^{(k,c)}} \E \left[\ellT\left(\frac{X+s}{k} \right) \right] ds - \lim_{k \to \infty}\frac{g^*_k}{\mu}\lim_{k \to \infty} \frac{\xi^{(k,c)}}{k}\\
    \leq & \lim_{k \to \infty} \frac{\xi^{(k,c)}}{k \mu} \left( \lim_{k \to \infty}\E \left[\ellT\left(\frac{X+\xi^{(k,c)}}{k} \right)\right] - \lim_{k \to \infty}g^*_k\right)<0,
\end{alignat*}
which contradicts the assumption and therefore completes the proof.
\Halmos

\section{Proof of Theorem \ref{thm:gamma*}}
The proofs of Parts \ref{thm:gamma*_FixedPoint} and \ref{thm:gamma*_algorithm} proceed analogously to the argument in Appendix \ref{section:Proof_thm:lambda*}, with the functions $g$, $\hT$, and $\xiT_\lambda$ replaced by $\gamma$, $h$, and $\xi_\lambda$, respectively.

The monotonicity of $Dh$ in $x$ follows from the monotonicity of $F_X(x \mid \cdot)$  and $\ell(x)$.
If $F_X$ is continuous, the last part of Theorem \ref{thm:gamma*} follows from the first-order condition.
In case $F_X$ is $d-$lattice, note that
\[
    Dh(\lambda, x, \theta) = h(\lambda, d \Big\lfloor \frac{x}{d} \Big\rfloor, \theta) - h(\lambda, d \Big\lfloor \frac{x}{d}-1 \Big\rfloor, \theta).
\]
Thus, the minimum of $h$ can be achieved at the roots of $Dh(\lambda, x, \theta)$, or, when no root exists, at one of the boundary points. \Halmos

\section{Proof of Theorem \ref{thm:costconvergancenolearning}}
By Lemma \ref{lemma:gasymptotic} and \ref{lemma:threshold_order}, for any sufficiently small $\epsilon > 0$, there exists $k_{l_1}, k_{l_2} > 0$, such that,
\[
    -\frac{\epsilon}{2} < g_k \left(\xio_k \right) - \gamma_k \left(\xio_k \right) < \frac{\epsilon}{2}, \qquad \mbox{with } \quad \forall k > k_{l_1},
\]
and
\begin{equation*}
    % \label{ineq:ggh}
    -\frac{\epsilon}{2} < \gamma_k \left(\xia_k \right) - g_k \left(\xia_k \right) < \frac{\epsilon}{2}, \qquad \mbox{with} \quad \forall k > k_{l_2}.
\end{equation*}
Letting $k_l = \max (k_{l_1}, k_{l_2})$, it follows that,
\[
    g_k\left(\xia_k \right) - g_k(\xi_k^*) - \epsilon < \gamma_k\left(\xia_k \right) - \gamma_k(\xi_k^*),
    \qquad k > k_{l}.
\]
We notice that $g_k\left(\xia_k \right) - g_k(\xi_k^*) \geq 0$ and $\gamma_k\left(\xia_k \right) - \gamma_k(\xi_k^*) \leq 0$, by optimality.
This implies that,
\[
    0 \leq g_k \left(\xia_k \right) - g_k(\xi_k^*) < \epsilon, \qquad k > k_{l},
\]
which gives us Theorem \ref{thm:costconvergancenolearning}\ref{thm:optimalitygap}.
Part \ref{thm:predictiongap} of Theorem \ref{thm:costconvergancenolearning} follows directly from Lemmas \ref{lemma:gasymptotic} and \ref{lemma:threshold_order}.
\Halmos

\section{Proof of Lemma \ref{lemma:PointEstimateConsistency}}
\label{section:PointEstimateConsistencyproof}
% We establish our results under the assumption of weak consistency of the estimators, which is slightly more challenging than the case of strong consistency.
% The proof for strong convergence follows along a similar approach.
% The following argument addresses the operating component with the true parameter $\theta$.
% We omit $\theta$ for notational simplicity when clear from context.
% In particular, $\Tcal_k(\theta)$ is substituted with $\Tcal_k$, 

\noindent
\textit{Part \ref{lemma:PointEstimateConsistency_theta}.}
By the weak consistency of $\thetaH_\tau$, we have for every $\epsilon, \eta>0$, there exists $\tau_{\epsilon,\eta}^l>0$ such that for all $\tau>\tau_{\epsilon,\eta}^l$
\begin{equation}
    \label{eq:weakConsistencyDef}
    \P \left\{ \Big \Vert\thetaH(\Ib_\tau) - \theta \Big\Vert>\epsilon \right\}<\eta.
\end{equation}
Moreover, by \eqref{eq:Tcal_lim} we have for all $\Tcal^l>0$ that there exists a $k_{\Tcal^l}$ such that $\Tcal_k(\theta)>\Tcal^l$ for all $k>k_{\Tcal^l}$ almost surely.
%, i.e.
%\begin{equation}
%    \label{eq:Tcal_lim_Def}
%    \P \left\{\exists k_{\Tcal^l} \mbox{ s.t. }    \Tcal_k(\theta)>\Tcal^l\quad \forall k>k_{\Tcal^l} \right\}=1,
%\end{equation}
%Note that $k_{\Tcal^l}$ is a random variable.
% for every $\Tcal^l>0$ there exists $k_{\Tcal^l}$ such that for all $k>k_{\Tcal^l}$
% \begin{equation}
%     \label{eq:Tcal_lim_Def}
%     \P \left\{ \Tcal_k(\theta) > \Tcal^l \right\}=1.
% \end{equation}
Now, choose $\Tcal^l>\tau_{\epsilon,\eta}^l$. Then we obtain that for all $k>k_{\Tcal^l}$
\[
    \P \left\{\Big \Vert\thetaH \left(\Ib_{\Tcal_k\left( \theta \right)} \right) - \theta \Big\Vert>\epsilon \Big \vert \Tcal_k \left( \theta \right)\right\} < \eta, \qquad \mbox{almost surely.}
    % \P \left \{\P \left\{\Big \Vert\thetaH_{\Tcal_k\left( \theta \right)} - \theta \Big\Vert>\epsilon \Big \vert \Tcal_k \left( \theta \right)\right\} < \eta \right\}=1.
\]
Using the law of total probability (condition on $\Tcal_k (\theta)$) completes the proof for Part \ref{lemma:PointEstimateConsistency_theta}.

\noindent
\textit{Part \ref{lemma:PointEstimateConsistency_xi}.}
By Assumption \ref{assum:FlatRigion}, we have that for all $\lambda \in \R$, $\theta \in \Theta$, $\max \mathcal{D}_{\lambda,\theta}-\min \mathcal{D}_{\lambda,\theta} \leq \xf$.
Consequently, for all $\lambda \in \R$, $\theta \in \Theta$,
\begin{equation}
    \label{eq:limgap}
    \lim \sup_{\theta^\prime \to \theta} \xi_\lambda (\theta^\prime) - \lim \inf_{\theta^\prime \to \theta} \xi_\lambda (\theta^\prime) \leq \xf.
\end{equation}
For $\epsilon>0$, define $k_{\epsilon}=2 \xf/\epsilon$.
By \eqref{eq:limgap}, for all $\epsilon>0$, there exists a neighborhood of $\theta$, $\Theta_{\theta, \epsilon}:= \left \{\theta^\prime \in \Theta: \Vert \theta^\prime - \theta \Vert < \epsilon/2 \right\}$ such that for all $k> k_\epsilon$ and $\theta^\prime \in \Theta_{\theta, \epsilon}$
\[
    \frac{1}{k}\Big \vert \xia_k (\theta^\prime) - \xia_k (\theta) \Big \vert \leq \frac{\xf}{k}+ \frac{\epsilon}{2k}<\frac{\xf}{k_\epsilon}+\frac{\epsilon}{2}=\epsilon.
\]
Choose $0 < \eta < 1$. Part~\ref{lemma:PointEstimateConsistency_theta} implies that there exists a $k_{\theta,\eta}$ such that $\P\left\{\thetaH \left(\Ib_{\Tcal_k\left( \theta \right)} \right) \in \Theta_{\theta,\epsilon}\right\}>1-\eta$ for each $k>k_{\theta,\eta}$.
Therefore for all $k> \max\{k_\epsilon, k_{\theta,\eta}\}$ we have
\begin{alignat*}{2}
    \label{eq:EstimateConsistency_xi}
    \P \left \{\frac{1}{k}\Big \vert \xipklast(\theta) - \xia_k (\theta) \Big \vert <\epsilon \right \} =  & \P \left \{\frac{1}{k}\Big \vert \E\left[ \xia_k \left(\thetaH \left(\Ib_{\Tcal_k\left( \theta \right)} \right) \right) \right]  - \xia_k (\theta) \Big \vert <\epsilon \right \} \nonumber\\
    \geq  & \P \left \{\frac{1}{k}\Big \vert \E\left[ \xia_k \left(\thetaH \left(\Ib_{\Tcal_k\left( \theta \right)} \right) \right) \right] - \xia_k (\theta) \Big \vert <\epsilon \Big \vert \thetaH \left(\Ib_{\Tcal_k\left( \theta \right)} \right)\in \Theta_{\theta,\epsilon} \right \} \P \left \{ \thetaH \left(\Ib_{\Tcal_k\left( \theta \right)} \right)  \in \Theta_{\theta,\epsilon}\right\} \nonumber\\
    = & \P \left \{ \thetaH \left(\Ib_{\Tcal_k\left( \theta \right)} \right)  \in \Theta_{\theta,\epsilon}\right\} > 1-\eta.
\end{alignat*}
% \begin{alignat}{2}
%     \label{eq:EstimateConsistency_xi}
%     \P \left \{\frac{1}{k}\Big \vert \xiH_{\Tcal_k\left( \theta \right),k} (\theta) - \xiT^*_k (\theta) \Big \vert <\epsilon \right \} =  & \P \left \{\frac{1}{k}\Big \vert \E\left[ \xiT^*_k \left(\thetaH \left(\Ib_{\Tcal_k\left( \theta \right)} \right) \right) \Big \vert \I_{t_0+\Tcal_k(\theta)}\right]  - \xiT^*_k (\theta) \Big \vert <\epsilon \right \} \nonumber\\
%     \geq  & \P \left \{\frac{1}{k}\Big \vert \E\left[ \xiT^*_k \left(\thetaH \left(\Ib_{\Tcal_k\left( \theta \right)} \right) \right) \Big \vert \I_{t_0+\Tcal_k(\theta)}\right] - \xiT^*_k (\theta) \Big \vert <\epsilon \Big \vert \thetaH \left(\Ib_{\Tcal_k\left( \theta \right)} \right)\in \Theta_{\theta,\epsilon} \right \} \P \left \{ \thetaH \left(\Ib_{\Tcal_k\left( \theta \right)} \right)  \in \Theta_{\theta,\epsilon}\right\} \nonumber\\
%     = & \P \left \{ \thetaH \left(\Ib_{\Tcal_k\left( \theta \right)} \right)  \in \Theta_{\theta,\epsilon}\right\} > 1-\eta.
% \end{alignat}
which establishes Part \ref{lemma:PointEstimateConsistency_xi}.

\noindent
\textit{Part \ref{lemma:PointEstimateConsistency_Exi}}.
Observe that for any $k>0$ and $\theta \in \Theta$, $\dfrac{1}{k}\Big \vert \xipklast(\theta) - \xia_k (\theta) \Big \vert$ is upper bounded by $\LT$, and consequently, uniformly integrable.
Thus, by Part \ref{lemma:PointEstimateConsistency_xi} it converges to zero in expectation, i.e., for all $\theta \in \Theta$,
\[
    \lim_{k \to \infty} \E \left [\frac{1}{k\mu(\theta)}\Big \vert \xipklast(\theta) - \xia_k (\theta) \Big \vert \right]=0.
\]
We can extend this relation to
\[
    \lim_{k \to \infty} \E \left [\frac{1}{k\mu(\theta)}\Big \vert \xipklast(\theta) +\E \left[ Y \left( \xipklast(\theta)\right)\right] - \left(\xia_k (\theta) +\E \left[ Y \left( \xia(\theta)\right)\right]\right) \Big \vert \right]=0.
\]
since $\E[Y (\cdot)]$ is bounded by Lorden's inequality.
This implies that,
\begin{alignat}{2}
    \label{eq:EstimateConsistency_Exi}
    &\lim_{k \to \infty}\E \left[ \displaystyle \int_\Theta \frac{1}{k\mu(\yb)}\Big \vert \xipklast(\yb) +\E \left[ Y \left( \xipklast(\yb)\right)\right] - \left(\xia_k (\yb) +\E \left[ Y \left( \xia(\yb)\right)\right]\right) \Big \vert dF_\theta(\yb) \right] = \nonumber\\ 
    &\lim_{k \to \infty} \displaystyle \int_\Theta \E \left[  \frac{1}{k\mu(\yb)}\Big \vert \xipklast(\yb) +\E \left[ Y \left( \xipklast(\yb)\right)\right] - \left(\xia_k (\yb) +\E \left[ Y \left( \xia(\yb)\right)\right]\right) \Big \vert  \right] dF_\theta(\yb) = 0.
\end{alignat}
The first equality in \eqref{eq:EstimateConsistency_Exi} follows from Fubini's theorem, which allows exchanging the order of integration and expectation, and the final inequality by the dominated convergence theorem.
Next, for every $\epsilon>0$ we have 
\begin{alignat*}{2}
    &\P \left\{ \Bigg \vert \displaystyle \int_\Theta \frac{1}{k\mu(\yb)}\xipklast(\yb) +\E \left[ Y \left( \xipklast(\yb)\right)\right] - \left(\xia_k (\yb) +\E \left[ Y \left( \xia(\yb)\right)\right]\right)  dF_\theta(\yb) \Bigg \vert > \epsilon \right\} \leq \\
    &\P \left\{ \displaystyle \int_\Theta \frac{1}{k\mu(\yb)}\Big \vert \xipklast(\yb) +\E \left[ Y \left( \xipklast(\yb)\right)\right] - \left(\xia_k (\yb) +\E \left[ Y \left( \xia(\yb)\right)\right]\right) \Big \vert dF_\theta(\yb) > \epsilon \right\} \leq \\
    &\frac{\E \left[ \displaystyle \int_\Theta \frac{1}{k\mu(\yb)}\Big \vert \xipklast(\yb) +\E \left[ Y \left( \xipklast(\yb)\right)\right] - \left(\xia_k (\yb) +\E \left[ Y \left( \xia(\yb)\right)\right]\right) \Big \vert dF_\theta(\yb) \right]}{\epsilon}.
\end{alignat*}
The first inequality follows from the integral inequality and the second inequality from Markov's inequality.
Combining this with \eqref{eq:EstimateConsistency_Exi} completes the proof for Part \ref{lemma:PointEstimateConsistency_Exi}.

\noindent
\textit{Proof of Part \ref{lemma:PointEstimateConsistency_ell}}.
We prove the result for continuous $F_X$.
The proof for $d-$lattice $F_X$ is analogous, with integration replaced by summation.
By continuity of $F_X$ in $\theta \in \Theta$ and Part \ref{lemma:PointEstimateConsistency_theta}, for any $\epsilon\ ,x>0$
\begin{equation}
    \label{eq:Eell_continuity}
    \lim_{k \to \infty}\P \left \{ \Big \vert \E \left[\ell(X_{1,1}+x,k) \mid \theta_1= \thetaH \left(\Ib_{\Tcal_k\left( \theta \right)} \right)\right] - \E \left[\ell(X_{1,1}+x,k) \mid \theta_1= \theta\right] \Big \vert < \epsilon \right\}=1
\end{equation}
Now by the definition of $r_k$ as in \eqref{eq:rkDef}
\begin{alignat*}{2}
    \frac{1}{k} \Big \vert r_k\left(\xipklast(\theta), \theta \right)- r_k(\xia_k(\theta),\theta) \Big \vert \leq&
    \frac{1}{k\mu(\theta)} \displaystyle\int_0^{\xia_k(\theta)} \Big \vert \E \left[\ell(X_{1,1}+x,k) \mid \theta_1= \thetaH \left(\Ib_{\Tcal_k\left( \theta \right)} \right)\right] - \E \left[\ell(X_{1,1}+x,k) \mid \theta_1= \theta\right] \Big \vert dx +\\
    & \frac{1}{k\mu(\theta)} \Bigg \vert \displaystyle\int_{\xia_k(\theta)}^{\xipklast \left(\theta \right)} \max \left\{ \E \left[\ell(X_{1,1}+x,k) \mid \theta_1= \thetaH \left(\Ib_{\Tcal_k\left( \theta \right)} \right)\right] , \E \left[\ell(X_{1,1}+x,k) \mid \theta_1= \theta\right] \right\} dx \Bigg \vert \\
    \leq &\frac{1}{k\mu(\theta)} \displaystyle\int_0^{\xia_k(\theta)} \Big \vert \E \left[\ell(X_{1,1}+x,k) \mid \theta_1= \thetaH \left(\Ib_{\Tcal_k\left( \theta \right)} \right)\right] - \E \left[\ell(X_{1,1}+x,k) \mid \theta_1= \theta\right] \Big \vert dx +\\
    & \frac{1}{k\mu(\theta)}\ell(\LT) \Big \vert \xipklast \left(\theta \right) - \xia_k(\theta)  \Big \vert.
\end{alignat*}
Combining this result with Part \ref{lemma:PointEstimateConsistency_xi} and \eqref{eq:Eell_continuity} give us, for every $\epsilon>0$
\begin{equation}
    \label{eq:convergence_r}
    \lim_{k \to \infty} \P \left\{ \frac{1}{k} \Big \vert r_k\left(\xipklast(\theta), \theta \right)- r_k \left(\xia_k(\theta),\theta \right) \Big \vert <\epsilon \right\}=1.
\end{equation}
The rest of the proof follows from an analogous argument as Part \ref{lemma:PointEstimateConsistency_Exi} with the result in Part \ref{lemma:PointEstimateConsistency_xi} is replaced by \eqref{eq:convergence_r}, and by noticing that for all $k>0$ and $\theta \in \Theta$
\[
    \frac{1}{k} r_k \left(\xia_k(\theta),\theta \right) \leq  \ell(\LT) \left(\frac{1}{k}+\frac{\LT}{\mu(\theta)} \right) < \infty.
\]
\noindent
\textit{Proof of Part \ref{lemma:PointEstimateConsistency_prob}}.
Let $\xia_{k,0}$ denote the Oracle's asymptotically optimal threshold function when $\ell(x)=0$ for all $x \in \R_+$. Then for $k>0$,
\begin{equation}
    \label{eq:xiT_0}
    \xia_{k,0}(\theta)\geq \xia_k(\theta),
\end{equation}
since by assumption $\ell(x)\geq0$.
Furthermore, let $\left\{\xil_k \right\}_{k>0}$ be a sequence of threshold functions that satisfy $\displaystyle \lim_{k \to \infty} \frac{\xil_k(\theta)}{L(k)} =1$ and $\displaystyle\lim_{k \to \infty} \phi_k \left(\xil_k(\theta) \right)=\infty$, for all $\theta \in \Theta$, when $\ell(x)=0$.
This also implies $\displaystyle \lim_{k \to \infty} \frac{\phi_k \left(\xil_k(\theta) \right)}{k}=0$.
Hence
\begin{equation}
    \label{eq:LinearThresholds}
    \lim_{k \to \infty} \gamma_k \left(\xil_k \right)= \dfrac{\cT_p}{\LT \displaystyle \int_\Theta \dfrac{dF_\theta(\yb)}{\mu(\yb)}}.
\end{equation}
% This result follows from the following observation.
% For $\lambda \in \R$ and $\theta \in \Theta$, let $\xi_{\lambda,0}(\theta)$ be given by Theorem \ref{thm:gamma*}\ref{thm:gamma*_threshold} when $\ell(x)=0$ for all $x \in \R_+$.
% Evidently, $\xi_{\lambda,0}(\theta) \geq \xi_\lambda(\theta)$, since $\ell(x) \geq0$ for all $x$.
We notice, from \eqref{eq:gamma_k}, that for any function $\xi^{(k)} \in \Xi_k$ a natural lower bound for the limit of $\gamma_k(\xi^{(k)})$ as $k \to \infty$ is as follows
\begin{equation*}
    \lim_{k \to \infty} \gamma_k \left(\xi^{(k)} (\theta) \right) \geq \dfrac{\cT_p}{\LT \displaystyle \int_\Theta \dfrac{dF_\theta(\yb)}{\mu(\yb)}}.
\end{equation*}
Combining this result with \eqref{eq:xiT_0} and \eqref{eq:LinearThresholds} gives $\lim_{k \to \infty} \left(\xia_k(\theta)- \xil_k(\theta)\right) \leq 0$.
Consequently, for all $\theta \in \Theta$
\[
    \lim_{k \to \infty} \phi_k \left(\xia_k(\theta) \right)=\infty,
\]
which immediately implies that both sides of Part \ref{lemma:PointEstimateConsistency_prob} converge to zero as $k \to \infty$, almost surely. \Halmos

\section{Proof of Theorem \ref{thm:costconvergancewithlearning}}
% Similar to Appendix \ref{section:PointEstimateConsistencyproof}, we present the proof for a weakly consistent estimator.
We notice that Lemma \ref{lemma:PointEstimateConsistency}\ref{lemma:PointEstimateConsistency_Exi} can be extended to
\begin{multline}
    \label{eq: ECL_weak_convergence}
    \dfrac{1}{k} \displaystyle \int_\Theta \frac{1}{\mu(\yb)} \left(\xipklast(\yb)+ \E \left[ \YT_1 \left(\xipklast (\yb)\right) \Big \vert \theta_1 = \yb, \xipklast (\yb)   \right]\right)dF_\theta(\yb) \xrightarrow{p} \\
    \dfrac{1}{k} \displaystyle \int_\Theta \frac{1}{\mu(\yb)} \left(\xia_k \left(\yb \right) +\E \left[\YT_1 \left(\xia_k \left(\theta_1 \right)  \right) \Big \vert \theta_1 = \yb \right]\right) dF_\theta(\yb).
\end{multline}
by applying Lemma \ref{lemma:elementwise_convergence}\ref{lemma:elementwise_convergence_CycleLength} to both sides of Lemma \ref{lemma:PointEstimateConsistency}\ref{lemma:PointEstimateConsistency_Exi}.
Similarly, we use Parts \ref{lemma:elementwise_convergence_E_Theta} and \ref{lemma:elementwise_convergence_Prob} of Lemma \ref{lemma:elementwise_convergence} to extend Parts \ref{lemma:PointEstimateConsistency_Exi} and \ref{lemma:PointEstimateConsistency_prob} of Lemma \ref{lemma:PointEstimateConsistency}, respectively
\begin{multline}
    \label{eq: opCost_weak_convergence}
    \frac{1}{k}\displaystyle \int_\Theta\E \left[\displaystyle\sum_{\tau=1}^{T \left(\xipklast(\yb) ,\theta_1 \right)} \ell(S_{1,\tau},k) \Bigg\vert \theta_1=\yb, \xipklast(\yb) \right] dF_\theta(\yb) \xrightarrow{p} \\
    \frac{1}{k}\displaystyle \int_\Theta\E \left[\displaystyle\sum_{\tau=1}^{T \left(\xia(\theta_1) ,\theta_1 \right)} \ell(S_{1,\tau},k) \Bigg\vert \theta_1=\yb \right] dF_\theta(\yb),
\end{multline}
\begin{multline}
    \label{eq: failProb_weak_convergence}
    \displaystyle\int_\Theta \P \left\{\YT_1(\xipklast(\yb)) > \phi_k \left(\xipklast(\yb) \right) \Big \vert \theta_1=\yb, \xipklast(\yb) \right\} dF_\theta(\yb) \xrightarrow{p} \\
    \displaystyle\int_\Theta \P \left\{\YT_1(\xia(\theta_1)) > \phi_k \left(\xi^{(k)}(\theta_1) \right) \Big \vert \theta_1=\yb \right\} dF_\theta(\yb).
\end{multline}
Observe that the RHS of \eqref{eq: ECL_weak_convergence}, \eqref{eq: opCost_weak_convergence}, and \eqref{eq: failProb_weak_convergence} are uniformly integrable, which yields the stronger mean convergence results.
This observation together with Fubini's theorem give us
\begin{multline}
    \label{eq: ECL_L1_convergence}
    \dfrac{1}{k} \displaystyle \int_\Theta \frac{1}{\mu(\yb)} \left(\xipklast(\yb)+ \E \left[ \YT_1 \left(\xipklast (\yb)\right) \Big \vert \theta_1 = \yb   \right]\right)dF_\theta(\yb) \to \\
    \dfrac{1}{k} \displaystyle \int_\Theta \frac{1}{\mu(\yb)} \left(\xia_k \left(\yb \right) +\E \left[\YT_1 \left(\xia_k \left(\theta_1 \right)  \right) \Big \vert \theta_1 = \yb \right]\right) dF_\theta(\yb),
\end{multline}
\begin{equation}
    \label{eq: opCost_L1_convergence}
    \frac{1}{k} \displaystyle \int_\Theta\E \left[\displaystyle\sum_{\tau=1}^{T \left(\xipklast(\yb) ,\theta_1 \right)} \ell(S_{1,\tau},k) \Bigg\vert \theta_1=\yb \right] dF_\theta(\yb) \to
    \frac{1}{k}\displaystyle \int_\Theta\E \left[\displaystyle\sum_{\tau=1}^{T \left(\xia(\theta_1) ,\theta_1 \right)} \ell(S_{1,\tau},k) \Bigg\vert \theta_1=\yb \right] dF_\theta(\yb),
\end{equation}
\begin{multline}
    \label{eq: failProb_L1_convergence}
    \displaystyle\int_\Theta \P \left\{\YT_1(\xipklast(\yb)) > \phi_k \left(\xipklast(\yb) \right) \Big \vert \theta_1=\yb\right\} dF_\theta(\yb) \to \\
    \displaystyle\int_\Theta \P \left\{\YT_1(\xia(\theta_1)) > \phi_k \left(\xi^{(k)}(\theta_1) \right) \Big \vert \theta_1=\yb \right\} dF_\theta(\yb).
\end{multline}
Next, by \eqref{eq:GeneralSLLNDef} as $k \to \infty$
\begin{alignat}{2}
    \label{eq:MainConvergenceResult}
    g_k(\xip) = & \frac{c_p + \displaystyle\int_\Theta \left(c_f\P \left\{\YT_1(\xipklast(\yb)) > \phi_k \left(\xipklast(\yb) \right) \Big \vert \theta_1=\yb\right\} + \frac{1}{k} \left[\displaystyle\sum_{\tau=1}^{T \left(\xipklast(\yb) ,\theta_1 \right)} \ell(S_{1,\tau},k) \Bigg\vert \theta_1=\yb \right] \right)
    dF_\theta(\yb)}{\dfrac{1}{k} \displaystyle \int_\Theta \frac{1}{\mu(\yb)} \left(\xipklast(\yb)+ \E \left[ \YT_1 \left(\xipklast (\yb)\right) \mid \theta_1 = \yb   \right]\right)dF_\theta(\yb)} \nonumber\\
    \to & \frac{c_p + \displaystyle\int_\Theta \left( c_f \P \left\{\YT_1(\xia(\theta_1)) > \phi_k \left(\xi^{(k)}(\theta_1) \right) \Big \vert \theta_1=\yb \right\} + \frac{1}{k} \E \left[\displaystyle\sum_{\tau=1}^{T \left(\xia(\theta_1) ,\theta_1 \right)} \ell(S_{1,\tau},k) \Bigg\vert \theta_1=\yb \right] \right)
    dF_\theta(\yb)}{\dfrac{1}{k} \displaystyle \int_\Theta \frac{1}{\mu(\yb)} \left(\xia_k \left(\yb \right) +\E \left[\YT_1 \left(\xia_k \left(\theta_1 \right) \right) \Big \vert \theta_1 = \yb \right]\right) dF_\theta(\yb)}\nonumber\\
    = & g_k(\xia) \to g_k(\xio).
\end{alignat}
The first equality in \eqref{eq:MainConvergenceResult} follows from Kolmogorov's law of large numbers as per \eqref{eq:GeneralSLLNDef} together with the expression of $g_k$, i.e., Equation \eqref{eq:g_k}, the first convergence from Equations \eqref{eq: opCost_L1_convergence} through \eqref{eq: failProb_L1_convergence} and the last convergence from Theorem~\ref{thm:costconvergancenolearning}.
\Halmos

\end{APPENDICES}

\end{document}